\documentclass[a4paper,10pt,leqno,twoside]{article}

\usepackage[utf8]{inputenc}

\usepackage{amsmath}  
\usepackage{mathrsfs}  
\usepackage{amsthm}
\usepackage{amsfonts}
\usepackage{amssymb, latexsym}
\usepackage{verbatim}

\usepackage{kerkis}

\usepackage{graphicx}
\usepackage{color}
\usepackage{xcolor}
\usepackage{tikz}
\usetikzlibrary{calc,intersections,through,backgrounds}
\usetikzlibrary{decorations.pathreplacing}
\usetikzlibrary{shapes}

\usepackage{float} 
\usepackage{enumerate}
\usepackage[margin=1.5in]{geometry}

\usepackage[pdfauthor={P. Tsatsoulis, H. Weber}, pdftitle={Spectral Gap for the Stochastic Quantization 
Equation on the 2-dimensional Torus},pdftex]{hyperref}
\hypersetup{pdfborder = {0 0 0.5 [2 2]}}

\usepackage{fancyhdr}

\usepackage[toc,page,titletoc]{appendix}

\noappendicestocpagenum

\usepackage{tabularx}
\usepackage{array}
\newcolumntype{C}[1]{>{\centering\arraybackslash}m{#1}} 

\usepackage{macros}

\definecolor{darkred}{rgb}{0.8,0.0,0.0}
\definecolor{darkblue}{rgb}{0,0,0.6}
\definecolor{darkgreen}{rgb}{0,0.6,0}

\newcommand{\BB}{\mathcal{B}} 
\newcommand{\CC}{\mathcal{C}} 

\newcommand{\PP}{\mathbb{P}} 
\newcommand{\EE}{\mathbb{E}} 
\newcommand{\TT}{\mathbb{T}} 
\newcommand{\dd}{\,\mathrm{d}} 
\newcommand{\RR}{\mathbb{R}} 
\newcommand{\ZZ}{\mathbb{Z}} 
\newcommand{\NN}{\mathbb{N}} 
\newcommand{\HH}{\mathcal{H}} 
\newcommand{\DD}{\mathcal{D}} 
\newcommand{\A}{\mathcal{A}} 

\newcommand{\lng}{\langle} 
\newcommand{\rng}{\rangle} 
\newcommand{\supp}{\mathrm{supp}\,} 

\newcommand{\e}{\mathrm{e}} 
\newcommand{\ii}{\mathrm{i}} 

\newcommand{\al}{\alpha}
\newcommand{\ee}{\varepsilon} 
\newcommand{\bt}{\beta}

\newcommand{\kk}{\kappa}

\newcommand{\VVert}[1]{{\left\vert\kern-0.25ex \left\vert\kern-0.25ex \left \vert #1 
    \right\vert\kern-0.25ex \right\vert\kern-0.25ex \right\vert}} 

\theoremstyle{plain}
\newtheorem{theorem}{Theorem}[section]
\newtheorem{proposition}[theorem]{Proposition}
\newtheorem{corollary}[theorem]{Corollary}
\newtheorem{lemma}[theorem]{Lemma}

\theoremstyle{definition}
\newtheorem{definition}[theorem]{Definition}

\theoremstyle{remark}
\newtheorem{remark}[theorem]{Remark}

\numberwithin{equation}{section}

\title{\textbf{Spectral Gap for the Stochastic Quantization Equation on the 2-dimensional Torus}}
\author{Pavlos Tsatsoulis, Hendrik Weber}
\date{}

\fancypagestyle{plain}{

  \fancyhf{}
  \fancyfoot[L]{\footnotesize \today} 
  \fancyfoot[C]{\thepage}
}

\begin{document}

\maketitle

\begin{abstract}
We study the long time behavior of the stochastic quantization equation. Extending recent results by 
Mourrat and Weber \cite{MWe15} we first establish a strong non-linear dissipative bound that gives control of 
moments of solutions at all positive times independent of the initial datum. We then establish that 
solutions give rise to a Markov process whose transition semigroup satisfies the strong Feller property. 
Following arguments by Chouk and Friz \cite{ChF16} we also prove a support theorem 
for the laws of the solutions. Finally all of these results are combined to show that the transition semigroup 
satisfies the Doeblin criterion which implies exponential convergence to equilibrium. 
 
Along the way we give a simple direct proof of the Markov property of solutions and an independent argument for 
the existence of an invariant measure using the Krylov--Bogoliubov existence theorem. Our method makes no use of 
the reversibility of the dynamics or the explicit knowledge of the invariant measure and it is therefore in principle 
applicable to situations where these are not available, e.g. the vector-valued case.


\end{abstract}

\setcounter{tocdepth}{1} 
\tableofcontents 


\section{Introduction}

We consider the stochastic quantization equation on the $2$-dimensional torus $\TT^2$ given by 
\begin{equation}\label{eq:RSQEq_intro}
\left\{
\begin{array}{lcll}
\partial_t X & = &\Delta X - X - \sum_{k=0}^n a_k :X^k: + \xi, & \text{in } \mathbb{R}_{+}\times \mathbb{T}^2 \\
X(0,\cdot) & = & x, & \text{on } \mathbb{T}^2
\end{array}
\right.,
\end{equation}
where $n$ is odd, $a_n>0$, $\xi$ is a Gaussian space time white noise and $x$ is a distribution of suitably negative 
regularity. Here $:X^k:$ stands for the $k$-th Wick power of $X$ (see Section 
\ref{s:Preliminaries} for its definition). This  equation was first proposed by Parisi and Wu (see \cite{PW81})
as a natural reversible dynamics for the $\Phi^{n+1}_2$ measure which is given by
\begin{equation}\label{eq:inv_m}
 \nu(\dd X) \propto 
 \exp\left\{- 2 \int_{\TT^2} \sum_{k=0}^n \frac{a_k}{k+1} :X^{k+1}:(z) \dd z\right\}
 \mu(\dd X),
\end{equation}
where $\mu$ is the law of a massive Gaussian free field. 

The interpretation and construction of solutions for \eqref{eq:RSQEq_intro} remained a challenge for many years
with important contributions by Jona--Lasinio and Mitter in \cite{JLM85}, (solution of a modified 
equation via Girsanov's transformation) and Albeverio and R\"ockner in \cite{AR91} (construction 
of solutions using the theory of Dirichlet forms).
In \cite{dPD03} da Prato and Debussche proposed a simple transformation 
of \eqref{eq:RSQEq_intro} which allowed them to prove local in time existence of strong solutions for any initial datum $x$ 
of suitable (negative) regularity and non-explosion for $x$ in a set of measure one with respect to \eqref{eq:inv_m}.
Recently Mourrat and Weber \cite{MWe15} obtained global in time solutions on the  the full space for any initial datum of suitable 
regularity by following a similar strategy. In \cite{RZZ15} R\"ockner et al. identified these solutions with the solutions 
obtained via Dirichlet forms.  

The aim of this paper is to establish exponential convergence to equilibrium for solutions of \eqref{eq:RSQEq_intro}.  
Building on the analysis in \cite{MWe15} and using a simple comparison test for non-linear ordinary differential equations we 
establish a strong dissipative bound for the solutions. We then prove the strong Feller property for the Markov semigroup 
generated by the solution generalizing the method in \cite[Section 4.2]{HSV07}. 
Although for convenience we make (moderate) use of global in time existence which follows from the strong dissipative bounds derived before, 
this part of the analysis could  also be implemented 
using only local existence (see Remark \ref{rem:cemetery}); the linearized dynamics
of Galerkin aproximations are controlled 
by combining  a localization via stopping times and the small-time bounds obtained from the local existence theory. 
We furthermore establish a support
theorem in the spirit of \cite{ChF16}. Finally, we combine all of these ingredients to show that the associated
Markov semigroup satisfies the 
Doeblin criterion which implies exponential convergence to the unique invariant measure uniformly over the state space. 
 
All steps are implemented for general odd $n$ except for the support theorem which we only show
in the case $n=3$. The reason for this restriction is explained in Remark \ref{rem:general_supp_theorem}. We expect
however that a support theorem for \eqref{eq:RSQEq_intro} holds true for all odd $n$ and that such a result could be  combined with  the results of this
paper to generalize Theorem \ref{thm:conv_rate} to the case of an arbitrary odd $n$. 

Along the way we give 
independent proofs of the Markov property for the dynamics as well as existence of the invariant measure.
The Markov property was already established previously in \cite{RZZ15} based on the identification of the 
dynamics with the solutions constructed via Dirichlet form. The same paper \cite{RZZ15} also 
established that \eqref{eq:inv_m} is a reversible (and in particular invariant) measure for the dynamics. 
We stress that our approach completely circumvents the theory of Dirichlet forms and uses neither 
the symmetry of the process nor the explicit form of the invariant measure. We therefore expect that our methods could be applied 
in situations where the reversibility is absent and where there is no explicit representation of the 
invariant measure, for example in situations where $X$ is vector rather than scalar valued. 

Finally, we would like to mention two independent works on a similar subject - one \cite{RZZ16} 
published very recently and one \cite{HM16} about to appear. In \cite{RZZ16} the authors establish that 
\eqref{eq:inv_m} is the unique invariant measure for the dynamics and that the transition
probabilities converge  to this invariant measure. Their method is based on the asymptotic coupling technique from 
\cite{HMS11} and relies on the bounds from \cite{MWe15}.  This analysis does however not include the strong 
Feller property or the support theorem and does not imply exponential convergence to equilibrium. In the 
forthcoming article \cite{HM16} the authors present a general method to establish the strong Feller property, 
for solutions of SPDE solved in the framework of the theory of regularity structures. As an example this 
method is implemented for the dynamic $\Phi^4_3$ model. We expect that their method can  also treat the case 
of \eqref{eq:RSQEq_intro} but at first glance it only implies continuity of the associated Markov semigroup 
with respect to the total variational norm, whereas Theorem \ref{thm:strong_feller} implies H\"older continuity with 
respect to this norm.  

\subsection{Outline}

In Section \ref{s:Preliminaries} we introduce some notation for Wick powers and their approximations. The results in 
this section are essentially contained in \cite{dPD03} and \cite{MWe15} and the purpose of the section is mostly to fix notation. 
In Section \ref{s:Solving_the_Equation} we first briefly sketch the construction of solutions to \eqref{eq:RSQEq_intro}  including a
short time bound and a stability result which are used in 
Section \ref{s:Strong_Feller_Property}. We then prove the 
strong dissipative bound which is independent of the initial condition, improving on the bounds obtained
in \cite{MWe15}. In Section \ref{s:Existence_of_Invariant_Measures} we prove the Markov property for the 
solution using a simple factorization argument as in \cite{dPZ92} and we furthermore prove existence of 
invariant measures based on the bounds obtained in Section \ref{s:Solving_the_Equation}. The strong Feller property for the associated Markov semigroup is shown in Section 
\ref{s:Strong_Feller_Property}. Finally, in Section \ref{s:Exponential_Mixing_of_the_Phi42} we prove a support theorem for \eqref{eq:RSQEq_intro} in the 
case of $n=3$ which we combine with the results of the previous sections  to prove exponential mixing. 

\subsection{Notation}

Let $\TT^d$ be the $d$-dimensional torus of size $1$. We denote by $C^\infty(\RR^d)$ and $C^\infty(\TT^d)$ 
the space of real-valued smooth functions over $\RR^d$ and $\TT^d$ respectively as well as by $\mathscr{S}'(\TT^d)$ 
the dual space of Schwarz distributions acting on $\CC^\infty(\TT^d)$. We furthermore denote by $L^p(\TT^d)$ the space
of $p$-integrable functions on $\TT^d$, endowed with the norm
\begin{equation*}
  \|f\|_{L^p} := \left(\int_{\TT^2} |f(z)|^p \dd z\right)^\frac{1}{p}.
\end{equation*}

Although we only deal with spaces of real-valued functions, we prefer to work with the 
orthonormal basis $\{e_m\}_{m\in\ZZ^2}$ of trigonometric functions
\begin{equation*}
 e_m(z) := \e^{2\pi\ii m\cdot z},
\end{equation*}
for $z\in \TT^d$. Thus some complex-valued functions appear and we write 
$$
\lng f, g \rng = \int_{\TT^d} f(z) \overline{g(z)} \dd z
$$
for their inner product. In this notation, for $f\in L^2(\TT^d)$, the $m$-th Fourier coefficient is given by 
$$
\hat f(m) := \lng f, e_m \rng
$$ 
and since $f$ is real-valued we have the symmetry condition
\begin{equation}\label{eq:symmetry_cond}
 \hat f(-m) =\overline{ \hat f(m) },
\end{equation}
for any $m\in\ZZ^d$. For $f\in \mathscr{S}'(\TT^d)$ we define the $m$-th Fourier coefficient as
\begin{equation*}
 \hat f(m) := \lng f, \cos(2\pi \ii m  \cdot)\rng + \ii \lng f, \sin(2\pi \ii m \cdot)\rng,
\end{equation*}
with the convention that $\lng f, \cdot \rng$ stands for the action of $f$ on $C^\infty(\TT^d)$. 

For $\zeta\in \RR^d$ and $r>0$ we denote by $B(\zeta, r)$ the ball of radius $r$ centered at $\zeta$. We consider the 
annulus $\A=B\left(0, \frac{8}{3}\right)\setminus B\left(0, \frac{3}{4}\right)$ and a dyadic partition 
of unity $\left(\chi_{\kappa}\right)_{\kappa\geq-1}$ such that
\begin{enumerate}[i.]
 \item $\chi_{-1}=\tilde{\chi}$ and $\chi_{\kappa}=\chi(\cdot/2^{\kappa})$, $\kappa\geq 0$, 
 for two radial functions $\tilde{\chi},\chi\in C^\infty(\mathbb{R}^d)$.
 \item $\supp \tilde{\chi}\subset B\left(0, \frac{4}{3}\right)$ and $\supp \chi\subset \A$.
 \item $\tilde{\chi}(\zeta)+\sum_{\kappa=0}^\infty \chi(\zeta/2^{\kappa})=1$, for all $\zeta\in \mathbb{R}^d$.
\end{enumerate}
We furthermore let 
\begin{align*}
\A_{2^\kk} & :=  2^\kk \A, \quad \kk\geq 0.
\end{align*}
Notice that $\supp \chi_\kk \subset A_{2^\kk}$, for every $\kk\geq 0$. We also keep the convention that
$\A_{2^{-1}} = B\left(0, \frac{4}{3}\right)$. The existence of such a dyadic partition of unity is given by \cite[Proposition 2.10]{BCD11}.

For a function $f\in C^\infty(\mathbb{T}^d)$ we define the $\kappa$-th Littlewood-Paley block as 
\begin{equation}\label{eq:LPblocks}
 \delta_{\kappa} f(z) := \sum_{m\in\ZZ^2} \chi_{\kappa}(m)\hat f(m) \e^{2\pi \ii m\cdot z}, \quad \kappa\geq -1.
\end{equation}
Sometimes it is convenient to write \eqref{eq:LPblocks} as $\delta_\kk f = \eta_{\kappa}*f$, $\kappa\geq -1$, where 
\begin{equation*}
 \eta_\kk*f(\cdot)=\int_{\TT^d} \eta_\kk(\cdot-z) f(z) \dd z,
\end{equation*}
and
$$
\eta_{\kappa}(z) :=\sum_{m\in\ZZ^2} \chi_\kk(m) \e^{2\pi \ii m\cdot z}. 
$$
For $\alpha\in \mathbb{R}$ and $p,q\in[1,\infty]$ we define the non-homogeneous periodic Besov norm (see \cite[Section 2.7]{BCD11}), 
\begin{equation}\label{eq:besov_norm}
\|f\|_{\mathcal{B}_{p,q}^\alpha}:=\left\|\left(2^{\alpha \kappa}\|\delta_{\kappa}f\|_{L^p}\right)_{\kappa\geq-1}\right\|_{\ell^q}.
\end{equation}
The Besov space $\mathcal{B}_{p,q}^\alpha$ is defined as the completion of $C^\infty(\mathbb{T}^d)$ with respect to
the norm \eqref{eq:besov_norm}. We are mostly interested in the Besov space $\BB_{\infty,\infty}^\alpha$ which from now on we denote
by $\CC^\alpha$. Note that for $p=q=\infty$ our definition of Besov spaces differs from the standard definition as the set of those distributions for which \eqref{eq:besov_norm} is finite.
Our convention has the advantage that all Besov spaces are separable. Some basic properties of Besov spaces are 
collected in  Appendix \hyperref[Appendix:Besov_Spaces]{A}. 

Throughout the rest of this article for $\al\in(0,1)$ we let 
\begin{equation} \label{eq:Z_space}
 C^{n,-\al}(0;T):= C\left([0,T];\CC^{-\alpha}\right)\times C\left((0,T];\CC^{-\alpha}\right)^{n-1}
\end{equation}
and denote by $\underline{Z}=  \left(Z^{(1)}, Z^{(2)}, \ldots, Z^{(n)}\right)$ a generic $C^{n,-\al}(0;T)$-valued 
vector. For  $\al'>0$  we also define 
\begin{equation*}
 \VVert{\underline{Z}}_{\al;\al';T} := \max_{k=1,2,\ldots,n}
 \left\{\sup_{0\leq t\leq T} t^{(k-1)\al'}\|Z^{(k)}_t\|_{\CC^{-\alpha}}\right\}.
\end{equation*}

From now on we fix $\al_0\in(0,\frac{1}{n})$ (minus the regularity of the initial condition) as well as 
$\bt>0$ (regularity of the remainder) and $\gamma>0$ (blowup of the remainder close to $0$) such that 
\begin{equation}\label{eq:beta_gamma_cond} 
\frac{\bt+\al_0}{2} < \gamma, \quad \frac{\bt}{2} +n \gamma <1. 
\end{equation}

Throughout the whole article $C$ denotes a positive constant which might differ from line to line but we make explicit 
the dependence on different parameters where necessary. Furthermore, through the proofs of our statements, in cases where we do 
not want to keep track of the various constants in the inequalities we use $\lesssim$ instead of $\leq C$.  Finally, we use $a\vee b$ and $a\wedge b$ to 
denote the maximum and the minimum of $a$ and $b$.

\subsection*{Acknowledgements} 

 The authors would like to thank Martin Hairer, Jonathan Mattingly and Philipp Schoenbauer for helpfull discussions.  PT is supported by ESPRC as part of the MASDOC DTC at the University of Warwick, Grant No. EP/HO23364/1.
HW is supported by the Royal Society through the University Research Fellowship UF140187. 

\tikzsetexternalprefix{./macros/stoch_heat_eq/}

\section{Preliminaries}\label{s:Preliminaries}


In this section we present the necessary stochastic tools to handle \eqref{eq:RSQEq_intro}. In Section 
\ref{s:The_Stochastic_Heat_Equation_and_its_Wick_Powers} we introduce the stochastic heat equation along with its Wick powers in terms 
of abstract iterated stochastic integrals in the spirit of \cite[Chapter 1]{Nu06}. In Section \ref{s:Finite_Dimensional_Approximations} we describe how
these iterated stochastic integrals arise as limits of powers of solutions to  finite dimensional 
approximations after renormalization. 

\subsection{The Stochastic Heat Equation and its Wick Powers}\label{s:The_Stochastic_Heat_Equation_and_its_Wick_Powers}

Let $\xi$ be a space-time white noise on $\mathbb{R}\times \TT^2$ (see Appendix \hyperref[Appendix:The_Space_Time_White_Noise]{B}) 
on some probability space $(\Omega,\mathcal{F}, \PP)$, which is fixed from now on. We set 
\begin{equation}\label{eq:sigma_algebra}
\tilde{\mathcal{F}}_t=\sigma\left(\{\xi(\phi): \phi|_{(t,+\infty)\times\TT^2}\equiv 0, \, \phi\in L^2(\RR\times\TT^2)\}\right),
\end{equation}
for $t>-\infty$ and denote by $(\mathcal{F}_t)_{t > -\infty}$ the usual augmentation (as in \cite[Chapter 1.4]{RY99}) of the 
filtration $(\tilde{\mathcal{F}}_t)_{t > -\infty}$.

Consider the stochastic heat equation with zero initial condition at time $s\in(-\infty,\infty)$ 
\begin{equation}\label{eq:stoc_heat_eq_0}
 \left\{
\begin{array}{lcll}
\partial_t \<1_black>_{s,t} & = &\Delta \<1_black>_{s,t} - \<1_black>_{s,t} + \xi, & \text{in } (s,\infty)\times \mathbb{T}^2 \\
\<1_black>_{s,s} & = & 0, & \text{on } \mathbb{T}^2
\end{array}
\right..
\end{equation}

There are several ways to give a meaning to this equation. We simply use Duhamel's principle  (see \cite[Section 2.3]{Ev10})
as a definition and set for every $\phi\in C^\infty(\TT^2)$ and $s<t$ 
\begin{equation}\label{eq:stoc_heat_eq_0_sol} 
 \<1_black>_{s,t}(\phi):=\int_s^t \int_{\TT^2} \left\langle\phi, H(t-s,z-\cdot)\right\rangle\xi(\dd s, \dd z),
\end{equation}
where $H(r,\cdot)$, $r\in \mathbb{R}\setminus \{0\}$, stands for the periodic heat kernel on $L^2(\TT^2)$ given by
\begin{equation}\label{eq:periodic_heat_kernel}
 H(r,z) := \sum_{m\in\ZZ^2} \e^{-(1+4\pi^2|m|^2)r} e_m(z), 
\end{equation} 
for all $z\in\TT^2$. We furthermore let
$$
S(t):=\e^{-t} \e^{t\Delta}
$$ 
be the semigroup associated to the generator $\Delta -1$ in $L^2(\TT^2)$, i.e. the convolution operator with respect to the space variable $z\in\TT^2$ 
with the kernel $H(t,\cdot)$. 

The integral in \eqref{eq:stoc_heat_eq_0_sol} is a stochastic integral (see Appendix \hyperref[Appendix:The_Space_Time_White_Noise]{B} 
for definitions) and for fixed $s<t$, $\<1_black>_{s,t}$ is a family of Gaussian random variables indexed by $C^\infty(\TT^2)$. 

 Since it is more convenient to work with stationary processes we extend definition \eqref{eq:stoc_heat_eq_0_sol} 
for $s=-\infty$. For $\phi\in C^\infty(\TT^2)$, $n\geq 2$ and $t>-\infty$ we also consider the multiple stochastic integral 
(see Appendix \hyperref[Appendix:The_Space_Time_White_Noise]{B}) given by 
\begin{equation}\label{eq:wick}
 \<n_black>_{-\infty,t}(\phi):=
 \int_{\left\{(-\infty,t]\times\TT^2\right\}^n} \left\langle\phi,\prod_{k=1}^n H(t-s_k,z_k-\cdot)\right\rangle 
 \xi\left(\otimes_{k=1}^n \dd s_k, \otimes_{k=1}^n \dd z_k\right).
\end{equation}
We call $\<n_black>_{-\infty,\cdot}$ the $n$-th Wick power of $\<1_black>_{-\infty,\cdot}$ and we recall that for every $n\geq 1$ and $\phi\in C^\infty(\TT^2)$, 
$\<n_black>_{-\infty,\cdot}(\phi)$ is an element in the $n$-th homogeneous Wiener chaos (see Appendix \hyperref[Appendix:The_Space_Time_White_Noise]{B}
for definitions). We furthermore point out that $\<n_black>_{-\infty,\cdot}(\phi)$ is stationary, for every $\phi\in C^\infty(\TT^2)$.

The next theorem collects the optimal regularity properties of the processes $\{\<n_black>_{-\infty,\cdot}\}$, $n\geq 1$ and is very similar 
to the bounds originally derived in  \cite[Lemma 3.2]{dPD03}. The precise statement is a 
consequence of the Kolmogorov-type criterion \cite[Lemma 5.2, Lemma 5.3]{MWe15} and the proof follows similar lines to the one of 
\cite[Theorem 5.4]{MWe15}.  

\begin{theorem}\label{thm:mod_sol} Let $p\geq 2$. For every $n\geq 1$ and $t_0>-\infty$, the process $\<n_black>_{-\infty,t_0+\cdot}$ 
admits a modification $\widetilde{\<n_black>}_{-\infty,t_0+\cdot}$ such that 
$\widetilde{\<n_black>}_{-\infty,t_0+\cdot}\in C\left([0,T];\CC^{-\alpha}\right)$, for every $T>0$ and 
$\al>0$. Furthermore, there exists $\theta\equiv \theta(\al)\in(0,1)>0$ and $C\equiv C(T,\alpha,p)$ such that 
 \begin{equation} \label{eq:cont_modif}
  \EE \sup_{s,t\in [0,T]}
  \frac{\|\widetilde{\<n_black>}_{-\infty,t_0+t}-\widetilde{\<n_black>}_{-\infty,t_0+s}\|_{\CC^{-\alpha}} ^p}{|t-s|^{p \theta}} \leq C.
 \end{equation}
 For notational convenience we always refer to $\widetilde{\<n_black>}_{-\infty,\cdot}$ as $\<n_black>_{-\infty,\cdot}$. 
\end{theorem}

\begin{proof} See Appendix \hyperref[proof:mod_sol]{D}. 
\end{proof}

Notice that for every $t\geq s$ we have that
\begin{equation*}
 \<1_black>_{s,t} = \<1_black>_{-\infty,t} - S(t-s)\<1_black>_{-\infty,s}. 
\end{equation*}
It is then reasonable to define (see also \cite[pp. 34]{MWe15} for equivalent definitions) 
the $n$-th shifted Wick power of $\<1_black>_{s,t}$, $t> s>-\infty$, as
\begin{equation}\label{eq:shift_wick}
 \<n_black>_{s,t} : = \sum_{k=0}^n \binom{n}{k} (-1)^k \Big(S(t-s)\<1_black>_{-\infty,s} \Big)^k \<n-k_black>_{-\infty,t}.
\end{equation}
Here and below we use the convention $\<k_black>_{s,t}\equiv 1$  for $k=0$ and any $-\infty\leq s<t$.
We furthermore point out that the $n$-th shifted Wick power is not an element of the $n$-th homogeneous 
Wiener chaos (see Appendix \hyperref[Appendix:The_Space_Time_White_Noise]{B} for definitions). We refer 
the reader to Proposition \ref{prop:diagram_convergence} below for a natural approximation of the objects 
defined in \eqref{eq:shift_wick}.

At this point we would like to mention that one might work directly with $\<n_black>_{-\infty, \cdot}$ 
instead of introducing \eqref{eq:shift_wick} (see for example \cite{dPD03} and \cite{Ha14}). This alternative approach
has the advantage that the diagrams are stationary in time. However, we prefer to work with \eqref{eq:shift_wick}
(as in \cite{MWe15}) because  when proving the Markov property (see 
Section \ref{s:Markov_Property}) we use heavily  that  $\<n_black>_{s,t} $ is independent of  $\mathcal{F}_s$  
for any $s<t$ (see Proposition \ref{prop:diagram_convergence}). 
A slight disadvantage of our convention is the  logarithmic divergence of $\<n_black>_{s,t} $ as $t \downarrow s$ 
(see \eqref{eq:log_div}).
 
The next proposition uses the regularization property of the heat semigroup (see Proposition \ref{prop:Heat_Smooth})
to show that for every $t>s$ and $n\geq 2$, $\<n_black>_{s,t}$ is a well-defined element in a Besov space of negative
regularity close to $0$. 

\begin{proposition}\label{prop:time_bounds} Let $p\geq 2$ and $T>0$. For every $s_0>-\infty$, $\al\in(0,1)$ and $\al'>0$ 
there exist $\theta\equiv\theta(\al,\al')>0$ and $C\equiv C(T,\al,\al',p,n)$ such that
\begin{equation}\label{eq:log_div}
\EE\sup_{0\leq s\leq t }\left(s^{(n-1)\al' p} \|\<n_black>_{s_0,s_0+s}\|_{\CC^{-\alpha}}^{p}\right)\leq C t^{p\theta} ,
\end{equation}
for every $t\leq T$. 
\end{proposition}

\begin{proof}  We show \eqref{eq:log_div} for $s_0=0$. 

Let $\bar \al < \al \wedge \frac{2}{3} \al'$ and 
$V(s) = S(s)\left(-\<1_black>_{-\infty,0}\right)$. Using \eqref{eq:alpha_ineq} as well as Propositions 
\ref{prop:Mult_Ineq_II} and \ref{prop:Heat_Smooth} we have that
\begin{align*}
 \|V(s)^n\|_{\CC^{-\al}} & \lesssim \|V(s)\|^{n-1}_{\CC^{2\bar\al}} \|V(s)\|_{\CC^{-\bar \al}} \\
 & \lesssim s^{-(n-1)\frac{3}{2}\bar \al} \|\<1_black>_{-\infty,0}\|_{\CC^{-\bar \al}}^n.
\end{align*}
In a similar way, for $k\notin \{0,n\}$, we have that
\begin{align*}
 \|V(s)^k\<n-k_black>_{-\infty,s}\|_{\CC^{-\al}} \lesssim  s^{-k\frac{3}{2}\bar \al} \|\<1_black>_{-\infty,0}\|_{\CC^{-\bar \al}}^k 
 \|\<n-k_black>_{-\infty,s}\|_{\CC^{-\bar \al}}.
\end{align*}
Thus
\begin{align*}
 \|\<n_black>_{0,s}\|_{\CC^{-\al}}  \lesssim 
 s^{-(n-1) \frac{3}{2} \bar \al} \|\<1_black>_{-\infty,0}\|_{\CC^{-\bar\al}}^{n} 
  + \sum_{k=0}^{n-1} \binom{n}{k} 
 s^{-k \frac{3}{2}\bar \al } \|\<1_black>_{-\infty,0}\|_{\CC^{-\bar\al}}^k \|\<n-k_black>_{-\infty,s}\|_{\CC^{-\bar \al}}.
\end{align*}
%
Hence
\begin{align*}
 \EE\sup_{0\leq s\leq t} s^{(n-1)\al'p} & \|\<n_black>_{0,s}\|_{\CC^{-\al}}^p \lesssim 
  t^{(n-1)(\al'- \frac{3}{2}\bar \al)p} \EE\|\<1_black>_{-\infty,0}\|_{\CC^{-\bar \al}}^{np} \\
 + \sum_{k=0}^{n-1} \binom{n}{k} & t^{((n-1) \al' -k \frac{3}{2} \bar \al)p} 
 \left(\EE \|\<1_black>_{-\infty,0}\|_{\CC^{-\bar \al}}^{2kp}\right)^{\frac{1}{2}}
 \left(\EE\sup_{0\leq s\leq t} \|\<n-k_black>_{-\infty,s}\|_{\CC^{-\bar \al}}^{2p}\right)^{\frac{1}{2}},
\end{align*}
where we use a Cauchy--Schwarz inequality  in the last line. Combining with \eqref{eq:cont_modif} we finally 
obtain \eqref{eq:log_div}.
\end{proof}

\subsection{Finite Dimensional Approximations} \label{s:Finite_Dimensional_Approximations}

Let $\rho_\ee(z) = \sum_{|m|< \frac{1}{\ee}} e_m(z)$ and define a finite dimensional approximation of $\<1_black>_{s,t}$ by
\begin{equation*}
 \<1_black>^\ee_{s,t}(z) := \<1_black>_{s,t}(\rho_\ee(z-\cdot)).
\end{equation*}
We introduce the renormalization constant
\begin{align}
 \Re^\ee & := \|\mathbf{1}_{[0,\infty)}H_\ee\|_{L^2(\RR\times\TT^2)}^2, \label{eq:renorm_constant}
\end{align} 
where $H_\ee(r,z) = (H(r,\cdot)*\rho_\ee)(z)$ noting that $\Re^\ee\sim \log\ee^{-1}$ as $\ee\to 0^+$. For any integer 
$n\geq 1$ and $s\geq -\infty$ we define
\begin{equation*}
 \<n_black>_{s,t}^\ee : = \mathcal{H}_n(\<1_black>^\ee_{s,t}, \Re^\ee),
\end{equation*}
where $\mathcal{H}_n(X, C)$, $X,C\in\RR$, stands for the $n$-th Hermite polynomial given by the recursive formula
\begin{equation}\label{eq:hermite_pol}
\left\{
\begin{aligned}
 \mathcal{H}_{-1}(X,C)  & = 0, \quad \mathcal{H}_0(X,C) = 1 \\
 \mathcal{H}_n(X,C) & = X \mathcal{H}_{n-1}(X,C)-(n-1)C \mathcal{H}_{n-2}(X,C)  
\end{aligned}
\right..
\end{equation}
 The first three Hermite polynomials are given by $\HH_1(X,C) = X$, $\HH_2(X,C) = X^2- C$, $\HH_3(X,C) = X^3 -3 CX$.

\begin{proposition}\label{prop:diagram_convergence} Let $\al,\al'>0$. Then for every $n\geq 1$ and $p\geq 2$ we have that 
\begin{align*}
 \lim_{\ee\to 0^+}\EE\sup_{0\leq t\leq T} \, \|\<n_black>_{-\infty,s+t}^\ee-\<n_black>_{-\infty,s+t}\|_{\CC^{-\al}}^p & = 0, \\
 \lim_{\ee\to 0^+}\EE\sup_{0\leq t\leq T} \, t^{(n-1)\al'p}\|\<n_black>_{s,s+t}^\ee-\<n_black>_{s,s+t}\|_{\CC^{-\al}}^p & = 0, 
\end{align*}
for every $s>-\infty$. In particular, $\<n_black>_{s,s+\cdot}$ is independent of $\mathcal{F}_s$ and for $s_1,s_2\neq-\infty$, 
$\<n_black>_{s_1,s_1+\cdot}\stackrel{\text{law}}{=} \<n_black>_{s_2,s_2+\cdot}$.
\end{proposition}
 
\begin{proof} See Appendix \hyperref[proof:diagram_convergence]{E}.
\end{proof}

 An immediate consequence of the above proposition is the following corollary which we later use in Section \ref{s:Existence_of_Invariant_Measures} to
prove the Markov property. 

\begin{corollary}\label{prop:t+h_expansion} For every $n\geq 1$ and $t,h> 0$ the following identity holds $\PP$-almost
surely,
 \begin{align}\label{eq:0_equal}
  \<n_black>_{0,t+h} = \sum_{k=0}^n \binom{n}{k} \Big (S(h)\<1_black>_{0,t} \Big)^k \, \<n-k_black>_{t,t+h}.
 \end{align}
\end{corollary}

\begin{proof} 
It suffices to check \eqref{eq:0_equal} for $\<n_black>_{0,t+h}^\ee$. The result then follows from the previous proposition.
\end{proof} 
\tikzsetexternalprefix{./macros/sol_exist/}

\section{Solving the Equation} \label{s:Solving_the_Equation} 

\subsection{Analysis of the problem}

We are interested in solving the following renormalized stochastic partial differential equation,
\begin{equation}\label{eq:RSQEq}
\left\{
\begin{array}{lcll}
\partial_t X & = &\Delta X - X - \sum_{k=0}^n a_k :X^k: + \xi, & \text{in } \mathbb{R}_{+}\times \mathbb{T}^2 \\
X(0,\cdot) & = & x, & \text{on } \mathbb{T}^2
\end{array}
\right.,
\end{equation}
where $:X^k:$ stands for the $k$-th Wick power of $X$ and $x\in \CC^{-\al_0}$. Motivated by the da Prato--Debussche method (see 
\cite{dPD03}) we search for solutions to \eqref{eq:RSQEq} by writing $X=\<1_black>_{0,\cdot}+v$, where $\<1_black>_{0,\cdot}$ is the solution 
to \eqref{eq:stoc_heat_eq_0} and  the remainder $v$ is a mild solution of the following random partial differential equation,
\begin{equation}\label{eq:Remainder_Eq_1}
\left\{
\begin{array}{lcll}
\partial_t v & = &\Delta v - v - \sum_{k=0}^n a_k \sum_{j=0}^k \binom{k}{j} v^{j} \,\<k-j_black>_{0,\cdot} \\
v(0,\cdot) & = & x
\end{array}
\right..
\end{equation}
\begin{remark} In \cite{MWe15} $\<1_black>_{0,\cdot}$ is started from $x$ and consequently there \eqref{eq:Remainder_Eq_1} 
is solved with zero initial condition. Our approach of starting $\<1_black>_{0,\cdot}$ from $0$ and the remainder $v$ from $x$ has the advantage 
that the strong non-linear damping  in \eqref{eq:Remainder_Eq_1} acts directly on the initial condition, yielding  a strong dissipative bound for $v$ 
that is independent of $x$ (see Proposition \ref{prop:A_priori_Bounds}).
\end{remark}

We can rewrite \eqref{eq:Remainder_Eq_1} as
\begin{equation}\label{eq:Remainder_Eq_2}
\left\{
\begin{array}{lcll}
\partial_t v & = &\Delta v - v - \sum_{j=0}^{n} v^j Z^{(n-j)} \\
v(0,\cdot) & = & x
\end{array}
\right.,
\end{equation}
where 
$$
Z^{(n-j)}=\sum_{k=j}^n a_k \binom{k}{j} \<k-j_black>_{0,\cdot},  
$$
for all $0\leq j \leq n-1$ and $Z^{(0)}=a_n$.  

Notice that for every $\al\in(0,1)$, $\underline{Z}\in C^{n,-\al}(0;T)$ (see \eqref{eq:Z_space} for the definition of the space), 
for every $T>0$, and by \eqref{eq:log_div} for every $\al'>0$ there exists $\theta>0$ such that
\begin{equation}\label{eq:log_div_Z}
 \EE\VVert{\underline{Z}}_{\al;\al';t}^p \leq C t^{p\theta}
\end{equation}
for every $t\leq T$, $p\geq 2$. 

We now fix $\al<\al_0$ small enough (the precise value is fixed below in the proof of Theorem 
\ref{thm:Local_Ex_Un}) and $\underline{Z}\in C^{n,-\al}(0;T)$, for every $T>0$, and a norm 
$\VVert{\cdot}_{\al;\al';T}$, for some $\al'>0$ but still sufficiently small. We furthermore let  
\begin{equation}\label{eq:non_lin}
F(v,\underline{Z}) := \sum_{j=0}^n v^j Z^{(n-j)}.
\end{equation} 

\subsection{Mild Solutions} \label{s:Mild_Solutions} 

We are interested in solutions to the PDE problem \eqref{eq:Remainder_Eq_2}.

\begin{definition}\label{def:Mild_Sol_Rem} Let $T>0$ and $x\in \CC^{-\al_0}$. We say that a function $v$ is a 
mild solution of \eqref{eq:Remainder_Eq_2} up to time $T$ if $v\in C((0,T];\CC^{\beta})$ and
\begin{equation}\label{eq:mild_form}
 v_t=S(t)x-\int_0^t S(t-s)F(v_s,\underline{Z}_s) \dd s,
\end{equation}
for every $t\leq T$. 
\end{definition} 

The next theorem implies the existence of local in time solutions to \eqref{eq:Remainder_Eq_2}.  

\begin{theorem} \label{thm:Local_Ex_Un} {\normalfont(\cite[Proposition 4.4]{dPD03}, \cite[Theorem 6.2]{MWe15})} Let $x\in\CC^{-\al_0}$ and $R>0$ such that 
$\|x\|_{\CC^{-\al_0}}\leq R$. Then for every $\bt,\gamma>0$ satisfying \eqref{eq:beta_gamma_cond} 
and $T>0$ there exists 
$T^*\equiv T^*(R,\VVert{\underline{Z}}_{\al;\al';T})\leq T$ such that \eqref{eq:Remainder_Eq_2} 
has a unique mild solution on $[0,T^*]$ and
\begin{equation*}
 \sup_{0\leq s\leq T^*}\, s^{\gamma}\|v_s\|_{\CC^{\beta}} \leq 1.
\end{equation*}
If we furthermore assume that $\VVert{\underline{Z}}_{\al;\al';T}\leq 1$, then there exists $\theta>0$ and a 
constant $C>0$ independent of $R$ such that
\begin{equation}\label{eq:T^*_form}
 T^* = \left(\frac{1}{C(R+1)}\right)^{\frac{1}{\theta}}.
\end{equation} 
\end{theorem}

\begin{proof}
This theorem is (essentially) proved in \cite[Theorem 6.2]{MWe15}, but the expression \eqref{eq:T^*_form}
is not made explicit there; we give a sketch. It is sufficient to prove that for $T^*$ as in \eqref{eq:T^*_form}
the  operator 
\begin{equation*}
\mathscr{M}_{T^*} v_t = S(t) x + \int_0^t S(t-s)F(v_s,\underline{Z}_s) \dd s 
\end{equation*}
is a contraction on the set $\mathscr{B}_{T^*}:=\{ \sup_{0\leq s\leq T^*} s^\gamma \|v_s\|_{\CC^\bt}\leq 1 \}$, i.e. we need to show that 
$\mathscr{M}_{T^*}$ maps $\mathscr{B}_{T^*}$ into itself and that for $v, \tilde v \in \mathscr{B}_{T^*}$  we have  $ \sup_{0\leq s\leq T^*} s^\gamma \|\mathscr{M}_{T^*}v_s - \mathscr{M}_{T^*} \tilde v_s \|_{\CC^\bt}\leq (1-\lambda) \sup_{0\leq s\leq T^*} s^\gamma \| v_s -  \tilde v_s \|_{\CC^\bt} $ for some $\lambda >0$.
We only show the first property. First notice that
\begin{align*}
 \| \mathscr{M}_{T^*}v_t\|_{\CC^\bt} \lesssim t^{-\frac{\bt+\al_0}{2}} \|x\|_{\CC^{-\al_0}} 
 +\int_0^t (t-s)^{-\frac{\al +\bt}{2}} s^{- n\gamma} \dd s,
\end{align*}
where we use Proposition \ref{prop:Heat_Smooth} and we furthermore assume that $\al'<\gamma$. Choosing 
$\al>0$ sufficiently small so that $\frac{\al+\bt}{2}+n\gamma<1$ (see also \eqref{eq:beta_gamma_cond}) we have 
that
\begin{align*}
 \|  \mathscr{M}_{T^*}v_t\|_{\CC^\bt} \lesssim t^{-\frac{\bt+\al_0}{2}} \|x\|_{\CC^{-\al_0}} + t^{1-\frac{\al+\bt}{2}-n\gamma} 
\end{align*} 
and multiplying both sides by $t^\gamma$ we obtain that
\begin{align*}
 t^\gamma \| \mathscr{M}_{T^*}v_t\|_{\CC^\bt} & \lesssim t^{\gamma - \frac{\bt+\al_0}{2}} R + t^{1-\frac{\al+\bt}{2}-(n-1)\gamma} \\
 & \lesssim t^\theta (R+1).
\end{align*}
Then, for $T^* \equiv T^*(R)$ as in \eqref{eq:T^*_form} and every $t\leq T^*$ we get that
\begin{equation*}
 \sup_{0\leq s\leq t} s^\gamma \| \mathscr{M}_{T^*}v_s\|_{\CC^{\bt}} \leq 1,
\end{equation*}
which implies that indeed $ \mathscr{M}_{T^*}$  maps $ \mathscr{B}_{T^*}$ into itself.  
\end{proof} 

The next proposition is a stability result which we use later on in Section \ref{s:Strong_Feller_Property}. We first 
introduce some extra notation.  Let $\{\underline{Z}^\ee\}_{\ee\in(0,1)}$ take values in $C^{n,-\al}(0;T)$  such that
\begin{align*}
 \lim_{\ee\to 0^+} \VVert{\underline{Z}^\ee-\underline{Z}}_{\al;\al';T} = 0.
\end{align*}
Furthermore, let $F_\ee = \hat \Pi_\ee F$,  where $\hat \Pi_\ee$ is a linear smooth approximation such that the following properties hold 
for every $\al\in (0,1)$,
\begin{enumerate}[i.]
 \item \label{it:i} $\|\hat \Pi_\ee\|_{\CC^{-\al}\to \CC^{-\al}}\leq C$, for every $\ee\in(0,1)$. 
 \item \label{it:ii} For every $\delta >0$ there exists $\theta\equiv \theta(\delta)$ such that  
 $$
 \|\hat \Pi_\ee x - x\|_{\CC^{-\al-\delta}} \leq C \ee^\theta \|x\|_{\CC^{-\al}}.
 $$
\end{enumerate}
One can check that 
$\hat \Pi_\ee = \sum_{-1\leq \kk < \log_2 \ee^{-1}} \delta_\kk$ is such a linear smooth approximation.

Denote by $v^\ee$ the corresponding mild solution of \eqref{eq:Remainder_Eq_2} with $F$ replaced by $F_\ee$, $\underline{Z}$ by $\underline{Z}^\ee$ and 
initial condition $x^\ee = \hat \Pi_\ee x$ (short time existence of $v^\ee$ is ensured by the same arguments as in the proof of
\cite[Theorem 6.1]{MWe15}). We then have the following proposition.

\begin{proposition} \label{prop:approx_eq} Let $v$ be the unique solution to \eqref{eq:Remainder_Eq_2} 
on a closed interval $[0,T^*]$ (i.e. the solution does not explode at $T^*$).
Then for every $\ee\in (0,1)$ there exists a unique solution $v_\ee$ to the approximate equation up to some (possibly infinite) 
explosion time $T^*_\ee$. Furthermore, there exists $\ee_0>0$ such that for every $\ee<\ee_0$, $T^*_\ee \geq T^*$, and we have
\begin{equation*}
 \lim_{\ee\to 0^+} \sup_{0\leq t\leq T^*_\ee\wedge T^*} t^\gamma \|v_t - v^\ee_t\|_{\CC^{\bt}} = 0.
\end{equation*}
\end{proposition}

\begin{proof}  Let $\delta >0$ such that
$$
  \frac{\al_0+\delta+\bt}{2} + n\gamma <1.
$$
For $\ee\in(0,1)$ we notice that
\begin{equation*}
 v_t-v_t^\ee = S(t)\left(x - x^\ee\right) - \int_0^t S(t-s) \left( F(v_s,\underline{Z}_s) 
 - F_\ee(v_s^\ee,\underline{Z}_s^\ee)\right) \dd s
\end{equation*}
and using \eqref{eq:Heat_Smooth} and property \ref{it:ii} of $\hat \Pi_\ee$ we get 
\begin{equation*}
 \|v_t-v^\ee_t\|_{\CC^{\bt}} \lesssim t^{-\frac{\al_0+\delta+\bt}{2}} \ee^\theta \|x\|_{\CC^{-\al_0}}
 + \int_0^t (t-s)^{-\frac{\al+\delta+\bt}{2}} \|F(v_s,\underline{Z}_s) - 
 F_\ee(v^\ee_s,\underline{Z}_s^\ee)\|_{\CC^{-\al-\delta}} \dd s.
\end{equation*}
Using the triangle inequality as well as the properties \ref{it:i} and \ref{it:ii} of $\hat \Pi_\ee$ we have that
\begin{align*}
 \|F(v_s,\underline{Z}_s) - F_\ee(v^\ee_s,\underline{Z}_s^\ee)\|_{\CC^{-\al-\delta}} & \lesssim 
 \ee^\theta \|F(v_s,\underline{Z}_s)\|_{\CC^{-\al}} + \| F(v_s, \underline{Z}_s) 
 - F(v_s^\ee,\underline{Z}_s)\|_{\CC^{-\al}} \\
 & + \|F(v_s^\ee, \underline{Z}_s) - F(v_s^\ee,\underline{Z}^\ee_s)\|_{\CC^{-\al}}.
\end{align*} 
Let $M= \sup_{t\leq T^*} t^\gamma \|v_t\|_{\CC^\bt}$, $N = \VVert{\underline{Z}}_{\al;\al';T}$ and
$\tau^\ee = \inf\{ t>0, \, t\leq T^*_\ee: t^\gamma \|v_t-v^\ee_t\|_{\CC^{\bt}}>1\}$. Then, for every 
$t\leq \tau^\ee\wedge T^*$, we have the following bounds,
\begin{align*}
 \|F(v_s,\underline{Z}_s)\|_{\CC^{-\al}} & \leq C_1 s^{-n\gamma}, \\
 \|F(v_s, \underline{Z}_s) - F(v_s^\ee,\underline{Z}_s)\|_{\CC^{-\al}} & \leq C_2 s^{-(n-1)\gamma} 
 \sup_{t\leq \tau^\ee\wedge T^*} t^\gamma \|v_t-v_t^\ee\|_{\CC^{\bt}}, \\
 \|F(v_s^\ee, \underline{Z}_s) - F(v_s^\ee,\underline{Z}^\ee_s)\|_{\CC^{-\al}} & \leq C_3 s^{-(n-1)\gamma}
 \VVert{\underline{Z}-\underline{Z}^\ee}_{\al;\al';T},
\end{align*}
where the constants $C_1$, $C_2$ and $C_3$ depend on $M$ and $N$. Thus there exists $C\equiv C(M,N)>0$
such that
\begin{align*} 
 \|v_t-v_t^\ee\|_{\CC^{\bt}} & \leq C \Big( t^{-\frac{\al_0+\delta+\bt}{2}} \ee^\theta \|x\|_{\CC^{-\al_0}} + 
 \ee^\theta t^{1-\frac{\al+\delta+\bt}{2}-n\gamma} \\
 & + \sup_{t\leq \tau^\ee\wedge T^*} t^\gamma \|v_t-v_t^\ee\|_{\CC^{\bt}} t^{1-\frac{\al+\delta+\bt}{2} -(n-1)\gamma} \\
 & + \VVert{\underline{Z}-\underline{Z}^\ee}_{\al;\al';T} t^{1-\frac{\al+\delta+\bt}{2} -(n-1)\gamma}\Big).
\end{align*}
Multiplying by $t^\gamma$ and choosing $\tilde T^*\equiv \tilde T^*(M,N) >0$ sufficiently small we can assure that 
\begin{equation*}
 \sup_{t\leq \tilde T^*} t^\gamma \|v_t-v_t^\ee\|_{\CC^\bt} \leq 
 \ee^\theta \|x\|_{\CC^{-\al_0}} + \VVert{\underline{Z}-\underline{Z}^\ee}_{\al;\al';T} + \ee^\theta.
\end{equation*} 
Iterating the procedure if necessary we find $N^*>0$, independent of $\ee$ since $\tau^\ee\wedge T^*\leq T^*$, 
and $C>0$ such that 
\begin{equation}\label{eq:approx_conv}
 \sup_{t\leq \tau^\ee\wedge T^*} t^\gamma \|v_t-v_t^\ee\|_{\CC^\bt} \leq (N^* C+1) 
 \left(\ee^\theta \|x\|_{\CC^{-\al_0}} + \VVert{\underline{Z}-\underline{Z}^\ee}_{\al;\al';T} + \ee^\theta\right).
\end{equation} 
Let $\ee_0>0$ such that for every $\ee< \ee_0$
\begin{equation*}
 \ee^\theta \|x\|_{\CC^{-\al_0}} + \VVert{\underline{Z}-\underline{Z}^\ee}_{\al;\al';T} + \ee^\theta < 
 \frac{1}{(N^* C+1)}.
\end{equation*}
Then for every $\ee< \ee_0$
\begin{equation*} 
 \sup_{t\leq \tau^\ee\wedge T^*} t^\gamma \|v_t-v_t^\ee\|_{\CC^{\bt}} <1
\end{equation*}
and the definition of $\tau^\ee$ implies that $\tau^\ee\wedge T^* = T^*$, which proves the first claim. For the 
second claim we just let $\ee\to 0^+$ in \eqref{eq:approx_conv}.
\end{proof}
 
\subsection{Weak Solutions} \label{s:Weak_Solutions}

%
%
%
%
%

\begin{proposition} \label{p:p_testing} {\normalfont(\cite[Proposition 6.8]{MWe15})} Let $v\in C\left((0,T];\CC^\beta\right)$ be 
a mild solution to \eqref{eq:Remainder_Eq_2}. Then for all $s_0>0$ and $p\geq 2$ 
\begin{align}
\frac{1}{p}\left( \|v_t\|_{L^p}^p - \|v_{s_0}\|_{L^p}^p\right) & = \label{eq:weak_form}  \\
\int_{s_0}^t \Big(-(p-1) & \langle\nabla v_s, v_s^{p-2}\nabla v_s\rangle -\langle v_s, v_s^{p-1} \rangle 
-\langle F(v_s,\underline{Z}_s), v_s^{p-1} \rangle\Big) \dd s, \nonumber
\end{align}
for all $s_0\leq t\leq T$. In particular, if we differentiate with respect to $t$, 
\begin{equation}\label{eq:weak_form_tested_dif}
 \frac{1}{p} \partial_t\|v_t\|_{L^p}^p= -(p-1) \langle\nabla v_t, v_t^{p-2}\nabla v_t\rangle -\langle v_t, v_t^{p-1} \rangle 
 -\langle F(v_t,\underline{Z}_t), v_t^{p-1} \rangle,
\end{equation} 
for every $t\in(0,T)$. 
\end{proposition}

\begin{remark} The proof of \eqref{eq:weak_form} requires some time regularity on $v$. In 
this particular case one can prove that $v$ is H\"older continuous as a function from $(0,T)$ to $L^\infty(\TT^2)$ 
(see \cite[Proposition 6.5]{MWe15}) for some exponent strictly greater that $\frac{1}{2}$, which is enough to 
obtain \eqref{eq:weak_form}. 
\end{remark}

\subsection{A priori Estimates}\label{s:A_priori_Estimates}

Global existence of \eqref{eq:Remainder_Eq_2} for $x\in\CC^\bt$ was already established in \cite{MWe15} based on 
a priori estimates of the $L^p$ norm of $v$. Here we derive a stronger bound which does not depend on the initial 
condition $x$ and we use later on to prove the main results of Sections \ref{s:Existence_of_Invariant_Measures} and 
\ref{s:Exponential_Mixing_of_the_Phi42}. 

\begin{proposition}\label{prop:A_priori_Bounds} Let $v\in C((0,T];\CC^\beta)$ be a weak solution of \eqref{eq:Remainder_Eq_2} with
initial condition $x\in\CC^{-\al_0}$ and $p\geq 2$ be an even integer. Then for every $0 < t\leq T$ and $\lambda = \frac{p+n-1}{p}$
\begin{equation}\label{eq:Det_A_priori_bound}
 \|v_t\|_{L^p}^p \leq C\left[t^{-\frac{1}{\lambda -1}} 
 \vee  \left(\sum_{j,i} t^{-\al'p_i^j}
 \sup_{0\leq r\leq t}\left( r^{\al'p_i^j}\|Z_s^{(n-j)}\|_{\CC^{-\alpha}}^{p_i^j}\right)\right)^{\frac{1}{\lambda}}\right],
\end{equation}
for some $p_i^j>0$. In particular, the bound is independent from $\|x\|_{\CC^{-\al_0}}$ and the randomness outside of the 
interval $[0,t]$.
\end{proposition}

\begin{proof} Let 
\begin{equation}\label{eq:p_alpha_cond}
 \al <\frac{1}{(p+n-1)(n-1)}
\end{equation}
and recall that $F(v_s,\underline{Z}_s)=\sum_{j=0}^n v_s^j Z_s^{(n-j)}$. Thus
\begin{equation*}
 \langle F(v_s,\underline{Z}_s), v_s^{p-1} \rangle = \sum_{j=0}^n \langle v_s^{p+j-1}, Z_s^{(n-j)}\rangle
 =\|v_s^{p+n-1}\|_{L^1}+ \langle g_s, v_s^{p-1} \rangle,
\end{equation*}
where $g_s=\sum_{j=0}^{n-1} v_s^j Z_s^{(n-j)}$, and we rewrite \eqref{eq:weak_form_tested_dif} as
\begin{align} 
 \frac{1}{p} \partial_s\|v_s\|_{L^p}^p
  &=-\left((p-1)\|v_s^{p-2}|\nabla v_s|^2\|_{L^1} + \|v_s^{p+n-1}\|_{L^1}+\|v_s^p\|_{L^1} \right)
  - \langle g_s, v_s^{p-1} \rangle, \label{eq:rhs}
\end{align}
for all $0< s \leq t$,  where we use that $p$ is an even integer. Let 
\begin{equation}\label{eq:K_L_def}
 K_s:=\|v_s^{p-2}|\nabla v_s|^2\|_{L^1}, \quad L_s:=\|v_s^{p+n-1}\|_{L^1}.
\end{equation}
The idea is to control the terms of $\langle g_s, v_s^{p-1} \rangle$ by $K_s$ and $L_s$. 

We start with the leading term of $\langle g_s, v_s^{p-1}\rangle$, $\langle v_s^{p+n-2}, Z_s^{(1)} \rangle$. By Proposition 
\ref{prop:Inner_Prod}
\begin{equation}\label{eq:bracket_est}
 \langle v_s^{p+n-2}, Z_s^{(1)} \rangle \lesssim \|v_s^{p+n-2}\|_{\BB_{1,1}^\alpha} \|Z_s^{(1)}\|_{\CC^{-\alpha}}.
\end{equation}
Using \eqref{eq:gradient_estimate}
\begin{equation}\label{eq:gradient_est}
 \|v_s^{p+n-2}\|_{\BB_{1,1}^\alpha} \lesssim \|v_s^{p+n-2}\|_{L^1}^{1-\alpha}\|v_s^{p+n-3}|\nabla v_s|\|_{L^1}^\alpha+\|v_s^{p+n-2}\|_{L^1}.
\end{equation}
We handle each term of \eqref{eq:gradient_est} separately. First we notice, using Jensen's inequality, that 
$\|v_s^{p+n-2}\|_{L^1} \lesssim L_s^{\frac{p+n-2}{p+n-1}}$. For the gradient term, using the Cauchy-Schwarz inequality we obtain
\begin{equation}\label{eq:grad_est1}
 \|v_s^{p+n-3}|\nabla v_s|\|_{L^1} \leq \|v_s^{p+2(n-2)}\|_{L^1}^{\frac{1}{2}} K_s^{\frac{1}{2}}.
\end{equation}
Recall the Sobolev inequality 
\begin{equation*}
 \|f\|_{L^q}\lesssim \left(\|f\|_{L^2}^2+\|\nabla f\|_{L^2}^2\right)^{\frac{1}{2}},
\end{equation*}
for every $q<\infty$ (see \cite[Section 6]{NPV11},\cite[Section 5.6]{Ev10} for Sobolev inequalities in the same spirit).
In particular, for $q=\frac{2(p+2(n-2))}{p}$, we have that
\begin{equation*}
 \|v_s^{\frac{p}{2}}\|_{L^q}^{\frac{q}{2}}\lesssim \|v_s^{\frac{p}{2}}\|_{L^2}^{\frac{q}{2}}+\|\nabla (v_s)^{\frac{p}{2}}\|_{L^2}^{\frac{q}{2}},
\end{equation*}
which implies
\begin{equation}\label{eq:grad_est2}
 \|v_s^{p+2(n-2)}\|_{L^1}^{\frac{1}{2}}\lesssim \|v_s^p\|_{L^1}^{\frac{1}{2}+\frac{n-2}{p}}+K_s^{\frac{1}{2}+\frac{n-2}{p}},
\end{equation}
where $\|v_s^p\|_{L^1}^{\frac{1}{2}+\frac{n-2}{p}}\lesssim L_s^{\frac{\frac{p}{2}+n-2}{p+n-1}}$ by Jensen's inequality. Combining
\eqref{eq:gradient_est}, \eqref{eq:grad_est1} and \eqref{eq:grad_est2}
\begin{equation}\label{eq:grad_est3}
 \|v_s^{p+n-2}\|_{\BB_{1,1}^\alpha}\lesssim K_s^{\frac{\alpha}{2}}L_s^{\frac{(p+n-2)-\frac{p}{2}\alpha}{p+n-1}}
 +K_s^{\left(1+\frac{n-2}{p}\right)\alpha}L_s^{\frac{(p+n-2)(1-\alpha)}{p+n-1}}+L_s^{\frac{p+n-2}{p+n-1}}.
\end{equation}
By \eqref{eq:p_alpha_cond} we notice that
\begin{equation*}
 \frac{\alpha}{2}+\frac{(p+n-2)-\frac{p}{2}\alpha}{p+n-1}<1
\end{equation*}
and 
\begin{equation*}
\left(1+\frac{n-2}{p}\right)\alpha+\frac{(p+n-2)(1-\alpha)}{p+n-1}<1,
\end{equation*}
thus we can find $\gamma_1, \gamma_2, \gamma_3, \gamma_4<1$ such that 
\begin{equation*}
\frac{\alpha}{2\gamma_1}+\frac{(p+n-2)-\frac{p}{2}\alpha}{(p+n-1)\gamma_2}=1
\end{equation*}
and
\begin{equation*}
\left(1+\frac{n-2}{p}\right)\frac{\alpha}{\gamma_3}+\frac{(p+n-2)(1-\alpha)}{(p+n-1)\gamma_4}=1.
\end{equation*}
In particular, we choose $\gamma_1=\frac{(p+n-1)\alpha}{2}$, $\gamma_2=\frac{(p+n-2)-\frac{p}{2}\alpha}{p+n-2}$,
$\gamma_3=\frac{(p+n-2)(p+n-1)\alpha}{p}$ and $\gamma_4=(1-\alpha)$.
Applying Young's inequality to \eqref{eq:grad_est3} and combining with \eqref{eq:bracket_est} we obtain that
\begin{equation*}
 \langle v_s^{p+n-2}, Z_s^{(1)} \rangle \lesssim \left(K_s^{\gamma_1}+L_s^{\gamma_2}+K_s^{\gamma_3}+
 L_s^{\gamma_4}+L_s^{\frac{p+n-2}{p+n-1}}\right) \|Z^{(1)}_s\|_{\CC^{-\alpha}},
\end{equation*}
while using the fact that $\sup_{\zeta\geq0}-\zeta+a\zeta^\gamma\lesssim a^{\frac{1}{1-\gamma}}$, $\gamma<1$, we obtain the final bound
\begin{equation}\label{eq:n_final}
 \langle v_s^{p+n-2}, Z_s^{(1)} \rangle \leq \frac{1}{n} \big(K_s+\frac{1}{2}L_s\big)+ C \sum_{i=1}^5 \left(\|Z_s^{(1)}\|_{\CC^{-\al}}^{\frac{1}{1-\gamma_i}}\right),
\end{equation}
where $\gamma_5=\frac{p+n-2}{p+n-1}$ and $C$ a positive universal constant. 

For the remaining terms of $\langle g_s, v_s^{p-1}\rangle$ we need to estimate $\langle v_s^{p+j-1}, Z_s^{(n-j)} \rangle$,
for all $0\leq j\leq n-2$. Proceeding in the same spirit of calculations as above we first obtain that
\begin{equation*}
 \|v_s^{p+j-1}\|_{\BB_{1,1}^\alpha} \lesssim K_s^{\frac{\alpha}{2}}L_s^{\frac{(p+j-1)-\frac{p}{2}\alpha}{p+n-1}}
 +K_s^{\left({1+\frac{j-1}{p}}\right)\alpha}L_s^{\frac{(p+j-1)(1-\alpha)}{p+n-1}}
 +L_s^{\frac{p+j-1}{p+n-1}}.
\end{equation*}
We define the exponents $\gamma_1^j=\frac{(p+n-1)\alpha}{2}$, $\gamma_2^j=\frac{(p+j-1)-\frac{p}{2}\alpha}{p+n-2}$,
$\gamma_3^j=\frac{(p+j-1)(p+j)\alpha}{p}$ and $\gamma_4^j=\frac{(p+j)(1-\alpha)}{p+n-1}$. Note that
\eqref{eq:p_alpha_cond} implies that $\gamma_1^j,\gamma_2^j,\gamma_3^j,\gamma_4^j<1$ and we also have that
\begin{equation*}
\frac{\alpha}{2\gamma_1^j}+\frac{(p+j-1)-\frac{p}{2}\alpha}{(p+n-1)\gamma_2^j}=1
\end{equation*}
and
\begin{equation*}
\left(1+\frac{j-1}{p}\right)\frac{\alpha}{\gamma_3^j}+\frac{(p+j-1)(1-\alpha)}{(p+n-1)\gamma_4^j}=1.
\end{equation*}
Applying once more Young's inequality
\begin{equation*}
 \langle v_s^{p+j-1}, Z_s^{(n-j)} \rangle \lesssim \left( K_s^{\gamma_1^j}+L_s^{\gamma_2^j}+K_s^{\gamma_3^j}
 +L_s^{\gamma_4^j}+L_s^\frac{p+j-1}{p+n-1}\right) \|Z_s^{(n-j)}\|_{\CC^{-\alpha}}.
\end{equation*}
As before (see \eqref{eq:n_final}), we obtain the bound
\begin{equation}\label{eq:j_final}
 \langle v_s^{p+j-1}, Z_s^{(n-j)} \rangle \leq \frac{1}{n}\big( K_s+ \frac{1}{2}L_s\big)
 + C \sum_{i=1}^5 \left(\|Z^{(n-j)}_s\|_{\CC^{-\alpha}}^{\frac{1}{1-\gamma_i^j}}\right),
\end{equation}
for all $0\leq j\leq n-2$, where $\gamma_5^j=\frac{p+j-1}{p+n-1}$. Thus, by \eqref{eq:n_final} and \eqref{eq:j_final},
\begin{equation}\label{eq:final_bound}
 \langle g_s, v_s^{p-1} \rangle \leq \big( K_s+ \frac{1}{2} L_s\big)
 + C \sum_{j=0}^{n-1} \sum_{i=1}^5 \left(\|Z^{(n-j)}_s\|_{\CC^{-\alpha}}^{\frac{1}{1-\gamma_i^j}}\right),
\end{equation}
where $\gamma_i^{n-1}=\gamma_i$, for all $i=\{1,\ldots, 5\}$.

Finally, for $p_i^j = \frac{1}{1-\gamma_i^j}$, combining \eqref{eq:rhs} and \eqref{eq:final_bound} we obtain
\begin{equation*}
 \frac{1}{p}\partial_s\|v_s\|_{L^p}^p + \|v_s\|_{L^p}^p+ (p-2) K_s + \frac{1}{2} L_s
 \leq C \sum_{j,i} \|Z^{(n-j)}_s\|_{\CC^{-\alpha}}^{p_i^j}.
\end{equation*}
Let $t>s$ and notice that by \eqref{eq:log_div_Z}, for $r\in(s,t)$,
\begin{equation*}
 \sum_{j,i} \|Z_r^{(n-j)}\|_{\CC^{-\alpha}}^{p_i^j} \leq 
 \sum_{j,i} r^{-\al'p_i^j} \sup_{s\leq r\leq t}\left(r^{\al'p_i^j}\|Z_r^{(n-j)}\|_{\CC^{-\alpha}}^{p_i^j}\right)  
\end{equation*}
for every $\al'>0$. Thus for $r\in[s,t]$
\begin{equation*}
 \frac{1}{p}\partial_r \|v_{r}\|_{L^p}^p+ \frac{1}{2} L_{r} \leq C 
 \sum_{j,i} s^{-\al'p_i^j} \sup_{s\leq r\leq t}\left( r^{\al'p_i^j}\|Z_r^{(n-j)}\|_{\CC^{-\alpha}}^{p_i^j}\right).
\end{equation*}
By Jensen's inequality, for $\lambda = \frac{p+n-1}{p}$, we get that
\begin{equation*}
 \partial_r \|v_r\|_{L^p}^p+ C_1 \left(\|v_r\|_{L^p}^p\right)^\lambda  \leq C_2
 \sum_{j,i} s^{-\al'p_i^j}\sup_{s\leq r\leq t}\left( r^{\al'p_i^j}\|Z_r^{(n-j)}\|_{\CC^{-\alpha}}^{p_i^j}\right),
\end{equation*}
and if we let $f(r) = \|v_r\|_{L^p}^p$, $r\geq s$, by Lemma \ref{lem:comp_test} 
\begin{align}
 f(r) \leq \frac{f(s)}{\left(1+(r-s)f(s)^{\lambda-1}(\lambda-1)\tilde{C}_1\right)^{\frac{1}{\lambda-1}}} & \label{eq:f(r)_bound} \\
 \vee \Bigg(\frac{2C_2 }{C_1}\sum_{j,i} s^{-\al'p_i^j} 
 \sup_{s\leq r\leq t} & \left( r^{\al'p_i^j}\|Z_r^{(n-j)}\|_{\CC^{-\alpha}}^{p_i^j}\right) \Bigg)^{\frac{1}{\lambda}}, \nonumber
\end{align}
where $\tilde{C}_1 = C_1/2$. In particular for $r=t$ and $s=t/2$ we have the bound
\begin{equation*}
 \|v_{t}\|_{L^p}^p \leq C\left[ t^{-\frac{1}{\lambda -1}} \vee 
 \left(\sum_{j,i} t^{-\al'p_i^j} \sup_{0\leq r\leq t}\left( r^{\al'p_i^j}\|Z_r^{(n-j)}\|_{\CC^{-\alpha}}^{\tilde{\gamma}_i^j}
 \right)\right)^{\frac{1}{\lambda}}\right],
\end{equation*}
which completes the proof.
\end{proof}

\begin{lemma}[Comparison Test]\label{lem:comp_test} Let $\lambda >1$ and $f: [0,T]\to [0,\infty)$ differentiable
such that
\begin{equation*}
 f'(t) + c_1 f(t)^\lambda \leq c_2,
\end{equation*}
for every $t\in [0,T]$. Then for $t>0$
\begin{equation*}
 f(t) \leq \frac{f(0)}{\left(1+tf(0)^{\lambda-1}(\lambda-1)\frac{c_1}{2}\right)^{\frac{1}{\lambda-1}}} 
 \vee \left(\frac{2c_2}{c_1}\right)^{\frac{1}{\lambda}} \leq  
 t^{-\frac{1}{\lambda-1}} \left((\lambda-1) \frac{c_1}{2}\right)^{-\frac{1}{\lambda-1}}
 \vee \left(\frac{2c_2}{c_1}\right)^{\frac{1}{\lambda}}.
\end{equation*}
\end{lemma}

\begin{proof} Let $t>0$. Then one of the following  holds:
\begin{enumerate} [I.]
 \item There exists $s_0\leq t$ such that $f(s_0) 
 \leq \left(\frac{2c_2}{c_1}\right)^{\frac{1}{\lambda}}$.
 \item For every $s\leq t$, $f(s) > \left(\frac{2c_2}{c_1}\right)^{\frac{1}{\lambda}}$.
\end{enumerate}
In the second case, using the assumption we have that for every $s\leq t$ 
$$
f'(s) + \frac{c_1}{2} f(s)^\lambda \leq 0
$$
and solving the above differential inequality on $[0,t]$ implies that
$$
f(t) \leq 
\frac{f(0)}{\left(1+tf(0)^{\lambda-1}(\lambda-1)\frac{c_1}{2}\right)^{\frac{1}{\lambda-1}}}.
$$
In the first case, assume for contradiction that $f(t) > 
\left(\frac{2c_2}{c_1}\right)^{\frac{1}{\lambda}}$ and let 
$$
s^* = \sup\left\{s < t : f(s) \leq \left(\frac{2c_2}{c_1}\right)^{\frac{1}{\lambda}}\right\}.
$$
Then $f(s) > \left(\frac{2c_2}{c_1}\right)^{\frac{1}{\lambda}}$, for every $s\in(s^*,t]$,
while $f(s^*) = \left(\frac{2c_2}{c_1}\right)^{\frac{1}{\lambda}}$ by continuity. 
However, the assumption implies
$$
f'(s) + \frac{c_1}{2} f(s)^\lambda \leq 0
$$
and in particular $f'(s) \leq 0$. But then 
$$
f(t) = f(s^*) + \int_{s^*}^t f'(s) ds \leq \left(\frac{2c_2}{c_1}\right)^{\frac{1}{\lambda}},
$$
which is a contradiction. 
\end{proof} 


The next theorem implies global existence of \eqref{eq:Remainder_Eq_2}. Though it was already established in
\cite{MWe15}, we present it here for completeness. 

\begin{theorem} \label{thm:global_ex}
For every initial condition $x\in \CC^{-\al_0}$ and $\bt>0$ as in \eqref{eq:beta_gamma_cond} there 
exists a unique solution $v\in C((0,\infty); \CC^\beta)$ of \eqref{eq:Remainder_Eq_2}.
\end{theorem}

\begin{proof} Let $T>0$. Using the a priori estimate \eqref{eq:Det_A_priori_bound} which depends only on 
$\VVert{\underline{Z}}_{\al;\al';T}$, by Theorem 
\ref{thm:Local_Ex_Un} there exists $T^*\leq T$ and a unique solution up to time $T^*$ of \eqref{eq:Remainder_Eq_2}. Using again 
\eqref{eq:Det_A_priori_bound} and Theorem \ref{thm:Local_Ex_Un} we construct a solution of \eqref{eq:Remainder_Eq_2} 
on $[T^*, 2T^*\wedge T]$ with initial condition $v_{T^*}$ which satisfies the same a priori bounds depending on 
$\VVert{\underline{Z}}_{\al;\al';T}$. We then proceed similarly until the whole interval $[0,T]$ is covered. To prove uniqueness 
we proceed as in the proof of Theorem \cite[Theorem 6.2]{MWe15}.  
\end{proof}

\begin{corollary}\label{cor:moments_est} For $x\in\CC^{-\al_0}$ let $X(\cdot;x) = \<1_black>_{0,\cdot} + v$, 
where $v$ is the solution to \eqref{eq:Remainder_Eq_1}. Then for every $\al>0$ and $p\geq 2$ 
\begin{equation} \label{eq:time_ind_moments}
 \sup_{x\in\CC^{-\al_0}}\sup_{t\geq 0}\, \left(t^{\frac{p}{n-1}}\wedge 1\right) \EE\|X(t;x)\|_{\CC^{-\al}}^p <\infty.
\end{equation}
\end{corollary}

\begin{remark} 
Notice that the bound \eqref{eq:time_ind_moments} does not follow immediately by taking the expectation 
of the a priori bound \eqref{eq:Det_A_priori_bound} on $v_t$. In fact the expectation of the supremum 
$ \sup_{0\leq r\leq t}\left( r^{\al'p_i^j}\|Z_s^{(n-j)}\|_{\CC^{-\alpha}}^{p_i^j}\right)$ on the right hand side of this estimate
is finite for every $t< \infty$ but it is not uniformly bounded in $t$. However, as \eqref{eq:Det_A_priori_bound} does not depend on
the initial condition we can just restart \eqref{eq:RSQEq} at time $t-1$ for $t>1$ and apply Proposition \ref{prop:A_priori_Bounds} for the restarted 
solution to obtain a bound which depends only on the randomness inside the interval $[t-1,t]$. Given that the diagrams have 
the same law on intervals of the same size (see Proposition \ref{prop:diagram_convergence}) we then obtain a bound which 
is independent of $t$.
\end{remark}

\begin{proof} Let $t>1$ and notice that by Lemma \ref{lem:re_expand} (see Section \ref{s:Invariant_Measures} for 
statement and proof) $X(t;x) = \<1_black>_{t-1,t} + \tilde{v}_{t-1,t}$ where $\tilde{v}_{t-1,r}$, $r\geq t-1$, solves 
\eqref{eq:Remainder_Eq_1} with initial condition $X(t-1;x)$ and  
$$
Z^{(n-j)} = \sum_{k=j}^n a_k \binom{k}{j} \<k_black>_{t-1,t-1+\cdot},
$$
for every $0\leq j \leq n-1$. Applying Proposition \ref{prop:A_priori_Bounds} on $\tilde{v}_{t-1,\cdot}$ we then have
\begin{equation}\label{eq:restarted_bound}
 \|\tilde{v}_{t-1,t}\|_{L^p}^p \lesssim 1 \vee \left(\sum_{j,i} \sup_{t-1\leq r\leq t}\left(\big(r-(t-1)\big)^{\al'p_i^j} 
 \|Z_r^{(n-j)}\|_{\CC^{-\alpha}}^{p_i^j}\right)\right)^{\frac{1}{\lambda}},
\end{equation}
for every $p\geq 2$. To prove \eqref{eq:time_ind_moments} we fix $\al>0$ and using the embedding 
$L^p\hookrightarrow \CC^{-\al}$ for $p\geq \frac{2}{\al}$ (see \eqref{eq:B^0_emb} and Proposition \ref{prop:Besov_Emb}) we first notice that for $t > 1$
\begin{align*}
 \EE\|X(t;x)\|_{\CC^{-\al}}^p & \lesssim \EE\|\<1_black>_{t-1,t}\|_{\CC^{-\al}}^p + \EE\|\tilde{v}_{t-1,t}\|_{\CC^{-\al}}^p \\
 & \lesssim \EE\|\<1_black>_{t-1,t}\|_{\CC^{-\al}}^p + \EE\|\tilde{v}_{t-1,t}\|_{L^p}^p. 
\end{align*}
Combining with \eqref{eq:restarted_bound} and given that for every $k\geq 1$ the law of $\<k_black>_{t-1,t+\cdot}$ does not 
depend on $t$ we obtain that
\begin{equation*}
 \sup_{t\geq 1} \EE\|X(t;x)\|_{\CC^{-\al}}^p <\infty.
\end{equation*}
Finally, using \eqref{eq:Det_A_priori_bound} (and by possibly tuning down $\al'$ in the same equation) for $t\leq 1$ we get
\begin{align*}
 \EE \|X(t;x)\|_{\CC^{-\al}}^p & \lesssim \EE\|\<1_black>_{-\infty,t}\|_{\CC^{-\al}}^p + \EE\|v_t\|_{L^p}^p \\
 & \lesssim 1 + t^{-\frac{p}{n-1}},
\end{align*}
which completes the proof.
\end{proof}
\tikzsetexternalprefix{./macros/inv_measure/}

\section{Existence of Invariant Measures}\label{s:Existence_of_Invariant_Measures}

\subsection{Markov Property}\label{s:Markov_Property}

For $x\in \CC^{-\al_0}$ we write $X(\cdot;x) = \<1_black>_{0,\cdot}+ v$ where $v$ is the solution to \eqref{eq:Remainder_Eq_1}
with initial condition $x$. We  introduce a variant of the notation \eqref{eq:non_lin} and set 
\begin{equation} \label{eq:tilde_non_lin}
 \tilde F\left(v,\left(\<k_black>_{0,\cdot}\right)_{k=1}^n\right) = 
 \sum_{k=0}^n a_k \sum_{j=0}^k \binom{k}{j} v^{j} \,\<k-j_black>_{0,\cdot}.
\end{equation}

We denote by $B_b(\CC^{-\al_0})$ and $C_b(\CC^{-\al_0})$ the spaces of bounded and continuous functions 
from $\CC^{-\al_0}$ to $\RR$, both endowed with the norm%
\begin{equation*}
 \|\Phi\|_\infty := \sup_{x\in\CC^{-\al_0}} |\Phi(x)|.
\end{equation*}

For every $\Phi\in B_b(\CC^{-\al_0})$ and $t\in [0,\infty)$ we define the map $P_t:\Phi\mapsto P_t\Phi$ by
\begin{equation}\label{eq:semigroup}
 P_t\Phi(x):= \EE\Phi(X(t;x)),
\end{equation}
for every $x\in \CC^{-\al_0}$.  

In this section we prove that $\{X(t;\cdot): t\geq 0\}$ is a Markov process with transition 
semigroup $\{P_t:t \geq 0\}$ with respect to the filtration $\{\mathcal{F}_t:t\geq 0\}$ defined in 
\eqref{eq:sigma_algebra}.  

We first prove the following lemma. 

\begin{lemma}\label{lem:re_expand} Let $X(\cdot;x)=\<1_black>_{0,\cdot} + v$. Then, for every $h>0$, 
\begin{equation*}
X(t+h;x)= \<1_black>_{t,t+h} + \tilde{v}_{t,t+h},
\end{equation*}
where the remainder $\tilde{v}_{t,t+\cdot}$ solves \eqref{eq:Remainder_Eq_1} driven by the vector 
$\left(\<k_black>_{t,t+\cdot}\right)_{k=1}^n$ and initial condition $X(t;x)$, i.e.
\begin{equation*}
 \tilde{v}_{t,t+h} = S(h)X(t;x) 
 -\int_0^h S(h-r) \tilde F\left(\tilde{v}_{t,t+r}, \left(\<k_black>_{t,t+r}\right)_{k=1}^n\right) \dd r.
\end{equation*}
\end{lemma}

\begin{proof} Notice that 
for $h>0$
\begin{align*}
 X(t+h;x) & = \<1_black>_{0,t+h} + v_{t+h} \\
 & = \<1_black>_{t,t+h} + \tilde{v}_{t,t+h},
\end{align*}
where
\begin{equation*}
 \tilde{v}_{t,t+h} = S(h)X(t;x) - \int_0^h S(h-r) \tilde F\left(v_{t+r}, \left(\<k_black>_{0,t+r}\right)_{k=1}^n\right)\dd r.
\end{equation*}
By \eqref{eq:0_equal} we have that
\begin{align*}
 \tilde F\left(v_{t+r}, \left(\<k_black>_{0,t+r}\right)_{k=1}^n\right) & = 
 \sum_{k=0}^n a_k \sum_{j=0}^k \binom{k}{j} v_{t+r}^j \<k-j_black>_{0,t+r} \\
 & = \sum_{k=0}^n a_k \sum_{i=0}^k \binom{k}{i} \tilde{v}_{t,t+r}^i \<k-i_black>_{t,t+r}, 
\end{align*}
where we use a binomial expansion of $v_{t+r}^j$ and a change of summation. Hence
\begin{equation*}
\tilde F\left(v_{t+r}, \left(\<k_black>_{0,t+r}\right)_{k=1}^n\right) = 
\tilde F\left(\tilde{v}_{t,t+r}, \left(\<k_black>_{t,t+r}\right)_{k=1}^n\right),
\end{equation*}
which completes the proof.
\end{proof}

The fact that $\{X(t;\cdot):t\geq 0\}$ is a Markov process is an  immediate consequence of the following theorem. 

\begin{theorem}\label{thm:Markov_Prop} Let $X(\cdot;x)$ be as in the lemma above with $x\in\CC^{-\al_0}$. Then for every 
$\Phi\in B_b(\CC^{-\alpha_0})$ and $t\geq 0$
\begin{equation*} 
 \EE(\Phi(X(t+h;x))|\mathcal{F}_t)=P_h\Phi(X(t;x)),
\end{equation*}
for all $h\geq 0$.
 
\end{theorem}

\begin{proof} Let $h\geq 0$ and $\Phi\in B_b(\CC^{-\alpha_0})$ and write 
\begin{equation*}
 \mathscr{T}\left(X(t;x); h;\left(\<k_black>_{t,t+\cdot}\right)_{k=1}^n\right)
\end{equation*}
to denote the solution of \eqref{eq:Remainder_Eq_1} at time $h$, driven by the vector 
$\left(\<k_black>_{t,t+\cdot}\right)_{k=1}^n$ and initial condition $X(t;x)$. By Proposition \ref{prop:t+h_expansion}
and \cite[Proposition 1.12]{dPZ92}
\begin{align*}
\EE(\Phi(X(t+h;x))|\mathcal{F}_t) & = \bar{\Phi}(X(t;x)),
\end{align*}
where for $w\in\CC^{-\al_0}$ 
\begin{equation*}
\bar{\Phi}(w) = \EE\Phi\left(\<1_black>_{t,t+h} 
+ \mathscr{T}\left(w; h; \left(\<k_black>_{t,t+\cdot}\right)_{k=1}^n \right)\right).
\end{equation*}
Here we use the fact that $X(t;x)$ is $\mathcal{F}_t$-measurable and that the vector 
$\left(\<k_black>_{t,t+\cdot}\right)_{k=1}^n$ is $\mathcal{F}_t$-independent (see Proposition 
\ref{prop:diagram_convergence}). Given that 
$\left(\<k_black>_{t,t+\cdot}\right)_{k=1}^n\stackrel{\text{law}}{=}\left(\<k_black>_{0,\cdot}\right)_{k=1}^n$ (see again
Proposition \ref{prop:diagram_convergence}) and the fact that \eqref{eq:Remainder_Eq_1} has a unique solution driven by any 
vector $\underline{Y}\in C^{n,-\al}(0;T)$, for $T>0$, and any initial condition $w\in\CC^{-\al_0}$, we have that
\begin{equation*}
 \bar{\Phi}(w) = P_h\Phi(w),
\end{equation*}
which completes the proof if we set $w=X(t;x)$. 
\end{proof}

The theorem above implies that $\{P_t:t\geq 0\}$ is a semigroup. We finally prove that it is Feller.

\begin{proposition}\label{prop:Feller}
Let $\Phi\in C_b(\CC^{-\al_0})$. Then, for every $t\geq 0$, $P_t\Phi\in C_b(\CC^{-\al_0})$. 
\end{proposition}

\begin{proof} It suffices to prove that the solution to \eqref{eq:Remainder_Eq_1} is continuous 
with respect to its initial condition. Fix $T>0$ and $x\in\CC^{-\alpha_0}$. 
Let $y\in \CC^{-\alpha_0}$ such that $\|x-y\|_{\CC^{-\alpha_0}}\leq1$ and
\begin{align*}
 v_t & = S(t)x -\int_0^t S(t-r) \tilde F\left(v_r, \left(\<k_black>_{0,r}\right)_{k=1}^n\right)  \dd r, \\
 u_t & = S(t)y -\int_0^t S(t-r) \tilde F\left(u_r, \left(\<k_black>_{0,r}\right)_{k=1}^n\right)  \dd r,
\end{align*}
as well as $\tau = \inf\{t>0: t^\gamma\|v_t-u_t\|_{\CC^{\bt}}> 1\}$ and
$$
M=\sup_{t\leq T} t^\gamma\|v_t\|_{\CC^{\bt}}, \quad N = \VVert{\left(\<k_black>_{0,\cdot}\right)_{k=1}^n}_{\al;\al';T}.
$$
Notice that
\begin{align*}
 \tilde F\left(v_r, \left(\<k_black>_{0,r}\right)_{k=1}^n\right) - \tilde F\left(u_r, \left(\<k_black>_{0,r}\right)_{k=1}^n\right) = 
 \sum_{k=0}^n a_k \sum_{j=0}^k \binom{k}{j} \left(u_r^k-v_r^k\right) \<k-j_black>_{0,r} 
\end{align*} 
and by Propositions \ref{prop:Heat_Smooth} and \ref{prop:Mult_Ineq_II} 
we obtain that for all $T_*\leq T\wedge\tau$
\begin{align*}
 \sup_{t\leq T_*}t^\gamma\|v_t-u_t\|_{\CC^{\bt}} 
 & \leq \sup_{t\leq T_*}t^\gamma\|v_t-u_t\|_{\CC^{\bt}}  \sum_{m=1}^n \lambda_m T_*^{\alpha_m} \\
 & + \|x-y\|_{\CC^{-\al_0}} \sum_{m=n+1}^{2n} \lambda_m T_*^{\alpha_m},
\end{align*}
where $\lambda_m\equiv\lambda_m(M,N,\|x\|_{\CC^{-\al_0}})$ and $\al_m\in(0,1]$. Choosing 
$T_*\equiv T_*(M,N,\|x\|_{\CC^{-\al_0}})\leq 1/2$ we obtain that
\begin{equation*}
 \sup_{t\leq T_*} t^\gamma\|v_t-u_t\|_{\CC^{\bt}} \leq \|x-y\|_{\CC^{-\al_0}}.
\end{equation*}
Iterating the procedure we find $N^*\in \ZZ_{\geq 0}$ and $C>0$ such that
\begin{equation*}
 \sup_{t\leq T\wedge \tau}t^\gamma\|v_t-u_t\|_{\CC^{\bt}} \leq (N^*C+1) \|x-y\|_{\CC^{-\al_0}},
\end{equation*}
for every $y\in\CC^{-\al_0}$ such that $\|x-y\|_{\CC^{-\al_0}}\leq 1$. At this point we should notice that for every $y\in\CC^{-\al_0}$ such
that $\|x-y\|_{\CC^{-\al_0}}\leq 1/2(N^*C+1)$ the above estimate implies that
\begin{equation*}
 \sup_{t\leq T\wedge \tau} t^\gamma\|v_t-u_t\|_{\CC^{\bt}} \leq \frac{1}{2},
\end{equation*}
thus $T\wedge \tau = T$ because of the definition of $\tau$. Hence, for all such $y\in\CC^{-\al_0}$,
\begin{equation*}
 \sup_{t\leq T} t^\gamma\|v_t-u_t\|_{\CC^{\bt}} \leq (N^*C+1) \|x-y\|_{\CC^{-\al_0}},
\end{equation*}
which implies convergence of $u_t$ to $v_t$ in $\CC^{\bt}$ for every $t\leq T$. Since $T$ was arbitrary, the last implies
continuity of the solution map of \eqref{eq:Remainder_Eq_1} with respect to its initial condition. The Feller property is then an 
immediate consequence of the above combined with the dominated convergence theorem. 
\end{proof}

\subsection{Invariant Measures}\label{s:Invariant_Measures}

We denote by $\{P^*_t:t\geq 0\}$ the dual semigroup of $\{P_t:t\geq 0\}$ acting on the set of all probability Borel measures on 
$\CC^{-\al_0}$ denoted by $\mathcal{M}_1(\CC^{-\al_0})$. In the next proposition we prove existence of invariant measures of $\{P_t:t\geq 0\}$ as a semigroup
acting on $C_b(\CC^{-\al_0})$. 

\begin{proposition} For every $x\in\CC^{-\al_0}$ there exists a measure $\nu_x\in\mathcal{M}_1(\CC^{-\al_0})$ and a sequence 
$t_k\nearrow\infty$ such that
$$
  \frac{1}{t_k}\int_0^{t_k} P_s^*\delta_x \, \dd s \xrightarrow{\text{w}} \nu_x.  
$$ 
In particular the measure $\nu_x$ is invariant for the Markov semigroup $\{P_t:t\geq 0\}$ on $\CC^{-\al_0}$. 
\end{proposition}

\begin{proof} For $t>0$ and $\al>0$ using Markov's and Jensen's inequality there exists a constant $C>0$ such that
\begin{align*}
 \PP(\|X(t;x)\|_{\CC^{-\al}}>K) & \leq \frac{C}{K} \left(\EE\|X(t;x)\|_{\CC^{-\al}}^p\right)^\frac{1}{p},
\end{align*}
for every $K>0$ and $p\geq 2$.  Thus
\begin{align*} 
 \int_0^t \PP(\|X(s;x)\|_{\CC^{-\al}}>K) \dd s & \leq 
 \frac{C}{K} \int_0^t \left(\EE\|X(s;x)\|_{\CC^{-\al}}^p\right)^\frac{1}{p} \dd s \\
 & \leq \frac{C}{K} \left[ \int_0^1 s^{-\frac{1}{n-1}}\dd s 
 + \int_1^t \dd s\right]\\
 & = \frac{C}{K} t
\end{align*}
where in the second inequality we use \eqref{eq:time_ind_moments}. If we let $R_t= \frac{1}{t}\int_0^t P^*_s\delta_x \dd s$, 
for $K_\ee = \frac{C}{\ee}$ we get
\begin{equation*}
R_t(\{f\in \CC^{-\alpha}: \|f\|_{\CC^{-\al}} > K_\ee\}) \leq \ee.
\end{equation*}
Choosing $\al<\al_0$ we can ensure that $\{f\in\CC^{-\al}: \|f\|_{\CC^{-\alpha}}\leq K_\ee\}$ is a compact subset of 
$\CC^{-\al_0}$ since the embedding $\CC^{-\al}\hookrightarrow\CC^{-\al_0}$ is compact for every $\al<\al_0$ (
see Proposition \ref{prop:Compact_Emb} and \eqref{eq:q_ineq}). 
This implies tightness of $\{R_t\}_{t\geq 0}$ in $\CC^{-\al_0}$ and by the Krylov--Bogoliubov existence Theorem 
(see \cite[Corollary 3.1.2]{dPZ96}) there exist a sequence $t_k\nearrow\infty$ and a measure 
$\nu_x\in\mathcal{M}_1(\CC^{-\al_0})$ such that $R_{t_k}\to \nu_x$ weakly in $\CC^{-\al_0}$ and $\nu_x$ is invariant 
for the semigroup $\{P_t:t\geq 0\}$ in $\CC^{-\al_0}$. 
\end{proof}
\tikzsetexternalprefix{./macros/strong_feller/}

\section{Strong Feller Property}\label{s:Strong_Feller_Property}

In this section we show that the Markov semigroup $\{ P_t : t \geq 0\}$ satisfies the strong Feller property.
The strong Feller property is to be expected when we deal with SPDEs where the noise forces 
every direction  in Fourier space. However, the fact that the process $X$ does not solve a self-contained equation forces
us to translate everything onto the level of the remainder $v$. The most important step is to obtain a 
Bismut--Elworthy--Li formula (see Theorem \ref{thm:Bismut_El_Li}) which captures enough information to
provide a good control of the linearization of the remainder equation. 

On the technical level, we work with a finite dimensional approximation $X^\ee$ for $X$. This choice 
and the fact that the equation is driven by white noise imply that the solution is Fr\'echet differentiable 
with respect to the (finite dimensional approximation of the) noise, so we can avoid working with Malliavin 
derivatives. This is expressed in Proposition \ref{prop:mal_differ} below, and in fact this proposition 
could even be established without splitting $X^\ee$ into $v^\ee$ and $\<1_black>^{\ee}_{0,\cdot}$. We make
strong use of the splitting in Proposition \ref{prop:cut_off_bounds} where the local solution theory 
is used to obtain deterministic bounds on $v^\ee$ and its linearization for small $t$ provided 
that we control the  diagrams   $\<n_black>^{\ee}_{0,\cdot}$. This control is uniform in $\ee$ and 
enters crucially the proof of Proposition \ref{prop:TV_approx_bound}.

From now on we fix $0<\al<\al_0$ sufficiently small. For $\ee\in(0,1)$ let $\Pi_\ee[L^2(\TT^2)]$ be the finite dimensional subspace 
of $L^2(\TT^2)$ spanned by $\{e_m\}_{|m|< \frac{1}{\ee}}$ (recall that we deal with real-valued functions and the symmetry condition 
\eqref{eq:symmetry_cond} is always valid) and denote by $\Pi_\ee$ the corresponding orthogonal projection. 
We also let $\hat\Pi_\ee$ be a linear smooth approximation taking values in $\Pi_\ee[L^2(\TT^2)]$ and having the properties 
\hyperref[it:i]{i} and \hyperref[it:ii]{ii} introduced in the discussion before Proposition \ref{prop:approx_eq}. 

Let $\Re^\ee$ be the renormalization constant defined in \eqref{eq:renorm_constant} and consider a finite 
dimensional approximation of \eqref{eq:RSQEq} given by
\begin{equation}\label{eq:ee_approx}
\left\{
\begin{array}{lcll}
\dd X^\ee(t) & = & \left(\Delta X^\ee(t) - X^\ee(t) - 
\sum_{k=0}^n a_k \hat\Pi_\ee\HH_k(X^\ee(t),\Re^\ee)\right)\dd t + \dd W_\ee(t,\cdot) \\
X^\ee(0,\cdot) & = & \hat\Pi_\ee x
\end{array}
\right.,
\end{equation}
for some initial condition $x\in\CC^{-\al_0}$. Here $W_\ee(t,z) = \sum_{|m|<\frac{1}{\ee}} \hat{W}_m(t) e_m(z)$, where 
$(\hat{W}_m)_{m\in\ZZ^2}$ is a family of complex Brownian motions such that 
$\hat{W}_{-m} = \overline{\hat{W}_m}$ and independent otherwise.  We furthermore assume that $W_\ee$ is 
defined on the same probability space $\Omega$ as $\xi$ via the identity
$$
\hat W_m(t) := \xi\left(\mathbf{1}_{[0,t]} \times e_m\right), \quad m\in\ \ZZ^2,
$$
which also makes it adapted with respect to the filtration 
$(\mathcal{F}_t)_{t\geq 0}$. It is convenient to write $W_\ee = G_\ee (\hat{W}_m)_{m\in \ZZ^2\cap[-d,d]^2}$ where 
$G_\ee:C([0,\infty);\RR^{(2d-1)^2})\to \Pi_\ee[L^2(\TT^2)]$ is such that 
\begin{equation*}
 G_\ee (\hat{W}_m)_{m\in \ZZ^2\cap[-d,d]^2}  = \sum_{|m|<\frac{1}{\ee}} \hat{W}_m e_m  
\end{equation*}
and $d=\left\lfloor \frac{1}{\ee} \right\rfloor + \frac{1}{2}$. The Cameron--Martin space of $W_\ee$ is 
given by
\begin{equation*}
\mathcal{CM} := W_0^{1,2}([0,\infty)) =\left\{w: \partial_t w\in L^2([0,\infty);\RR^{(2d-1)^2}), \, w(0) = 0\right\}.
\end{equation*}
Last, we have the identity
\begin{equation}\label{eq:stoch_int_approx}
\<1_black>^\ee_{0,t} = 
\sum_{|m|< \frac{1}{\ee}} \int_0^t \e^{-(1+4\pi^2|m|^2) (t-s)} \dd \hat{W}_m(s) \, e_m,
\end{equation}
where $\<1_black>^\ee_{0,\cdot}$ is the finite dimensional approximation defined in Section
\ref{s:Finite_Dimensional_Approximations}.

As in \eqref{eq:tilde_non_lin}, for $v\in \CC^{\bt}$ and $\underline{Z}\in\left(\CC^{-\al}\right)^n$, $\al<\bt$, we use the 
notation
\begin{equation*}
 \tilde F(v, \underline{Z}) = \sum_{k=0}^n a_k \sum_{j=0}^k \binom{k}{j} v^j Z^{(k-j)}
\end{equation*}
with the convention that $Z^{(0)}\equiv 1$ and we let
\begin{equation*}
 \tilde F'(v, \underline{Z}) = \sum_{k=1}^n k a_k \sum_{j=0}^{k-1} \binom{k-1}{j} v^j Z^{(k-1-j)}.
\end{equation*}
Formally, $\tilde F'$ stands for the derivative of $\sum_{k=0}^n a_k :X^k:$ with respect to $X$, with $:X^k:$ 
replaced by $\sum_{j=0}^k \binom{k}{j} v^j Z^{(k-j)}$.

Existence and uniqueness of local in time solutions to \eqref{eq:ee_approx} up to some random explosion time 
$\tau^*_\ee>0$ can be proven following the same method as in Section \ref{s:Solving_the_Equation}, i.e. using the ansatz $X^\ee = 
\<1_black>_{0,\cdot}^\ee + v^\ee$ and solving the PDE problem
\begin{equation}\label{eq:approx_remainder_eq}
\left\{
\begin{array}{lcll}
\partial_t v^\ee & = &\Delta v^\ee - v^\ee - \hat\Pi_\ee \tilde F(v^\ee,\underline{Z}^\ee)\\
v^\ee(0,\cdot) & = & \hat\Pi_\ee x
\end{array}
\right.,
\end{equation}
where $\underline{Z}^\ee = \left(\<k_black>^\ee_{0,\cdot}\right)_{k=1}^n$ (see Section \ref{s:Finite_Dimensional_Approximations} for definitions).

Notice that for fixed $v$, $\tilde F$ is Fr\'echet differentiable with respect to any $\underline{Z}\in\left(\CC^{-\al}\right)^n$ 
as a function taking values in $\CC^{-\al}$. Recall that $\<k_black>_{0,\cdot}^\ee = \HH_k (\<1_black>_{0,\cdot}^\ee, \Re^\ee)$, for every
$1\leq k\leq n$, so that the map
\begin{equation}\label{eq:sol_map}
 (v^\ee,\<1_black>^\ee_{0,\cdot})\mapsto S(t)\hat\Pi_\ee x - \int_0^t S(t-s) 
 \tilde F\left(v^\ee_s,\left(\HH_k(\<1_black>^\ee_{0,s},\Re^\ee)\right)_{k=1}^n\right) \dd s,
\end{equation}
for $(v^\ee,\<1_black>^\ee_{0,\cdot})\in C([0,t];\CC^{\bt})\times C([0,t];\Pi_\ee[L^2(\TT^2)])$ and $t>0$, is Fr\'echet differentiable  as
a composition of $\tilde{F}$ with a linear operator shifted by a constant, since the mapping 
$$
C([0,t];\Pi_\ee[L^2(\TT^2)])\ni \<1_black>^\ee_{0,\cdot}\mapsto \left(\HH_k(\<1_black>_{0,\cdot},\Re^\ee)\right)_{k=1}^n
\in C^{n,-\al}(0;t)
$$
is Fr\'echet differentiable for any $\al>0$, with respect to any $\VVert{\cdot}_{\al;\al';t}$, for $\al'>0$ fixed. Thus, for fixed $x\in \CC^{-\al_0}$
and $\<1_black>_{0,\cdot}^\ee\in C([0,t];\Pi_\ee[L^2(\TT^2)])$ the implicit function theorem for Banach spaces (see \cite[Theorem 4E]{Ze95})
can be applied up to time $\tau^*_\ee\equiv \tau^*_\ee(x, \<1_black>_{0,\cdot}^\ee)$ where existence of $v^\ee$ is ensured. Hence, for 
$t\in(0,\tau^*_\ee)$ there exists an open neighborhood $\mathcal{U}_{\<1_black>_{0,\cdot}^\ee}\subset C([0,t];\Pi_\ee[L^2(\TT^2)])$ of 
$\<1_black>_{0,\cdot}^\ee$ such that the 
solution map $\mathscr{T}^{\ee,x}_{t}:\mathcal{U}_{\<1_black>_{0,\cdot}^\ee} \to \CC^{\bt}$ of \eqref{eq:approx_remainder_eq} 
is Fr\'echet differentiable at $\<1_black>_{0,\cdot}^\ee$. 

 Using It\^o's formula the stochastic integrals in \eqref{eq:stoch_int_approx} can be written as
\begin{align}
 \int_0^\cdot \e^{-(1+4\pi^2|m|^2) (\cdot-s)} \dd \hat{W}_m(s) & = \label{eq:int_by_parts} \\
 \hat{W}_m(\cdot) - (1+4\pi^2|m|^2) & \int_0^\cdot \e^{-(1+4\pi^2|m|^2) (\cdot-s)} \hat{W}_m(s) \dd s. \nonumber
\end{align}
We can replace $(\hat{W}_m)_{m\in\ZZ^2\cap[-d,d]^2}$ in \eqref{eq:int_by_parts} by any $w\in C([0,t];\RR^{(2d-1)^2})$ ,
thereby obtaining a continuous linear function on $C([0,t];\RR^{(2d-1)^2})$. Thus $\<1_black>_{0,\cdot}^\ee$ as a function from $C([0,t]; \RR^{(2d-1)^2})$ to 
$C([0,t];\Pi_\ee[L^2(\TT^2)])$ is Fr\'echet differentiable. Combining all the above we finally obtain 
Fr\'echet differentiability of $v_t^\ee$ from $C([0,t];\RR^{(2d-1)^2})$ to $\CC^{\bt}$. We denote by $\DD$ the Fr\'echet derivative with
respect to elements in $C([0,t];\RR^{(2d-1)^2})$ (i.e. with respect to the noise), for $t>0$.

We let $\hat W_\ee = (W_m)_{|m|< \frac{1}{\ee}}$ and for 
$w\in C([0,t];\RR^{(2d-1)^2})$ we write
\begin{equation*}
 \int_0^t S(t-s) G_\ee \dd w(s) := 
 \sum_{|m|< \frac{1}{\ee}} \int_0^t \e^{-(1+4\pi^2|m|^2) (t-s)} \dd w_m(s).
\end{equation*}
%

In the next proposition we summarize the results of the previous discussion. 

\begin{proposition}\label{prop:mal_differ}  For fixed $x\in \CC^{-\al_0}$, $\hat W_\ee\in C([0,\infty);\RR^{(2d-1)^2})$ and 
$\<1_black>_{0,\cdot}^\ee\equiv \<1_black>_{0,\cdot}^\ee(\hat W_\ee)\in C([0,\infty)];\Pi_\ee[L^2(\TT^2)])$,
let $\tau^*_\ee \equiv \tau^*_\ee(x,\<1_black>_{0,\cdot}^\ee)>0$ be the explosion time of $v^\ee$. 
Then for all  $t< \tau^*_\ee $ there exists an open neighborhood $\mathcal{O}_{\hat W_\ee}\subset C([0,t];\RR^{(2d-1)^2})$ 
of $\hat W_\ee$ such that $\mathscr{X}^{\ee,x}_t\equiv X^\ee(t;x)(= \<1_black>_{0,t}^\ee +v^\ee_t)$ is Fr\'echet differentiable as a function from
$\mathcal{O}_{\hat W_\ee}$ to $\CC^{-\al_0}$ and for any $w\in C([0,t];\RR^{(2d-1)^2})$ its directional derivative 
$\DD\mathscr{X}^{\ee,x}_t(w)$ is given in mild form as
\begin{equation}\label{eq:mal_der}
\DD\mathscr{X}^{\ee,x}_t(w) = -\int_0^t S(t-s) \hat \Pi_\ee\left[\tilde F'(v^\ee_s,\underline{Z}_s^\ee)
\DD\mathscr{X}^{\ee,x}_s(w)\right]\dd s + \int_0^t S(t-s) G_\ee \dd w(s).
\end{equation}
%
%
%
%
\end{proposition}

We denote by $D$ the Fr\'echet derivative with respect to elements in $\CC^{-\al_0}$ (i.e. with respect to the initial condition).
For $h\in\CC^{-\al_0}$, we let $h_\ee = \hat\Pi_\ee h$ and for $t\geq s$ we also consider the following linear equation,
\begin{equation} \label{eq:pert_eq}
\left\{
\begin{array}{lcll}
\partial_t J_{s,t}^\ee h_\ee & = &\Delta J_{s,t}^\ee h_\ee - J_{s,t}^\ee h_\ee - 
\hat\Pi_\ee\left[ \tilde F'(v_t^\ee,\underline{Z}_t^\ee) J_{s,t}^\ee h_\ee\right] \\
J_{s,s}^\ee h_\ee & = & h_\ee  
\end{array}
\right..
\end{equation}
Then $J^\ee_{0,t}h_\ee = DX^\ee(t;x)(h)$, i.e. it is the derivative of $X^\ee(t;\cdot)$ in the direction $h$, and its 
existence for every $t\leq \tau^*_\ee$ is ensured by a similar argument as the one discussed before Proposition   
\ref{prop:mal_differ}.

At this point we should comment on the relation between \eqref{eq:mal_der} and \eqref{eq:pert_eq}. Given that
\eqref{eq:pert_eq} has a unique solution for every $h_\ee\in\Pi_\ee[L^2(\TT^2)]$ up to time $t>0$, then 
for $w\in\mathcal{CM}$, i.e. $w(0) = 0$ and $\partial_t w\in L^2([0,\infty);\RR^{(2d-1)^2})$, by Duhamel's principle
\begin{equation} \label{eq:mal_der_op}
 \DD\mathscr{X}^{\ee,x}_t(w) = \int_0^t J^\ee_{s,t}G_\ee\partial_s w(s)\dd s,
\end{equation}
where $J^\ee_{s,t} : \CC^{-\al_0} \to \CC^{\bt}$ is the solution map of \eqref{eq:pert_eq}.

\begin{remark} In the framework of Malliavin calculus $\mathcal{D}_s\mathscr{X}^{\ee,x}_t = J^\ee_{s,t}G_\ee$
as an element of the dual of $L^2([0,\infty);\RR^{(2d-1)^2})$ is the Malliavin derivative (see \cite[Section 1.2]{Nu06}) 
in the sense that the latter coincides with the former when it acts on $\mathscr{X}^{\ee,x}_t$. In our case, the presence of 
additive noise implies Fr\'echet differentiability with respect to the noise as an element in $C([0,t];\RR^{(2d-1)^2})$
(see Proposition \ref{prop:mal_differ}), which is of course stronger than Malliavin differentiability with respect to the noise.
\end{remark} 

For $r\in[\frac{1}{4},1]$  (the precise value of $r$ will be fixed below) and $0<\al'<\al$ we consider the stopping times
\begin{equation}\label{eq:stopping_time}
\begin{array}{lcl}
\tau^{\ee,r} &  : = & \inf\left\{t>0: \|\<1_black>_{0,t}^\ee\|_{\CC^{-\al}}\vee t^{\al'}\|\<2_black>_{0,t}^\ee\|_{\CC^{-\al}} 
\vee \ldots \vee t^{\al'(n-1)}\|\<n_black>_{0,t}^\ee\|_{\CC^{-\al}} > r \right\} \\
\tau^r & : = & \inf\left\{t>0: \|\<1_black>_{0,t}\|_{\CC^{-\al}}\vee t^{\al'}\|\<2_black>_{0,t}\|_{\CC^{-\al}} 
\vee \ldots \vee t^{(n-1)\al'}\|\<n_black>_{0,t}\|_{\CC^{-\al}} > r \right\}
\end{array}.
\end{equation}
 Let $\bar B_1(x)$ be the closed unit ball centered at $x$ in $\CC^{-\al_0}$. The next proposition provides local bounds on $v^\ee$ and
$J^\ee_{0,\cdot}$ given deterministic control on $\underline{Z}^\ee$ (see also Theorem \ref{thm:Local_Ex_Un}).

\begin{proposition} \label{prop:cut_off_bounds} Let $x\in\CC^{-\al_0}$ and let $R = 2\|x\|_{\CC^{-\al_0}} +1$. Then there 
exists a deterministic time $T^*\equiv T^*(R) >0$, independent of $\ee$, such that for all $t\leq T^*\wedge \tau^{\ee,r}$ and initial 
conditions $y\in \bar{B}_1(x)$, 
\begin{equation*}
\sup_{s \leq t} s^\gamma \|v^\ee_s\|_{\CC^{\bt}}\leq 1 \quad \text{and} \quad 
\sup_{s \leq t} s^\gamma \|J^\ee_{0,t}h_\ee\|_{\CC^{\bt}} \leq 2 \|h_\ee\|_{\CC^{-\al_0}},
\end{equation*}
for $\bt,\gamma$ as in \eqref{eq:beta_gamma_cond}, uniformly in $\ee$, for every 
$h_\ee \in \hat\Pi_\ee[L^2(\TT^2)]$.  
\end{proposition}

\begin{proof} Let $t\leq \tau^{\ee,r}\wedge T^*$ where $T^*\equiv T^*(R)$ is defined as in \eqref{eq:T^*_form}. We can also 
assume that $t\leq 1$. Then, from Theorem \ref{thm:Local_Ex_Un}, we have that 
\begin{equation*}
 \sup_{s\leq t} s^\gamma \|v^\ee_s\|_{\CC^{\bt}} \leq 1,
\end{equation*}
for every $y\in \bar B_1(x)$. Furthermore, for every $\al>0$,
\begin{equation}\label{eq:pol_bound}
 \|\tilde F'(v^\ee_s,\underline{Z}^\ee_s) \|_{\CC^{-\al}} 
 \leq C s^{-(n-1)\gamma}
\end{equation}
where $C$ is a constant independent of $\ee$. Using Proposition 
\ref{prop:Heat_Smooth}, \eqref{eq:p_ineq} and \eqref{eq:a_q_ineq} we get that
\begin{align*}
 \|S(t-s)\hat\Pi_\ee[\tilde F'(v^\ee_s, \underline{Z}^\ee_s)J^\ee_{0,s}h_\ee] \|_{\CC^\bt} & \leq C (t-s)^{-\frac{\bt+\al}{2}} 
 \|\hat\Pi_\ee [\tilde F'(v^\ee_s,\underline{Z}^\ee_s)J^\ee_{0,s}h_\ee]\|_{\CC^{-\al}}\\
 & \leq C (t-s)^{-\frac{\bt+\al}{2}} s^{-(n-1)\gamma} \|J^\ee_{0,s}h_\ee\|_{\CC^{\beta}},
\end{align*}
where we also use the fact that $\|\hat\Pi_\ee f\|_{\CC^{-\al}} \lesssim \|f\|_{\CC^{-\al}}$, for every 
$f\in\CC^{-\al}$. We are now ready to retrieve the appropriate bounds on the operator norm of $J^\ee_{0,\cdot}$. For 
$h_\ee\in \Pi_\ee [L^2(\TT^2)]$ we have in mild form,
\begin{equation*}
 J^\ee_{0,t} h_\ee = S(t)h_\ee 
 - \int_0^s S(t-s) \hat\Pi_\ee\left[\tilde F'(v_s^\ee,\underline{Z}^\ee_s)J^\ee_{0,s}h_\ee\right] \dd s.
\end{equation*}
Thus for every $\al>0$ and $s\leq t\leq \tau^{\ee,r}\wedge T^*$ by \eqref{eq:pol_bound} 
\begin{equation*}
 \|J^\ee_{0,s}h_\ee\|_{\CC^{\bt}} \leq C s^{-\frac{\bt+\al_0}{2}} \|h_\ee\|_{\CC^{-\al_0}} 
 + C s^{1-\frac{\bt+\al}{2} - n \gamma} \sup_{s\leq t} s^\gamma \|J^\ee_{0,s}h_\ee\|_{\CC^\bt}. 
\end{equation*}
Multiplying the above inequality by $s^\gamma$ and using the fact that 
$\gamma - \frac{\bt+\al_0}{2}>0$ we get
\begin{equation*}
 \sup_{s\leq t} s^\gamma\|J^\ee_{0,s}h_\ee\|_{\CC^{\bt}} \leq C \|h_\ee\|_{\CC^{-\al_0}} 
 + C t^{1-\frac{\bt+\al}{2} - (n-1) \gamma}  \sup_{s\leq t} s^\gamma \|J^\ee_{0,s}h_\ee\|_{\CC^{\bt}}.
\end{equation*}
Possibly increasing the value of the constant $C$ in \eqref{eq:T^*_form} we finally obtain the bound
\begin{equation}
 \sup_{s\leq t} s^\gamma \|J^\ee_{0,s}h_\ee\|_{\CC^\bt} \leq 2 \|h_\ee\|_{\CC^{\al_0}},
\end{equation}
which completes the proof.
\end{proof}

We denote by $C_b^1(\CC^{-\al_0})$ the set of continuously differentiable functions on $\CC^{-\al_0}$. 
We furthermore let $\chi\in C^\infty(\RR)$ such that $\chi(\zeta)\in[0,1]$, for every $\zeta\in\RR$, and
\begin{equation*}
 \chi(\zeta) = 
 \begin{cases} 1, & \text{ if } |\zeta|\leq \frac{r}{2}\\
 0, & \text{ if } |\zeta|\geq r
 \end{cases},
\end{equation*}
for $r$ as in \eqref{eq:stopping_time}. For simplicity we also let $\VVert{\cdot}_t :=\VVert{\cdot}_{\al;\al';t}$, $t\geq 0$. 
Inspired by \cite{No86}, we prove the following version of the Bismut--Elworthy--Li formula.

\begin{theorem}[Bismut--Elworthy--Li Formula]\label{thm:Bismut_El_Li} Let $x\in\CC^{-\al_0}$, $\Phi\in C_b^1(\CC^{-\al_0})$ and let $t>0$. 
Let  $w$ be a process taking values in the Cameron-Martin space  $ \mathcal{CM}$ with $\partial_s w$ adapted. Furthermore, assume that
 there exists a deterministic constant $C\equiv C(t)>0$ such that $\|\partial_s w\|_{L^2([0,t];\RR^{(2d-1)^2})}^2\leq C$ $\PP$-almost 
surely. Then we have that
\begin{align} 
 \EE \left[D\Phi(\mathscr{X}^{\ee,x}_t)\left(\DD\mathscr{X}^{\ee,x}_t(w)\right)\chi(\VVert{\underline{Z}^\ee}_{t})\right] 
 & = \EE\left(\Phi(\mathscr{X}^{\ee,x}_t) \int_0^t \partial_s w(s) \cdot \dd \hat{W}_\ee(s) \, \chi(\VVert{\underline{Z}^\ee}_{t}) \right) 
 \label{eq:Bismut_El_Li} \\
 & - \EE\Big(\Phi(\mathscr{X}^{\ee,x}_t) \, \partial_{+} \chi(\VVert{\underline{Z}^\ee}_{t})(w)\Big) \nonumber
\end{align}
where 
\begin{align}\label{eq:sub_dif} 
 \partial_{+} \chi(\VVert{\underline{Z}^\ee}_{t})(w) = \partial_\zeta\chi\left(\VVert{\underline{Z}^\ee}_{t}\right) 
 \partial_+\VVert{\underline{Z}^\ee}_{t} & \Big(0,2 \<1_black>_{0,\cdot}^\ee Q_w(\cdot), \ldots, 
 n\<n-1_black>_{0,\cdot}^\ee Q_w(\cdot)\Big)  
\end{align}
$\partial_+\VVert{\cdot}_{t}:C^{n,-\al}(0;t)\to C^{n,-\al}(0;t)^*$ is the one-sided 
derivative of $\VVert{\cdot}_{t}$ given by
\begin{equation*}
 \partial_+\VVert{\underline{Z}^\ee}_{t}(\underline{Y}) = \lim_{\delta\to 0^+} \frac{\VVert{\underline{Z}^\ee + \delta 
 \underline{Y}}_{t} - 
 \VVert{\underline{Z}^\ee}_{t}}{\delta},
\end{equation*}
for every direction $\underline{Y}\in C^{n,-\al}(0;t)$,  and  
$$
Q_w(\cdot) := \int_0^{\cdot} S(\cdot - s) G_\ee \partial_s w(s) \dd s.
$$
\end{theorem}

\begin{remark}  The presence of $\partial_+\VVert{\cdot}_{t}$ in the theorem above is 
based on the fact that norms are not in general Fr\'echet differentiable functions. 
However, their one-sided derivatives always exist (see \cite[Appendix D]{dPZ92}) and 
they behave nicely in terms of the usual rules of differentiation.
\end{remark}

\begin{proof} Let $\delta>0$ and $u=\partial_t w$, which is an $L^2([0,\infty);\RR^{(2d-1)^2})$ function. For every $n\geq 1$, 
we define the shift $T_{\delta u}$ by 
\begin{align*}
 T_{\delta u}\<n_black>^{\ee}_{0,t} & = \sum_{k=0}^n \binom{n}{k}  
 \Big( \delta Q_w(t) \Big)^k \<n-k_black>_{0,t}^\ee
\end{align*}
and we let $T_{\delta u}\underline{Z}^\ee = \left(T_{\delta u} \<k_black>^\ee_{0,\cdot}\right)_{k=1}^n$.

Let $X^{\ee,\delta}(\cdot;x) = T_{\delta u}\<1_black>^\ee_{0,\cdot} + v^{\ee,\delta}$, where the remainder
$v^{\ee,\delta}$ solves the equation
\begin{equation*}
\left\{
\begin{array}{lcll}
\partial_t v^{\ee,\delta} & = & \Delta v^{\ee,\delta} - v^{\ee,\delta} 
- \hat\Pi_\ee \tilde F(v^{\ee,\delta},T_{\delta u}\underline{Z}^\ee) \\
v^{\ee,\delta}(0,\cdot) & = & \hat\Pi_\ee x
\end{array}
\right..
\end{equation*}
As in \cite{No86}, our aim is to construct a probability measure $\PP^{\delta}$ such that the law of 
$T_{\delta u}\<1_black>^{\ee}_{0,\cdot}$ under $\PP^{\delta}$ is the same as the law of $\<1_black>^\ee_{0,\cdot}$ under $\PP$. 
That way we obtain the identity 
\begin{equation} \label{eq:shifted_exp}
  \partial_{\delta^+} \EE_{\PP^{\delta}}\left(\Phi\left(X^{\ee,\delta}(t;x)\right)
  \chi(\VVert{T_{\delta u}\underline{Z}^\ee}_{t})\right)\Big|_{\delta = 0} = 0,
\end{equation}
since $\<k_black>_{0,\cdot}^\ee$ is a continuous function of $\<1_black>_{0,\cdot}^\ee$ for every $k\geq 2$,
the solution map to \eqref{eq:approx_remainder_eq} is a continuous functions of the  $\<k_black>_{0,\cdot}^\ee$ ,
and $\chi$ is a continuous function of $\<1_black>^\ee_{0,\cdot}$. Above 
$\partial_{\delta^+}$ stands as a shortcut of the directional derivative of a function as $\delta\to 0^+$. 
We will then show below that the result follows by an expansion of the derivative in the above expression.

We start with the construction of $\PP^{\delta}$. Let $\mathsf{Y}^{\delta}(r):=-\int_0^r \delta u(s) \cdot \dd \hat{W}_\ee(s)$ where
$\cdot$ is the scalar product on $\RR^{(2d-1)^2}$, and define the exponential process 
\begin{equation*}
 \mathsf{Z}^{\delta}(r): = \exp\left\{ \mathsf{Y}^{\delta}(r) - \frac{1}{2} \int_0^r |\delta u(s)|^2 \dd s\right\}.
\end{equation*}
Notice that by the assumptions on $w$ Novikov's condition is satisfied, i.e.
\begin{equation*}
 \EE \exp{\frac{1}{2} \int_0^t |\delta u(s)|^2 \dd s} <\infty,
\end{equation*}
thus by \cite[Chapter 8, Proposition 1.15]{RY99} $\mathsf{Z}^{\delta}$ is a strictly positive martingale 
and we have that $\EE \mathsf{Z}^{\delta}(t) = 1$. We define $\PP^\delta$ by 
its Radon--Nikodym derivative with respect to $\PP$
\begin{equation*}
 \frac{\dd \, \PP^{\delta}}{\dd \, \PP} = \mathsf{Z}^{\delta}(t). 
\end{equation*}
By Girsanov's Theorem (see \cite[Chapter 4, Theorem 1.4]{RY99}) we have that
$\hat{W}_\ee^{\delta}(r) : = \hat{W}_\ee(r) - \left[\hat{W}_\ee(\cdot),\mathsf{Y}^{\delta}(\cdot)\right]_r$, $r\leq t$, 
under $\PP^{\delta}$ has the same law as $\hat{W}_{\ee}$ under $\PP$, where $\left[\, \cdot \, , \cdot \,\right]_r$ stands for the 
quadratic variation at time $r$. We furthermore have that 
$\left[\hat{W}_\ee(\cdot),\mathsf{Y}^{\delta}(\cdot)\right]_r = - \int_0^r \delta u(s) \dd s$ 
as well as $\<1_black>^\ee_{0,t} = \int_0^t S(t-s)G_\ee\dd \hat{W}_\ee(s)$ and 
$T_{\delta u}\<1_black>^\ee_{0,t} = \int_0^t S(t-s) G_\ee \dd \hat{W}^{\delta}_\ee(s)$. Since the law of 
$\hat{W}^{\delta}_\ee$ under $\PP^{\delta}$ is the same as the law of $\hat{W}^\ee$ under $\PP$, this is also the case for 
$T_{\delta u}\<1_black>^\ee_{0,\cdot}$ and $\<1_black>^\ee_{0,\cdot}$ (recall that $\<1_black>_{0,\cdot}^\ee$ is a continuous 
function of $\hat W_\ee$, when the later is seen as an element in $C([0,t];\RR^{(2d-1)^2})$ endowed with the 
supremum norm because of \eqref{eq:int_by_parts}).  Thus $\PP^{\delta}$ is the required measure and \eqref{eq:shifted_exp} in the form
\begin{align}\label{eq:shifted_exp_eq0}
 \partial_{\delta^+} & \EE\left(\Phi\left(X^{\ee,\delta}(t;x)\right)\chi(\VVert{T_{\delta u}\underline{Z}^\ee}_{t})
 \mathsf{Z}^{\delta}(t)\right)
 \Big|_{\delta=0} = 0
\end{align}
follows. Using the chain rule, $\partial_\delta\Phi\left(X^{\ee,\delta}(x;t)\right) 
= D\Phi\left(X^{\ee,\delta}(x;t)\right)\left(\partial_\delta X^{\ee,\delta}(x;t)\right)$ and
$\partial_\delta \mathsf{Z}^{\delta}(t) = -\mathsf{Z}^{\delta}(t) \left(\int_0^t u(s)\cdot \dd \hat{W}_\ee(s) +
\delta \int_0^t |u(s)|^2 \dd s\right)$. 
For the directional derivative of $\chi(\VVert{T_{\delta u}\underline{Z}^\ee}_{t})$ at $\delta^+=0$ it suffices 
to check the existence of the limit 
\begin{equation*}
 \lim_{\delta\to 0^+} \frac{\VVert{T_{\delta u}\underline{Z}^\ee}_{t}-\VVert{\underline{Z}^\ee}_{t}}{\delta}.
\end{equation*}
We claim that the above limit is the same as 
\begin{equation*}
 \partial_+\VVert{\underline{Z}^\ee}_{t}(\underline{Y}^\ee) : = 
 \lim_{\delta\to 0^+} \frac{\VVert{\underline{Z}^\ee+\delta\underline{Y}^\ee}_{t}-
 \VVert{\underline{Z}^\ee}_{t}}{\delta}
\end{equation*}
where $\underline{Y}^\ee = \left(0,2 \<1_black>_{0,\cdot}^\ee Q_w(\cdot),
\ldots, n\<n-1_black>_{0,\cdot}^\ee Q_w(\cdot)\right)$. Using the fact that 
$\VVert{\cdot}_{t}$ is a norm, we have that
\begin{equation*}
\frac{\VVert{T_{\delta u}\underline{Z}^\ee}_{t} - \VVert{\underline{Z}^\ee}_{t}}{\delta} = 
\frac{\VVert{\underline{Z}^\ee+\delta \underline{Y}^\ee} - \VVert{\underline{Z}^\ee}_{t}}{\delta} 
+ \mathsf{Error}_\delta,
\end{equation*}
where $\mathsf{Error}_\delta \to 0$  
%
%
%
%
as $\delta \to 0^+$. Subtracting $\partial_+\VVert{\underline{Z}^\ee}_{t}(\underline{Y}^\ee)$ from both sides 
of the above equation and letting $\delta\to 0^+$ we get
\begin{equation}\label{eq:limsup}
 \limsup_{\delta\to 0^+} \left(\frac{\VVert{T_{\delta u}\underline{Z}^\ee}_{t} - 
 \VVert{\underline{Z}^\ee}_{t}}{\delta} -
 \partial_+\VVert{\underline{Z}^\ee}_{t}(\underline{Y}^\ee)\right) \leq 0.
\end{equation}
In a similar way we can prove that the reverse inequality of \eqref{eq:limsup} is valid with the $\limsup$ replaced by a $\liminf$, which
makes $\partial_+\VVert{\underline{Z}^\ee}_{t}(\underline{Y}^\ee)$ the appropriate limit.

Using the bounds in Proposition \ref{prop:cut_off_bounds}, by the dominated convergence theorem we can pass the derivative inside the expectation in
\eqref{eq:shifted_exp_eq0} (see also \cite{No86} for more details) and integrate by parts to obtain the identity 
\begin{align*}
 \EE\Big(D\Phi\left(X^{\ee,\delta}(t;x)\right)\left(\partial_\delta X^{\ee,\delta}(t;x)\right) 
 & \chi(\VVert{T_{\delta u} \underline{Z}^\ee}_{t}) \mathsf{Z}^{\delta}(t) \Big)\Big|_{\delta = 0}
   = \\ 
 - & \EE \Big(\Phi  \left(X^{\ee,\delta}(t;x)\right) 
 \chi(\VVert{T_{\delta u} \underline{Z}^\ee}_{t}) \partial_\delta \mathsf{Z}^{\delta}(t)\Big)\Big|_{\delta = 0} \\
 - & \EE \Big(\Phi  \left(X^{\ee,\delta}(t;x)\right)  \partial_{\delta^+}
 \chi(\VVert{T_{\delta u}\underline{Z}^\ee}_t)(w)\mathsf{Z}^{\delta}(t) \Big) \Big|_{\delta^+ = 0}.
\end{align*}
Since $\partial_\delta X^{\ee,\delta}(x;t)\Big|_{\delta = 0} = D\mathscr{X}^{\ee,x}_t(\underline{Y}^\ee)$
and $\partial_\delta \mathsf{Z}^{\delta}(t)\Big|_{\delta = 0} = - \int_0^t u(s)\cdot \dd \hat{W}_\ee(s)$ we get 
\eqref{eq:Bismut_El_Li} which completes the proof.
\end{proof}

Let $\{P^\ee_t:t\geq 0\}$ defined via the identity
\begin{equation*}
 P^\ee_t\Phi(x) := \EE\Phi(X^\ee(t;x))\mathbf{1}_{\{t< \tau^*_\ee(x)\}}
\end{equation*}
for every $\Phi\in C_b(\CC^{-\al_0})$, where we write $\tau^*_\ee(x)$ dropping the dependence of $\tau^*_\ee$ on $\<1_black>_{0,\cdot}^\ee$.
We use \eqref{eq:Bismut_El_Li} to prove the following proposition.

\begin{proposition} \label{prop:TV_approx_bound}  There exist a 
universal constant $C$ and $\theta_1>0$ such that  
\begin{equation}\label{eq:TV_bound}
 |P^\ee_t\Phi(x) - P^\ee_t\Phi(y)| \leq C \frac{1}{t^{\theta_1}} \|\Phi\|_\infty \|x-y\|_{\CC^{-\al}} 
 + 2 \|\Phi\|_\infty \PP(t\geq \tau^{\ee,\frac{r}{2}})
\end{equation}
for every $x\in\CC^{-\al_0}$, $y\in \bar{B}_1(x)$, $\Phi\in C^1_b(\CC^{-\al_0})$ and 
$t\leq T^*\equiv T^*(R)$ (defined in  Proposition \ref{prop:cut_off_bounds}), where $R=2\|x\|_{\CC^{-\al_0}} +1$.
\end{proposition}

\begin{proof} Let $\Phi\in\ C^1_b(\CC^{-\al})$ and $t\leq T^*$. Then
\begin{align*}
 |P_t^\ee\Phi(x) - P_t^\ee\Phi(y)|  = 
 \left|\EE\left[\Phi\left(X^\ee(t;x)\right)\mathbf{1}_{\{t< \tau^*_\ee(x)\}} - \Phi\left(X^\ee(t;y)\right)\mathbf{1}_{\{t< \tau^*_\ee(y)\}}\right]\right|
\end{align*}
and the latter term is bounded by $I_1+I_2$, where
\begin{align*}
 I_1 := & \left|\EE\left[\Big(\Phi\left(X^\ee(t;x)\right) - \Phi\left(X^\ee(t;y)\right)\Big)\chi(\VVert{\underline{Z}^\ee}_{t})\right]\right|, \\
 I_2 := & \left|\EE\left[\Big(\Phi\left(X^\ee(t;x)\right)\mathbf{1}_{\{t< \tau^*_\ee(x)\}} 
 - \Phi\left(X^\ee(t;y)\right)\mathbf{1}_{\{t< \tau^*_\ee(y)\}}\Big) 
 \left(1-\chi(\VVert{\underline{Z}^\ee}_{t})\right)\right]\right|.
\end{align*}
For the second term we have that $I_2\leq 2 \|\Phi\|_\infty \PP(t\geq \tau^{\ee,\frac{r}{2}})$ while by the mean 
value theorem we get that
\begin{align*}
 I_1 & = \left|\EE \left(\int_0^1 D\Phi\left(\mathscr{X}^{\ee,x+\lambda (y-x)}_t\right)(y-x) \dd \lambda 
 \, \chi(\VVert{\underline{Z}^\ee}_{t})\right)\right| \\
 & = \left|\int_0^1 \EE\left(D\Phi\left(\mathscr{X}^{\ee,x+\lambda (y-x)}_t\right)(y-x)
 \, \chi(\VVert{\underline{Z}^\ee}_{t})\right)\dd \lambda\right|.
\end{align*}
 For any $h_\ee \in \Pi_\ee[L^2(\TT^2)]$ let $w$ be such that $\partial_s w(s) = \left(\lng J^\ee_{0,s} h_\ee, e_m\rng\right)_{|m|<\frac{1}{\ee}}$ 
for $s\leq \tau^{\ee,r}$ and $0$ otherwise. Then $\partial_s w$ is an adapted process and by Proposition \ref{prop:cut_off_bounds} there exists $C\equiv 
C(t)>0$ such that $\|\partial_s w\|_{L^2([0,t];\RR^{(2d-1)^2})}^2\leq C$, $\PP$-almost surely, for every initial condition $z_\lambda = 
x + \lambda (y-x)$ (recall that $J^\ee_{0,\cdot}$ depends on the initial condition and that $z_\lambda\in \bar B_1(x)$, for every $\lambda \in [0,1]$, 
thus the estimates in Proposition \ref{prop:cut_off_bounds} hold uniformly in $\lambda$). Furthermore, 
$\DD\mathscr{X}_t^{\ee,z_\lambda}(w) = t D\mathscr{X}_t^{\ee,z_\lambda}(h_\ee)$, for every $t\leq \tau^{\ee,r}$, and as in 
\cite{No86} we can use \eqref{eq:Bismut_El_Li} for this particular choice of $w$ to obtain the following identity,
\begin{align*}
 \EE\left(D\left[\Phi(\mathscr{X}^{\ee,z_\lambda}_t)\right](h_\ee)\chi(\VVert{\underline{Z}^\ee}_{t})\right) & =
 \frac{1}{t} \EE\left(\Phi(\mathscr{X}^{\ee,z_\lambda}_t) \int_0^t \lng J^\ee_{0,s}h_\ee, \dd W_\ee(s)\rng 
 \chi(\VVert{\underline{Z}^\ee}_{t})\right)\\
 & - \frac{1}{t} \EE\left(\Phi(\mathscr{X}^{\ee,z_\lambda}_t) \partial_+\chi(\VVert{\underline{Z}^\ee}_{t})(J^\ee_{0,\cdot}h_\ee)
 \right),
\end{align*}
where we slightly abuse the notation since, as we already mentioned, the operator $J^\ee_{0,\cdot}$ depends on the initial condition 
$z_\lambda$. In particular this is true for $h_\ee = \hat \Pi_\ee (y-x)$, hence
\begin{align*}
 I_1 & \leq \frac{1}{t} \|\Phi\|_\infty \int_0^1 \EE\left| \int_0^t \left\lng J^\ee_{0,s}\hat \Pi_\ee(y-x), \dd W_\ee(s)\right\rng 
 \,\chi(\VVert{\underline{Z}^\ee}_{t})  \right| \dd\lambda \\
 & + \frac{1}{t} \|\Phi\|_{\infty} \int_0^1
 \EE\left|\partial_+\chi(\VVert{\underline{Z}^\ee}_{t})\left(J^\ee_{0,\cdot}\hat \Pi_\ee(y-x)\right)\right|\dd \lambda.
\end{align*}
 Estimating the first term above we get
\begin{align*}
 \EE\left| \int_0^t \left\lng J^\ee_{0,s}\hat \Pi_\ee(y-x), \dd W_\ee(s)\right\rng \,\chi(\VVert{\underline{Z}^\ee}_{t})\right| 
 & \leq \EE\left| \int_0^{t\wedge\tau^{\ee,r}} \left\lng J^\ee_{0,s}\hat \Pi_\ee(y-x), \dd W_\ee(s)\right\rng \right| \\
  & \leq \left(\EE\int_0^{t\wedge\tau^\ee} \|J^\ee_{0,s}(y-x)\|_{L^2}^2 \dd s\right)^\frac{1}{2} \\
 & \leq C t^{\frac{1}{2}- \gamma} \|x-y\|_{\CC^{-\al_0}},
\end{align*}
where  we use a Cauchy-Schwarz inequality and It\^o's isometry in the second step and Proposition 
\ref{prop:cut_off_bounds}  in the third step. Here we use crucially, that the deterministic bound on $J^{\ee}_{0,s}$ provided in 
Proposition \ref{prop:cut_off_bounds} holds uniformly in $\ee>0$ (and in $\lambda$).
 Using the explicit
form \eqref{eq:sub_dif} of $\partial_+\chi(\VVert{\underline{Z}^\ee}_{t})$ we also have the uniform in $\lambda$ bound
\begin{equation*}
 \EE\left|\partial_+\chi(\VVert{\underline{Z}^\ee}_{t})\left(J^\ee_{0,\cdot}\hat \Pi_\ee(y-x)\right)\right| 
 \leq C t^{1-\gamma} \|x-y\|_{\CC^{-\al_0}},
\end{equation*}
since 
$$
\partial_+\VVert{\underline{Z}^\ee}_t\left(0,2 \<1_black>_{0,\cdot}^\ee Q_w(\cdot), \ldots, 
n\<n-1_black>_{0,\cdot}^\ee Q_w(\cdot)\right) \leq C \VVert{\underline{Z}^\ee}_t t^{1-\gamma} \|x-y\|_{\CC^{-\al_0}}
$$
and the fact that $\VVert{\underline{Z}^\ee}_t$ multiplied by $\partial_\zeta\chi(\VVert{\underline{Z}^\ee}_t)$ is bounded by $1$.
Thus 
\begin{equation*}
 I_1 \leq C \frac{1}{t^\gamma} \|\Phi\|_\infty \|x-y\|_{\CC^{-\al_0}} 
\end{equation*}
and using both the bounds on $I_1$ and $I_2$ we get that for every $t\leq T^*$ 
\begin{equation*}
 |P^\ee_t(x) -P_t^\ee(y)| \leq C \frac{1}{t^\gamma} \|\Phi\|_\infty \|x-y\|_{\CC^{-\al_0}} 
 + 2 \|\Phi\|_\infty \PP(t\geq \tau^{\ee,\frac{r}{2}}),
\end{equation*}
which completes the proof.  
\end{proof}

Given that the vector $\left(\<k_black>_{0,\cdot}^\ee\right)_{k=1}^n$ converges in law to 
$\left(\<k_black>_{0,\cdot}\right)_{k=1}^n$ on $C^{n,-\al}(0;T)$, for every $\al>0$ and with respect to every norm
$\VVert{\cdot}_{\al;\al';T}$, for every $T>0$, we have that $\tau^{\ee,\frac{r}{2}}$ converges in law to $\tau^{\frac{r}{2}}$ 
when the mapping 
\begin{equation}\label{eq:tau_mapping}
  \underline{Z} \mapsto 
  \inf\left\{t>0: \|\underline{Z}^{(1)}_{t}\|_{\CC^{-\al}}\vee t^{\al'}\|\underline{Z}^{(2)}_{t}\|_{\CC^{-\al}} 
  \vee\ldots \vee t^{(n-1)\al'}\|\underline{Z}^{(n)}_{t}\|_{\CC^{-\al}} > \frac{r}{2}\right\} 
\end{equation}
is $\PP$-almost surely continuous on the path $\left(\<k_black>_{0,\cdot}\right)_{k=1}^n$.  But if 
$$
\mathsf{L}:= 
\left\{r\in(0,1] : \PP\left(\eqref{eq:tau_mapping} \text{ is discontinuous on }  
\left(\<k_black>_{0,\cdot}\right)_{k=1}^n\right)>0 \right\}
$$
and $M:[0,\infty) \to [0,\infty)$ is the mapping
\begin{equation*}
 t\mapsto \|\<1_black>_{0,t}\|_{\CC^{-\al}}\vee t^{\al'}\|\<2_black>_{0,t}\|_{\CC^{-\al}} 
 \vee\ldots \vee t^{(n-1)\al'}\|\<n_black>_{0,t}\|_{\CC^{-\al}},
\end{equation*}
then 
$$
\mathsf{L} \subset 
\left\{ r\in(0,1]: \PP\left(M \text{ has a local maximum at height } r\right)>0 \right\}
$$
and the latter set is at most countable (see \cite[proof of Theorem 6.1]{MWe14}), 
thus we can choose $r\in[\frac{1}{4},1]$ in \eqref{eq:stopping_time} such that \eqref{eq:tau_mapping} is $\PP$-almost surely continuous on 
$\left(\<k_black>_{0,\cdot}\right)_{k=1}^n$. This implies convergence in law of $\tau^{\ee,\frac{r}{2}}$ to $\tau^{\frac{r}{2}}$, thus 
$$
\limsup_{\ee\to 0^+} \PP(t\geq \tau^{\ee,\frac{r}{2}}) \leq \PP(t\geq \tau^{\frac{r}{2}}).
$$
Notice that global existence of $v_t$ (see Theorem \ref{thm:global_ex}) implies global existence of $X(t;x)$ and in particular 
existence for every $t\leq T^*(R)$. 
Using Propositions \ref{prop:diagram_convergence} and \ref{prop:approx_eq}, $\liminf_{\ee \to 0^+} \tau^*_\ee \geq T^*(R)$ 
and $\sup_{t\leq \tau^*_\ee\wedge T^*(R)}\|X^\ee(t;x^\ee)-X(t;x)\|_{\CC^{-\al_0}}\to 0$ $\PP$-almost surely, 
for every $x\in \CC^{-\al}$. By the dominated convergence theorem $P^\ee_t\Phi(x)$ converges to $P_t\Phi(x)$, for every 
$\Phi\in C_b(\CC^{-\al_0})$, and we retrieve \eqref{eq:TV_bound} for the limiting semigroup $P_t$, for every $t\leq T^*(R)$, in the form
\begin{equation}\label{eq:TV_bound_limit}
 |P_t\Phi(x) - P_t\Phi(y)| \leq C \frac{1}{t^{\theta_1}}\|\Phi\|_\infty \|x-y\|_{\CC^{-\al_0}} 
 + 2 \|\Phi\|_\infty \PP(t\geq \tau^{\frac{r}{2}}).
\end{equation} 
\begin{remark}\label{rem:cemetery}  The above argument can be modified   to retrieve \eqref{eq:TV_bound_limit} without the knowledge of
global existence for the limiting process. In this case,  one can define the semigroup $P_t$ by  introducing a cemetery 
state for the process $X(t;x)$. 
\end{remark}

We finally prove the following theorem. Below we denote by $\|\mu_1-\mu_2\|_{\mathrm{TV}}$ the total variation distance of 
two probability measures $\mu_1,\mu_2\in\mathcal{M}_1(\CC^{-\al_0})$ given by
\begin{equation*}
 \|\mu_1-\mu_2\|_{\mathrm{TV}} : = \frac{1}{2}\sup_{\|\Phi\|_\infty\leq 1} 
 |\EE_{\mu_1}\Phi - \EE_{\mu_2}\Phi|.
\end{equation*}

\begin{theorem} \label{thm:strong_feller} There exists $\theta\in(0,1)$ and $\sigma>0$ such that for every 
$x\in\CC^{-\al_0}$ and $y\in\bar{B}_1(x)$
\begin{equation*}
 \|P_t(x) - P_t(y)\|_{\mathrm{TV}} \leq C (1+\|x\|_{\CC^{-\al_0}})^\sigma \|x-y\|_{\CC^{-\al_0}}^\theta,
\end{equation*}
for every $t\geq 1$. In particular, for every $t\geq 1$, $P_t$ is locally uniformly $\theta$-H\"older continuous with respect to 
the total variation norm in $\CC^{-\al_0}$.  
\end{theorem}

\begin{proof} Let $R=2\|x\|_{\CC^{-\al_0}}+1$. By \cite[Section 7.1]{dPZ96}, 
\eqref{eq:TV_bound_limit} is equivalent to
\begin{equation*}
 \|P_t(x)-P_t(y)\|_{\mathrm{TV}} \leq C \frac{1}{t^{\theta_1}} \|x-y\|_{\CC^{-\al_0}} + 2 \PP(t\geq \tau^{\frac{r}{2}}),
\end{equation*}
for every $t\leq T^*$ and $y\in \bar B_1(x)$. Notice that
\begin{equation*}
 \PP(t\geq \tau^{\frac{r}{2}}) \leq \PP\left(\VVert{\underline{Z}}_{\al;\al';t} > \frac{r}{2}\right)
\end{equation*}
and by Theorem \ref{thm:mod_sol} 
\begin{equation*}
 \PP\left(\VVert{\underline{Z}}_{\al;\al';t} > r\right) \leq C \frac{1}{r} t^{\theta_2},
\end{equation*}
for some $\theta_2\in (0,1)$. Since we can assume that $T^*\leq 1$, we have that
\begin{equation*}
 \|P_1(x)- P_1(y)\|_{\mathrm{TV}} \leq \|P_{T^*}(x) - P_{T^*}(y)\|_{\mathrm{TV}} 
\end{equation*}
where 
\begin{equation*}
 \|P_{T^*}(x) - P_{T^*}(y)\|_{\mathrm{TV}} \leq \inf_{t\leq T^*} 
 \left\{ C_1 \frac{1}{t^{\theta_1}} \|x-y\|_{\CC^{-\al_0}} + C_2 \frac{1}{r} t^{\theta_2}\right\}.
\end{equation*}
Let $f(t): = C_1 \frac{1}{t^{\theta_1}} \|x-y\|_{\CC^{-\al_0}} + C_2 \frac{1}{r} t^{\theta_2}$, $t>0$, and notice that for 
$t_0 = \left(\frac{\theta_1 C_1}{\theta_2 C_2}\right)^{\frac{1}{\theta_1+\theta_2}}$, $f(t_0) = \inf_{t>0} f(t)$. If 
$t_0\leq T^*$, then there exists $C\equiv C(\theta_1,\theta_2,r)$ such that
\begin{align*}
 \|P_{T^*}(x) - P_{T^*}(y)\|_{\mathrm{TV}} & \leq f(t_0) 
 = C \|x-y\|_{\CC^{-\al_0}}^{\frac{\theta_2}{\theta_1+\theta_2}}.
\end{align*}
Otherwise $t_0\geq T^*$, which implies that
\begin{align*}
 \|P_{T^*}(x) - P_{T^*}(y)\|_{\mathrm{TV}} & \leq C_1 \frac{1}{(T^*)^{\theta_1}} \|x-y\|_{\CC^{-\al_0}} 
 + C_2 \frac{1}{r} (T^*)^{\theta_2} \\
 & \leq C_1 \frac{1}{(T^*)^{\theta_1}} \|x-y\|_{\CC^{-\al_0}} + C_2 \frac{1}{r} t_0^{\theta_2} \\
 & = C_1 \frac{1}{(T^*)^{\theta_1}} \|x-y\|_{\CC^{-\al_0}} + \tilde{C}_2 
 \frac{1}{r} \|x-y\|_{\CC^{-\al_0}}^{\frac{\theta_2}{\theta_1+\theta_2}}
\end{align*}
and using the explicit estimate of $T^*$ (see \eqref{eq:T^*_form}) we get
\begin{align*}
 \|P_{T^*}(x)-P_{T^*}(y)\|_{\mathrm{TV}} & \leq \tilde{C}_1 (1+R)^{3\gamma\theta_1} \|x-y\|_{\CC^{-\al_0}} 
 + \tilde{C}_2 \frac{1}{r} 
 \|x-y\|_{\CC^{-\al_0}}^{\frac{\theta_2}{\theta_1+\theta_2}} \\
 & \leq C (1+R)^{3\theta_0\theta_1 +\frac{\theta_1}{\theta_1+\theta_2}} \|x-y\|_{\CC^{-\al_0}}^{\frac{\theta_2}{\theta_1+\theta_2}}
\end{align*}
for a constant $C\equiv C(\theta_1,\theta_2,r)$ and some $\theta_0>0$. Combining all the above we finally get
\begin{equation*}
 \|P_1(x)- P_1(y)\|_{\mathrm{TV}} \leq C (1+R)^{3\theta_0\theta_1+\frac{\theta_1}{\theta_1+\theta_2}} 
 \|x-y\|_{\CC^{-\al_0}}^{\frac{\theta_2}{\theta_1+\theta_2}},
\end{equation*}
which completes the proof.
\end{proof}
\tikzsetexternalprefix{./macros/support/}

\section{Exponential Mixing of the \texorpdfstring{$\Phi^4_2$}{Phi42}}\label{s:Exponential_Mixing_of_the_Phi42}

From now on we restrict ourselves in the case $n=3$ (see Remark \ref{rem:general_supp_theorem}). In this section
following \cite{ChF16} we first prove a support theorem for the solution to the $\Phi^4_2$ equation. After that we combine 
this result with Corollary \ref{cor:moments_est} and Theorem \ref{thm:strong_feller} and prove exponential convergence to a 
unique invariant measure with respect to the total variation norm. 

\subsection{Another Support Theorem}\label{s:Another_Support_Theorem} 

We consider $\underline{Y} = \left(\<k_black>_{-\infty,\cdot}\right)_{k=1}^3$ as an element of $C([0,T];\CC^{-\al})^3$ endowed with the norm
$\VVert{\cdot}_{\al;0;T}$, for some $\al\in(0,1)$, given by
\begin{equation*}
 \VVert{\underline{Y}}_{\al;0;T}:=\max_{k=1,2,3} \left\{\sup_{t\leq T}\|\<k_black>_{0,t}\|_{\CC^{-\al}}\right\}.
\end{equation*}
Here we are allowed to use a non-weighted norm since there is no blow up of $\<k_black>_{-\infty,\cdot}$
at zero.  We furthermore let 
\begin{equation*}
 \mathscr{H}(T) := \left\{ h\big|_{[0,T]}: h(t) = \int_{-\infty}^t S(t-r) f(r) \dd r, \, t\geq 0, \text{ and }
 f\in L^2(\RR\times \TT^2)\right\}.
\end{equation*}
It is worth mentioning that $\mathscr{H}(T)$ consists of those $L^2$-integrable space-time functions with zero initial datum and with one derivative 
in time and two derivatives in space in $L^2$.

\begin{lemma}\label{lem:f_m} Let $\{C_m\}_{m\geq 1}$ be a sequence of positive numbers such that $C_m\leq C (m+1)$. Then there exists 
a sequence of smooth functions $\{f_m\}_{m \geq 1}$ such that
\begin{enumerate}[i.]
 \item $f_m \in\CC^{-\al}$, for every $\al\in(0,1)$.
 \item $\EE|\lng f_m, e_l\rng|^2 = C_m$ if $l=2^m$ or $l=-2^m$ and $0$ otherwise.
 \item \label{prop:iii} For every $n=1,2,3$, $H_n(f_m, C_m) \to 0$ in $\CC^{-\al}$, for every $\al\in (0,1)$.
\end{enumerate}
\end{lemma}

\begin{proof} Let 
\begin{equation*}
 f_m(z) := \frac{ \e^{2\pi \ii 2^m z_0\cdot z}+\e^{-2\pi \ii 2^m z_0\cdot z}}{2^{1/2}} C_m^{1/2},
\end{equation*}
where $z_0 = (1,1) \in \ZZ^2$, $z\in\TT^2$. Then for $\kk\geq -1$
\begin{align*}
 \delta_\kk f_m(z) & = \frac{C_m^{1/2}}{2^{1/2}} \mathbf{1}_{ \{m=\kk \}} \left(\e^{2\pi \ii 2^m z_0\cdot z}+\e^{-2\pi \ii 2^m z_0\cdot z}\right) \\
 \delta_\kk f_m(z)^2 - C_m & = \frac{C_m}{2} \mathbf{1}_{ \{m+1 =\kk \}} \left(\e^{2\pi \ii 2^{m+1} z_0\cdot z}+\e^{-2\pi \ii 2^{m+1} z_0\cdot z}\right) \\
 \delta_\kk f_m(z)^3 & = \frac{C_m^{3/2}}{2^{3/2}} 
 \Big [ \chi_\kk(2^m 3 z_0) \left(\e^{2\pi \ii 2^m 3 z_0\cdot z}+\e^{-2\pi \ii 2^m 3 z_0\cdot z}\right) \\
 & +  \mathbf{1}_{\{m=\kk\} } 3 \left(\e^{2\pi \ii 2^m z_0\cdot z}+\e^{-2\pi \ii 2^m z_0\cdot z}\right) \Big ].
\end{align*}
Notice here we have used the convenient fact that the particular choice of $z_0$ has the property that $\chi_\kk(2^m z_0) = \mathbf{1}_{\{ m=\kk\}}$. Thus  we have
\begin{align*}
\|f_m\|_{\CC^{-\al}} & \lesssim C_m^{1/2} 2^{-\al m}  \\
\|f_m^2-C_m\|_{\CC^{-\al}} & \lesssim C_m 2^{-\al m} \\
\|f_m^3- 3 C_m f_m \|_{\CC^{-\al}} & \lesssim C_m^{3/2} 2^{-\al m}.
\end{align*}
Given that $C_m \lesssim m+1$ all the above quantities tend to $0$ as $m\to\infty$, which completes the proof.
\end{proof}

\begin{remark} \label{rem:general_supp_theorem} The sequence $\{f_m\}_{m\geq 1}$ introduced in the lemma above satisfies property 
\hyperref[prop:iii]{iii} for every odd $n$. For such $n$ every term appearing in $H_n(f_m, C_m) $ is a multiple of  $C_m^{k_1} e_{2^m k_2 z_0} $  for 
a  $k_2 \neq 0$ and the fast (exponential) decay of $\| e_{2^m k_2 z_0} \|_{\CC^{-\al}}$ compensates the slow (polynomial) growth of $C_m^{k_1}$. 
However, for even $n$ this property fails, because for such $n$ the  $H_n(f_m, C_m) $ contains a multiple of $C_m^n$ which does not need to vanish.
We suspect, that a first step in order to generalize Theorem \ref{thm:noise_support} to the case of general $n$ 
would be the construction of a sequence $\{f_m\}_{m\geq 1}$ with Fourier support on an annulus and such that
$$
\int_{\TT^2} f_m(z)^k \dd z = \HH_k(0, C^m),
$$
for every $k\geq 1$. 

\end{remark}

We now prove the following support theorem.

\begin{theorem}\label{thm:noise_support} Let $\PP_{\underline{Y}}$ be the law of $\underline{Y}$ in $C([0,T];\CC^{-\al})^3$ endowed with the norm 
$\VVert{\cdot}_{\al;0;T}$. Then 
\begin{equation*}
 \supp \PP_{\underline{Y}} = 
 \overline{\left\{\Big(\HH_k(h,\Re)\Big)_{k=1}^3: h\in \mathscr{H}(T), \, \Re\geq 0\right\}}^{\VVert{\cdot}_{\al;0;T}}.
\end{equation*}
\end{theorem}

\begin{proof} For $h\in \mathscr{H}(T)$ and $\underline{Y} \in C^{3,-\al}(0;T)$  let $T_h$ be the shift
\begin{equation*}
 T_h Y^{(k)} = \sum_{j=0}^k \binom{k}{j} h^k \, Y^{(k-j)}, \quad k=1,2,3,
\end{equation*}
where we use again the convention that $Y^{(0)} \equiv 1$, and write $T_h\underline{Y} = \left(T_h Y^{(k)}\right)_{k=1}^3$. Here we slightly abuse the notation since the action of $T_h$ on $Y^{(k)}$
needs information on lower the order terms.

As in \cite{ChF16}, it suffices to prove that $(0, -\Re, 0)\in \supp \PP_{\underline{Y}}$, for every $\Re\geq 0$.
Then, given that shifts of the initial probability measure in the direction of the Cameron--Martin 
space generate equivalent probability measures, for every $h\in\mathscr{H}(T)$, $T_h(0, -\Re , 0)\in \supp \PP_{\underline{Y}}$, which completes the proof since by the definition of $T_h$ the 
latter  is equal to $\left(\HH_k(h,\Re)\right)_{k=1}^3$  (see also \cite[Corollary 3.10]{ChF16}).

For $\lambda> 0$ and $\rho_{\lambda 2^{m}}(z) = \sum_{|\bar m|< \lambda 2^m} e_{\bar m}(z)$ we let 
$$
\<1_black>_{-\infty,t}^m(z) := \<1_black>_{-\infty,t}(\rho_{\lambda 2^{m}}(z-\cdot)), \,
\Re^m := \EE\<1_black>^m_{-\infty,t}(0)^2,
$$ 
where $\<1_black>_{-\infty,t}^m$ coincides with $\<1_black>_{-\infty,t}^\ee$ in Section \ref{s:Finite_Dimensional_Approximations} for $m=\frac{1}{\ee}$. 
Notice that for $\Re\geq 0$ there exists $m_0\equiv m_0(\Re)>1$ such that $\Re^m - \Re > 0$, for every $m\geq m_0$ (recall that $\Re^m \sim \log m$). Thus 
if we set $C_m =0$ for $m\leq m_0$ and $C_m = \Re^m- \Re$ otherwise, then $C_m\geq 0$ and $C_m \lesssim m+1$. We consider $f_m$ as in Lemma \ref{lem:f_m} 
for this particular choice of $C_m$ and for $\lambda_m=1+4\pi^2 2^{2m} |z_0|^2$ we let
$$
h_m(t): =\left(1-\e^{-\lambda_m(t+1)}\right) f_m,
$$
for $t\in[0,T]$. Then $h_m\in \mathscr{H}(T)$ since $h_m(t) = \frac{1}{\lambda_m}\int_{-1}^t S(t-r)f_m \dd r$  
and we furthermore have the uniform in $t$ estimates
\begin{align*}
 \|h_m(t)\|_{\CC^{-\al}} & \leq \|f_m\|_{\CC^{-\al}}, \\
 \|h_m(t)^2 - C_m \|_{\CC^{-\al}} & \leq \|f_m^2 - C_m \|_{\CC^{-\al}}^2 
 + 2\e^{-\lambda_m } C_m , \\
 \|h_m(t)^3\|_{\CC^{-\al}} & \leq \|f_m^3\|_{\CC^{-\al}}.
\end{align*}
Finally, we define 
$$
w_m := -\<1_black>^{m}_{-\infty,\cdot} - h_m.
$$
 We prove that the following convergences hold in every stochastic $L^p$ space of random variables taking values
in $C([0,T];\CC^{-\al})$,
\begin{align*}
 T_{w_m} \<1_black>_{-\infty,\cdot} \to 0, \, 
 T_{w_m} \<2_black>_{-\infty,\cdot} \to -\Re, \, 
 T_{w_m} \<3_black>_{-\infty,\cdot} \to 0.
\end{align*}
 By the same argument as in 
\cite[Lemma 3.13]{ChF16} this implies the result. For the reader's convenience, we sketch the argument here.
Since $w_m\in \mathscr{H}(T)$ $\PP$-almost surely, by Lemma \cite[Corollary 3.10]{ChF16} there exists a subset $\Omega'$ of $\Omega$ of 
probability one such that for every $\omega \in \Omega'$
$$
(T_{w_m(\omega)} \<1_black>_{-\infty,\cdot}(\omega), T_{w_m(\omega)} \<2_black>_{-\infty,\cdot}(\omega), 
T_{w_m(\omega)} \<3_black>_{-\infty,\cdot}(\omega))
\in \supp \PP_{\underline{Y}},
$$
for every $m\geq 1$. Given that $\supp \PP_{\underline{Y}}$ is closed under the norm $\VVert{\cdot}_{\al;0;T}$, we can conclude that $(0, - \Re, 0)
\in \supp \PP_{\underline{Y}}$ as soon as the above convergence holds for a single element $\omega \in \Omega'$. The stochastic $L^p$ convergence implies almost sure convergence along a subsequence which is   sufficient. 

The convergence of $T_{w_m}\<1_black>_{-\infty,\cdot}$ to $0$ is an immediate consequence of Proposition \ref{prop:diagram_convergence} 
and Lemma \ref{lem:f_m}.
 
If we compute the corresponding shift for $\<2_black>_{-\infty,t}$ we get
\begin{align*}
 T_{w_m}\<2_black>_{-\infty,t} & = \<2_black>_{-\infty,t} +\left((\<1_black>^m_{-\infty,t})^2-\Re^m\right) 
 - 2 \left(\<1_black>_{-\infty,t}\<1_black>_{-\infty,t}^m - \Re^m\right) \\
 & + 2 \<1_black>_{-\infty,t}^m h_m(t)  + \HH_2(h_m(t), \Re^m),
\end{align*}
where we also add and subtract $2\Re^m$ where necessary.  If we choose $\lambda$ sufficiently small we can ensure that 
\begin{equation*}
 \<1_black>_{-\infty,t}^m\circ h_m(t) \equiv 0,
\end{equation*}
where $\<1_black>_{-\infty,t}^m\circ h_m(t)$ is the resonant term define in \eqref{eq:resonant_term}. Using the Bony estimates 
(see Proposition \ref{prop:Bony_Est}), Lemma \ref{lem:f_m} and the fact that $\<1_black>_{0,\cdot}^m$ is bounded 
in every stochastic $L^p$ space taking values in $C([0,T];\CC^{-\al})$ we get that $\<1_black>_{-\infty,t}^m h_m(t)  \to 0$.
For the term
$$
\<2_black>_{-\infty,t} + \left((\<1_black>^m_{-\infty,t})^2-\Re^m\right) - 2 \left(\<1_black>_{-\infty,t}\<1_black>_{-\infty,t}^m 
- \Re^m\right)
$$ 
by Proposition \ref{prop:diagram_convergence} it suffices to compute the limit of $\<1_black>_{-\infty,t}\<1_black>_{-\infty,t}^m -\Re^m$.  We only give a sketch of the proof since 
the idea is similar to the one in the proof of Proposition \ref{prop:diagram_convergence}. Notice that for $m'>m$, 
$\EE\<1_black>_{-\infty,t}^{m'}\<1_black>_{-\infty,t}^m = \Re^m$, thus using \cite[Proposition 1.1.2]{Nu06} we have that
\begin{equation*}
 \<1_black>_{-\infty,t}^{m'}\<1_black>_{-\infty,t}^m -\Re^m = \<1_black>_{-\infty,t}^{m'} \otimes \<1_black>_{-\infty,t}^m, 
\end{equation*}
where $\otimes$ denotes the renormalized product given by
\begin{align*}
 \<j_black>_{-\infty,t}^{m'}\otimes\<i_black>_{-\infty,t}^m(z) & := \\
 \int_{\{(-\infty,t]\times\TT^2\}^{j+i}} \prod_{\substack{ 1\leq j'\leq j\\ 1\leq i'\leq i}} H_{m'}(t-r_{i'},z-z_{i'}) &
 H_m(t-r_{j'},z-z_{j'}) \xi(\otimes_{k=1}^{i+j} \dd z_k,\otimes_{k=1}^{i+j} \dd r_k),
\end{align*}
for every $z\in \TT^2$ and $i,j\geq 1$. In the same spirit as in the proof of Proposition \ref{prop:diagram_convergence} (see Appendix 
\hyperref[proof:diagram_convergence]{E}) we can prove that
\begin{equation*}
 \lim_{m\to\infty}\lim_{m'\to\infty} \EE \sup_{t\leq T} 
 \|\<1_black>_{-\infty,t}^{m'} \otimes \<1_black>_{-\infty,t}^m - 
 \<2_black>_{-\infty,t}\|_{\CC^{-\al}}^p= 0,
\end{equation*}
for every $p\geq 2$. Combining the above with the fact that $\sup_{t\leq T} \|h_m(t)^2- (\Re^m - \Re)\|_{\CC^{-\al}}$ converges to $0$, 
we obtain that $T_{w_m}\<2_black>_{-\infty,\cdot} \to -\Re$. 

For the term $T_{w_m} \<3_black>_{-\infty,t}$, by adding and subtracting multiples of 
$\Re^m\<1_black>^m_{-\infty,t}$ and $\Re^m$ where necessary we have that
\begin{align*}
 T_{w_m}\<3_black>_{-\infty,t} & = \<3_black>_{-\infty,t} - \left( (\<1_black>^m_{-\infty,t})^3 - 3 \Re^m \<1_black>_{-\infty,t}^m\right) 
   - 3 \left( \<1_black>^m_{-\infty,t} \<2_black>_{-\infty,t} - 2\Re^m \<1_black>_{-\infty,t}^m\right) \\
 & + 3 \left( \<1_black>_{-\infty,t} (\<1_black>_{-\infty,t}^m)^2 - 3\Re^m \<1_black>^m_{-\infty,t}\right) \\
 &  + 3 h_m(t)\left( \<2_black>_{-\infty,t} +((\<1_black>_{-\infty,t}^m)^2 - \Re^m) 
   - 2 (\<1_black>_{-\infty,t} \<1_black>_{-\infty,t}^m - \Re^m)\right) \\
 &  + 3 h_m(t)^2 \left(\<1_black>_{-\infty,t} -\<1_black>_{-\infty,t}^m\right)
   + \HH_3(h_m(t), \Re^m).
\end{align*}
For the terms $\<1_black>^m_{-\infty,t} \<2_black>_{-\infty,t} - 2\Re^m \<1_black>_{-\infty,t}^m$, 
$\<1_black>_{-\infty,t} (\<1_black>_{-\infty,t}^m)^2 - 3\Re^m \<1_black>^m_{-\infty,t}$ using again 
\cite[Proposition 1.1.2]{Nu06} for $m'>m$ we have that
\begin{align*}
 \<1_black>^m_{-\infty,t} \<2_black>_{-\infty,t}^{m'} - 2\Re^m \<1_black>_{-\infty,t}^m & = \<1_black>_{-\infty,t}^m\otimes 
 \<2_black>_{-\infty,t}^{m'} + 2\Re^m (\<1_black>_{-\infty,t}^{m'} - \<1_black>_{-\infty,t}^m)\\
 \<1_black>_{-\infty,t}^{m'} (\<1_black>_{-\infty,t}^m)^2 - 3\Re^m \<1_black>^m_{-\infty,t} & = \<1_black>_{-\infty,t}^{m'}\otimes 
 \<2_black>_{-\infty,t}^m + \Re^m (\<1_black>_{-\infty,t}^{m'} - \<1_black>_{-\infty,t}^m).
\end{align*}
If we proceed again in the spirit of the proof of Proposition \ref{prop:diagram_convergence} 
(see Appendix \hyperref[proof:diagram_convergence]{E})  we obtain that
\begin{align*}
 \lim_{m\to\infty}\lim_{m'\to\infty} \EE\sup_{t\leq T}\|\<1_black>_{-\infty,t}^m\otimes 
 \<2_black>_{-\infty,t}^{m'} - \<3_black>_{-\infty,t}\|_{\CC^{-\al}}^p & = 0\\
 \lim_{m\to\infty}\lim_{m'\to\infty} \EE\sup_{t\leq T}\|\<1_black>_{-\infty,t}^{m'}\otimes 
 \<2_black>_{-\infty,t}^m - \<3_black>_{-\infty,t}\|_{\CC^{-\al}}^p & = 0 \\
 \lim_{m\to\infty}\lim_{m'\to\infty} \left(\Re^m\right)^p \EE\sup_{t\leq T}
 \|\<1_black>_{-\infty,t}^{m'} - \<1_black>_{-\infty,t}^m\|_{\CC^{-\al}}^p & = 0,
\end{align*}
for every $p\geq 2$. It remains to handle the terms
\begin{align}
 h_m(t) \Big(\<2_black>_{-\infty,t} & - (\<1_black>_{-\infty,t}\<1_black>_{-\infty,t}^m - \Re^m)\Big), \label{eq:term_1}\\
 h_m(t) \Big((\<1_black>_{-\infty,t}^m)^2 - & \Re^m - (\<1_black>_{-\infty,t}\<1_black>_{-\infty,t}^m - \Re^m)\Big) \label{eq:term_2}
\end{align}
and 
\begin{equation}\label{eq:term_3}
h_m(t)^2(\<1_black>_{-\infty,t}-\<1_black>^m_{-\infty,t}).
\end{equation}
 We only show that \eqref{eq:term_1} converges to $0$ since \eqref{eq:term_2} and \eqref{eq:term_3} can be handled in a similar way. 
In particular due to Bony estimates (see Proposition \ref{prop:Bony_Est}) and the convergence of both factors individually, it suffices to prove that the resonant term 
\begin{align*}
 h_m(t) \circ \left(\<2_black>_{-\infty,t} - (\<1_black>_{-\infty,t}\<1_black>_{-\infty,t}^m - \Re^m)\right) & = \\
 \sum_{|\kk_1-\kk_2|\leq 1} & \delta_{\kk_1}h_m(t)\delta_{\kk_2}
 \left[\<2_black>_{-\infty,t} - (\<1_black>_{-\infty,t}\<1_black>_{-\infty,t}^m - \Re^m)\right],
\end{align*}
converges to $0$. Since the Fourier modes of $h_m$ are localized at the 
points $2^mz_0$ and $-2^mz_0$ we have that 
\begin{align*}
 h_m(t) \circ \left(\<2_black>_{-\infty,t} - (\<1_black>_{-\infty,t}\<1_black>_{-\infty,t}^m - \Re^m)\right) & = \\
  h_m(t) \sum_{i=-1,0,1} & \delta_{m+i}\left[\<2_black>_{-\infty,t}
 - (\<1_black>_{-\infty,t}\<1_black>_{-\infty,t}^m - \Re^m)\right].
\end{align*}
Let $\kk\geq -1$ and $Y_m(t) =\<2_black>_{-\infty,t} - (\<1_black>_{-\infty,t}\<1_black>_{-\infty,t}^m - \Re^m)$. Then, for $i=-1,0,1$, 
\begin{align*}
 \EE\delta_\kk[h_m(t_1)\delta_{m+i}Y_m(t_1)](z_1) \delta_\kk[h_m(t_2)\delta_{m+i}Y_m(t_2)](z_2) & = \\
 \int_{\TT^2\times\TT^2} C_{m,i}(t_1-t_2,\bar z_1-\bar z_2) 
 \eta_\kk(z_1-\bar z_1) \eta_\kk(z_2-\bar z_2) & h_m(t_1,\bar z_1)h_m(t_2,\bar z_2) \dd \bar z_1 \dd \bar z_2,
\end{align*}
where
\begin{equation*}
 C_{m,i}(t_1-t_2,\bar z_1-\bar z_2) = \EE\delta_{m+i}[Y_m(t_1)](\bar z_1) \delta_{m+i}[Y_m(t_2)](\bar z_2).
\end{equation*}
 For $m'>m$ using \cite[Proposition 1.1.2]{Nu06} we have that $\<1_black>_{-\infty,t}^{m'}\<1_black>_{-\infty,t}^m - \Re^m =
\<1_black>_{-\infty,t}^{m'}\otimes\<1_black>_{-\infty,t}^m$. Let $Y_{m,m'}(t) = \<2_black>_{-\infty,t} -\<1_black>_{-\infty,t}^{m'}\otimes
\<1_black>_{-\infty,t}^m$ and notice that
\begin{align*}
 \EE\delta_{m+i}[Y_{m,m'}(t_1)](z_1) \delta_{m+i}[Y_{m,m'}(t_2)](z_2) & = \\
 C \sum_{\substack{|l_1|> \lambda 2^{m'} \\ |l_2|> \lambda 2^{m}}}  
 \prod_{j=1,2}&\frac{1-\e^{-I_{l_j}|t_2-t_1|}}{2 I_{l_j}} |\chi_{m-i}(l_1+l_2)|^2 e_{l_1+l_2}(z_1-z_2),
\end{align*}
 for some constant $C$ independent of $m$ and $m'$. Then for every $\gamma\in (0,\frac{1}{2})$ by a change of variables
\begin{align*}
 \int_{\TT^2\times\TT^2} C_{m,m',i}(t_1-t_2,\bar z_1-\bar z_2)
 \eta_\kk(z_1-\bar z_1) \eta_\kk(z_2-\bar z_2) & h_m(t_1,\bar z_1)h_m(t_2,\bar z_2) \dd \bar z_1 \dd \bar z_2 \lesssim \\
 (m+1) |t_1-t_2|^{2\gamma} & \underbrace{\sum_{\substack{ l\in \mathcal{A}_{2^{m-i}}\\ l+2^mz_0\in \mathcal{A}_{2^\kk}}} 
 K^\gamma \star^2_{>\lambda 2^m} K^\gamma(l)}_{I},
\end{align*}
where $K^\gamma(l) = \frac{1}{(1+|l|^2)^{1-\gamma}}$ and $C_{m,m',i}$ is defined as $C_{m,i}$ with $Y_m$ replaced by
$Y_{m,m'}$. By Corollary \ref{cor:nested_kernel_conv}
\begin{equation*}
 I \lesssim \sum_{\substack{ l\in \mathcal{A}_{2^{m-i}} \\
 l+2^mz_0\in\mathcal{A}_{2^\kk}}} \frac{1}{(1+|l|^2)^{1-2\gamma}},
\end{equation*}
thus for every $\ee>2\gamma$ 
\begin{equation*}
 I \lesssim 2^{2\ee k} \sum_{l\in\ZZ^2}  \frac{1}{(1+|l|^2)^{1-2\gamma}} \frac{1}{(1+|l+2^mz_0|^2)^\ee}.
\end{equation*}
Using Corollary \ref{cor:nested_kernel_conv} we obtain
\begin{align*}
 \EE\delta_\kk[h_m(t_1)\delta_{m+i}Y_{m,m'}(t_1)](z_1) \delta_\kk[h_m(t_2)\delta_{m+i}Y_{m,m'}(t_2)](z_2) &
 \lesssim \\
 & \frac{2^{2\lambda \kk}(m+1)}{(1+|2^mz_0|^2)^{\ee-2\gamma}} |t_1-t_2|^{2\gamma},
\end{align*}
for every $\gamma\in(0,\frac{1}{2})$ and $\ee>2\gamma$. Using Nelson's estimate \eqref{eq:Nelson's_Estimate} for every $p\geq 2$, the usual Kolmogorov's criterion and the embedding
$\BB^{-\al+\frac{2}{p}}_{p,p} \hookrightarrow \CC^{-\al}$ we finally obtain that
\begin{equation*}
 \lim_{m\to\infty} \lim_{m'\to \infty} \EE\sup_{t\leq T}
 \|h_m(t) \circ \left(\<2_black>_{-\infty,t} - (\<1_black>_{-\infty,t}\<1_black>_{-\infty,t}^m 
 - \Re^m)\right)\|_{\CC^{-\al}}^p
 =0.
\end{equation*}
Convergence of $h_m\left(\<2_black>_{-\infty,\cdot} - (\<1_black>_{-\infty,\cdot}\<1_black>_{-\infty,\cdot}^m - \Re^m)\right)$
to $0$ then follows from Bony estimates (see Proposition \ref{prop:Bony_Est}).   
\end{proof}

For $x\in \CC^{-\al_0}$, $f\in L^2(\RR\times \TT^2)$ and $\Re\geq 0$, let $\mathscr{T}(x;f;\Re)$ be the solution map of the equation
\begin{equation}\label{eq:h_eq}
\left\{
\begin{array}{lcll}
\partial_t X & = & \Delta X - X - \sum_{k=0}^3 a_k \HH_k(X,\Re) +  f \\
f(0,\cdot) & = & x
\end{array}
\right..
\end{equation}
The following corollary is an immediate consequence of Theorem \ref{thm:noise_support}.
\begin{corollary}\label{cor:support} Let $X(\cdot;x)$ be the solution to \eqref{eq:Remainder_Eq_1}
for $n=3$ and $x\in\CC^{-\al_0}$ and denote by $\PP_{X(\cdot;x)}$ its law in 
$C([0,T];\CC^{-\al_0})$. Then
\begin{equation*}
 \supp \PP_{X(\cdot;x)} = 
 \overline{\left\{\mathscr{T}(x;f;\Re): f\in L^2(\RR\times \TT^2), \, \Re\geq 0\right\}}^{C([0,T];\CC^{-\al_0})}.
\end{equation*}
\end{corollary}

\begin{proof} See the proof of \cite[Theorem 1.1]{ChF16}. 
\end{proof}

Using the above corollary we prove that for every $y\in \CC^{-\al_0}$ and every $\ee>0$
\begin{equation}\label{eq:irr} 
 \PP(X(T;x)\in B_\ee(y)) >0.
\end{equation}
To do so, it suffices to prove that for every $y\in C^\infty(\TT^2)$ there exist $f\in L^2(\TT^2)$ and
$\Re\geq 0$ such that $\mathscr{T}(x;f; \Re)(T) = y$.  But if we set 
\begin{equation*}
 X(t) = S(t) x + \frac{t}{T} \left(y - S(T) x\right),
\end{equation*}
for any choice of $\Re\geq 0$ and
$$
f(t) = \sum_{k=0}^3 a_k \HH_k(X(t), \Re) + \frac{1}{T} (y- S(T)x) -\frac{t}{T} (\Delta - I)(y- S(T))
$$
we have that $X = \mathscr{T}(x;f; \Re)$. Then the result follows by Corollary \ref{cor:support} and the fact that $C^\infty(\TT^2)$ is dense in $\CC^{-\al}$. 

\subsection{Convergence Rate}\label{s:Convergence_Rate}

 We recall that for any coupling $M$ of probability measures $\mu_1, \mu_2$ and $F,G$ measurable functions with respect to the 
corresponding $\sigma$-algebras we have the identity 
\begin{equation}\label{eq:coupling_equal}
 \int \left(F(x) - G(y)\right) M(\dd x, \dd y) = \int \int \left(F(x) - G(y)\right) \mu_1(\dd x) \mu_2(\dd y).
\end{equation}
We finally combine the results of the previous sections to prove the following theorem.

\begin{theorem}\label{thm:conv_rate} Let $\{P_t:t\geq 0\}$ be the Markov semigroup \eqref{eq:semigroup}
associated to the solution of \eqref{eq:Remainder_Eq_1} for $n=3$. Then there exists
$\lambda\in(0,1)$ such that
\begin{equation}\label{eq:minor}
 \|P_t(x) - P_t(y)\|_{\mathrm{TV}} \leq 1-\lambda,
\end{equation}
for every $x,y\in\CC^{-\al_0}$, $t\geq 3$. 
\end{theorem}

\begin{proof} Let $0<\al<\al_0$ and for $R>0$ consider the subset of $\CC^{-\al_0}$
\begin{equation*}
 A_R :=\{x\in\CC^{-\al_0} : \|x\|_{\CC^{-\al}}\leq R\}
\end{equation*}
which is compact since the embedding $\CC^{-\al} \hookrightarrow 
\CC^{-\al_0}$ is compact (see Proposition \ref{prop:Compact_Emb}). By Theorem \ref{thm:strong_feller} 
for every $a\in(0,1)$ there exists $r\equiv r(a)>0$ such that for every $x,y\in \bar B_r(0)$ 
and $t\geq 1$
\begin{equation*}
 \|P_t(x) - P_t(y)\|_{\mathrm{TV}} \leq 1-a.
\end{equation*}
By \eqref{eq:irr} for every $x \in A_R$ 
\begin{equation*}
 P_{1}(x;\bar B_r(0)) >0,
\end{equation*}
which combined with the strong Feller property  (which implies the continuity of $P_{1}(x;A)$ as a function of 
$x$ for fixed measurable set $A$) and the fact that $A_R$ is compact implies that there exists $b\equiv b(R)>0$ such that
\begin{equation*}
 \inf_{x\in A_R} P_{1}(x;\bar B_r(0)) \geq b.
\end{equation*}
For $t\geq 0$ and $x,y\in A_R\setminus \bar B_r(0)$, let 
$\PP^{x,y}_t\in\mathcal{M}_1(\CC^{-\al_0}\times\CC^{-\al_0})$ be the trivial product coupling of 
$P_t(x)$ and $P_t(y)$ given by
\begin{equation*}
 \PP^{x,y}_t(A\times B) = P_t(x;A) P_t(y;B),
\end{equation*}
for every measurable sets $A, B\subset \CC^{-\al_0}$. Then, for $x,y\in A_R$, $t\geq 2$ and $\Phi\in C_b(\CC^{-\al_0})$,
\begin{align*}
 |P_t\Phi(x) - P_t\Phi(y)| & = |\EE\left[P_{t-1}\Phi(X(1;x)) - P_{t-1}\Phi(X(1;y))\right]| \\
 & = \left|\int \left[P_{t-1}\Phi(\tilde x) - P_{t-1}\Phi(\tilde y)\right] \PP^{x,y}_1(\dd \tilde x,\dd \tilde y)\right|,
\end{align*}
where in the first equality we use the Markov property and  \eqref{eq:coupling_equal} in the second equality. This implies that
\begin{align*}
 \|P_t(x) - P_t(y)\|_{\mathrm{TV}} & \leq \PP^{x,y}_1\left(\left(\bar B_r(0)\times \bar B_r(0)\right)^c\right) + (1-a) \PP^{x,y}_1\left(\bar B_r(0)\times 
 \bar B_r(0)\right) \\
 & \leq 1-a \PP^{x,y}_1\left(\bar B_r(0)\times \bar B_r(0)\right)\\
 & \leq 1-ab^2.
\end{align*}
By \eqref{eq:time_ind_moments} we can choose $R>0$ sufficiently large such that
\begin{equation*}
 \inf_{x\in \CC^{-\al_0}}\inf_{t\geq 1} \PP(\|X(t;x)\|_{\CC^{-\al}}\leq R) 
 > \frac{1}{2}.
\end{equation*}
Then, for any $x,y\in \CC^{-\al_0}$ and $t\geq 3$, using the same coupling argument as above we get 
\begin{equation*} 
 \|P_t(x) - P_t(y)\|_{\mathrm{TV}} \leq 1-\frac{ab^2}{4},
\end{equation*}
which completes the proof if we set $\lambda =\frac{ab^2}{4}$. 
\end{proof}

The following corollary contains our main result, the exponential convergence to a unique invariant measure. 

\begin{corollary} There exists a unique invariant measure $\mu\in\mathcal{M}_1(\CC^{-\al_0})$ for the semigroup $\{P_t:t\geq 0\}$ associated to
the solution of \eqref{eq:Remainder_Eq_1} for $n=3$ such that
\begin{equation}\label{eq:conv_rate}
 \|P_{t}(x) - \mu\|_{\mathrm{TV}} \leq \left(1-\lambda\right)^{\left\lfloor 
 \frac{t}{3}\right\rfloor} 
 \|\delta_x -\mu\|_{\mathrm{TV}},
\end{equation}
for every $x\in\CC^{-\al_0}$, $t\geq 3$. 
\end{corollary}

\begin{proof} We first notice that for $\mu_1,\mu_2\in\mathcal{M}_1(\CC^{-\al_0})$ and every $t\geq 0$ by \eqref{eq:coupling_equal} 
we have that
\begin{equation*}
 \|P^*_t\mu_1-P^*_t\mu_2\|_{\mathrm{TV}} \leq \frac{1}{2}\sup_{\|\Phi\|_{\infty}\leq 1} 
 \int\int \left|P_t\Phi(x)-P_t\Phi(y)\right| M(\dd x, \dd y),
\end{equation*}
for any coupling $M\in \mathcal{M}_1(\CC^{-\al_0}\times\CC^{-\al_0})$ of $\mu_1$ and $\mu_2$. Thus by \eqref{eq:minor} 
for $t\geq 3$
\begin{equation*}
 \|P^*_t\mu_1-P^*_t\mu_2\|_{\mathrm{TV}} \leq \left(1-\lambda\right) \left(1-M(\{(x,x):x \in \CC^{-\al_0}\})\right)
\end{equation*}
and using the characterization of the total variation distance given by
\begin{equation*}
 \|\mu_1-\mu_2\|_{\mathrm{TV}} 
 =  2\inf\left\{ 1-M(\{(x,x):x \in \CC^{-\al_0}\}) : M \text{ coupling of } \mu_1 \text{ and } \mu_2\right\}
\end{equation*}
we get that
\begin{equation*}
 \|P^*_t\mu_1-P^*_t\mu_2\|_{\mathrm{TV}} \leq \left(1-\lambda\right) \|\mu_1-\mu_2\|_{\mathrm{TV}}.
\end{equation*}
This implies that $\{P_t:t\geq 0\}$ has a unique invariant measure $\mu\in\mathcal{M}_1(\CC^{-\al_0})$, since by Proposition 
\cite[Proposition 3.2.5]{dPZ96} any two distinct invariant measures are singular. Finally, for
$x\in \CC^{-\al_0}$ and $t\geq 3$
\begin{align*}
 \|P_t(x) - \mu\|_{\mathrm{TV}} \leq (1-\lambda) \|P_{t-3}(x) - \mu\|_{\mathrm{TV}},
\end{align*}
which implies \eqref{eq:conv_rate}.
\end{proof}

\begin{appendices} 

\stepcounter{section}
\section*{Appendix \thesection}\label{Appendix:Besov_Spaces}
\setcounter{theorem}{0}
\setcounter{equation}{0} 

The following three propositions can be found in \cite[Section 3: pp. 11-12]{MWe15}.
\begin{proposition} Let $\alpha_1,\alpha_2 \in \mathbb{R}$, $p_1,p_2,q_1,q_2\in[1,\infty]$. Then,
\begin{align}
 \|f\|_{\mathcal{B}_{p_1,q_1}^{\alpha_1}} & \leq C\|f\|_{\mathcal{B}_{p_1,q_1}^{\alpha_2}}, \text{ whenever } 
  \alpha_1 \leq \alpha_2, \label{eq:alpha_ineq}\\
 \|f\|_{\mathcal{B}_{p_1,q_1}^{\alpha_1}} & \leq \|f\|_{\mathcal{B}_{p_1,q_2}^{\alpha_1}}, \text{ whenever } 
  q_1\geq q_2,\label{eq:q_ineq}\\
 \|f\|_{\mathcal{B}_{p_1,q_1}^{\alpha_1}} & \leq C\|f\|_{\mathcal{B}_{p_2,q_1}^{\alpha_1}}, \text{ whenever }
  p_1\leq p_2,\label{eq:p_ineq}\\
 \|f\|_{\mathcal{B}_{p_1,q_1}^{\alpha_1}} & \leq C\|f\|_{\mathcal{B}_{p_1,q_2}^{\alpha_2}}, \text{ whenever }
  \alpha_1 < \alpha_2. \label{eq:a_q_ineq}
\end{align}
\end{proposition}
\begin{proposition} Let $p\in [1,\infty]$. Then the space $\mathcal{B}_{p,1}^0$ is continuously embedded in $L^p$  
and
\begin{equation}\label{eq:L^p_emb}
  \|f\|_{L^p}  \leq\|f\|_{\mathcal{B}_{p,1}^0}. 
\end{equation}
On the other hand, $L^p$ is continuously embedded in $\mathcal{B}_{p,\infty}^0$ and
\begin{equation}\label{eq:B^0_emb}
 \|f\|_{\mathcal{B}_{p,\infty}^0}  \leq C \|f\|_{L^p}.
\end{equation}
\end{proposition}

\begin{proposition}\label{prop:Besov_Emb} Let $\alpha \leq \beta$ and $p,q\geq 1$ such that $p\geq q$ and $\beta=\alpha+ d \left(\frac{1}{q}-\frac{1}{p}\right)$. Then
\begin{equation*}
  \|f\|_{\BB_{p,\infty}^\alpha}\leq C \|f\|_{\BB_{q,\infty}^\beta}.
\end{equation*} 
\end{proposition}

The following proposition can be found in \cite{BCD11}[Corollary 2.96] and it is generally true for Besov spaces over
compact sets.

\begin{proposition} \label{prop:Compact_Emb} Let $\al<\al'$. Then the embedding $\BB_{\infty,\infty}^{\al'} \hookrightarrow \BB^\al_{\infty,1}$ 
is compact.
\end{proposition}

In the following proposition we describe the smoothing properties of the heat semigroup $\left(\e^{t\Delta}\right)_{t\geq 0}$ 
with generator $\Delta$ in space (see \cite[Proposition 3.11]{MWe15}). 

\begin{proposition} \label{prop:Heat_Smooth} 
Let $f\in\BB_{p,q}^\alpha$. Then, for all $\beta\geq \alpha$,  
\begin{equation}\label{eq:Heat_Smooth}
  \|\e^{t\Delta}f\|_{\BB_{p,q}^\beta}\leq C t^{\frac{\alpha-\beta}{2}}\|f\|_{\BB_{p,q}^\alpha},
\end{equation}
for every $t\leq 1$.
\end{proposition}

%
%

For $f,g\in C^\infty(\TT^d)$ we define the paraproduct $f\prec g$ and the resonant term 
$f \circ g$ by
\begin{align}
 f\prec g  & := \sum_{\iota<\kk -1} \delta_\iota f\delta_\kk g, \label{eq:paraproduct}\\
 f \circ g & := \sum_{|\iota - \kk|\leq 1} \delta_\iota f\delta_\kk g\label{eq:resonant_term}.
\end{align}
We also let $f\succ g:= g\prec f$. Notice that formally
$$
fg = f\prec g+ f\circ g+ f\succ g.
$$
We then have the following estimates due to Bony.

\begin{proposition}\label{prop:Bony_Est}{\normalfont(\cite[Theorems 2.82 and 2.85]{BCD11})} Let $\al,\bt\in \RR$ and $g\in \CC^\bt$.
 \begin{enumerate}[i.]
  \item If $f\in L^\infty$, 
  $\|f\prec g\|_{\CC^\bt} \leq C \|f\|_{L^\infty} \|g\|_{\CC^{\beta}}$.
  \item If $\al<0$ and $f\in \CC^\al$, $\|f\prec g\|_{\CC^{\al+\bt}}\leq C \|f\|_{\CC^\al} \|g\|_{\CC^\bt}$.
  \item If $\al+\bt>0$ and $f\in \CC^\al$, 
  $\|f\circ g\|_{\CC^{\al+\bt}} \leq C \|f\|_{\CC^\al} \|g\|_{\CC^\bt}$.
 \end{enumerate}
\end{proposition}

The above proposition allows us to define the product of a distribution and a function in a canonical way
under certain regularity assumptions (see \cite[Corollary 3.21]{MWe15}).

%
%

\begin{proposition}\label{prop:Mult_Ineq_II}
Let $f\in \CC^\alpha$ and $g\in \CC^\beta$, where $\alpha< 0 <\beta$, $\alpha+\beta>0$. Then $fg$ 
can be uniquely defined as an element in $\CC^\alpha$ such that
\begin{equation*} 
  \|fg\|_{\CC^\alpha} \leq C \|f\|_{\CC^\alpha} \|g\|_{\CC^\beta}.
\end{equation*}
\end{proposition}

Regarding the inner product on $L^2(\TT^d)$ we have the following extension result (see \cite[Proposition 3.23]{MWe15}).

\begin{proposition}\label{prop:Inner_Prod} Let $p,q\geq 1$ and 
$p',q'$ their conjugate exponents. Then, 
for every $0\leq \alpha <1$, the $L^2(\TT^d)$ inner product can be uniquely extended to a continuous bilinear form on 
$\BB_{p,q}^\alpha\times \BB_{p',q'}^{-\alpha}$ such that
\begin{equation*}
|\langle f, g\rangle|\leq C \|f\|_{\BB_{p,q}^\alpha} \|g\|_{\BB_{p',q'}^{-\alpha}},  
\end{equation*}
for all $(f,g)\in\BB_{p,q}^\alpha\times \BB_{p',q'}^{-\alpha}$.
\end{proposition}

Finally we have the following gradient estimate for functions of positive regularity (see \cite[Proposition 3.25]{MWe15}).   

\begin{proposition} Let $f\in \BB_{1,1}^\alpha$, $\alpha\in(0,1)$. Then
\begin{equation}\label{eq:gradient_estimate} 
\|f\|_{\BB_{1,1}^\alpha}\leq C\left(\|f\|_{L^1}^{1-\alpha}\|\nabla f\|_{L^1}^\alpha+\|f\|_{L^1}\right).
\end{equation}
\end{proposition}
\stepcounter{section}
\section*{Appendix \thesection}\label{Appendix:The_Space_Time_White_Noise}
\setcounter{theorem}{0}
\setcounter{equation}{0}

\begin{definition} Let $\left\{\xi(\phi)\right\}_{\phi\in L^2(\RR\times\TT^d)}$ be a family of centered Gaussian random 
variables on a probability space $(\Omega,\mathcal{F},\PP)$ such that
\begin{equation*}
 \EE(\xi(\phi)\xi(\psi))=\langle \phi, \psi\rangle_{L^2(\RR\times\TT^d)},
\end{equation*}
for all $\phi,\psi\in L^2(\RR\times\TT^d)$. Then $\xi$ is called a space-time white noise on $\RR\times\TT^d$. 
 
\end{definition}

The existence of such a family of random variables on some probability space $(\Omega, \mathcal{F}, \PP)$ 
is assured by Kolmogorov's extension theorem and by definition we can check that it is linear, i.e. for all $\lambda,\nu\in \RR$, 
$\phi,\psi\in L^2(\RR\times\TT^d)$ we have that $\xi(\lambda\phi+\nu\psi)=\lambda\xi(\phi)+\nu\xi(\psi)$ $\PP$-almost surely 
(see \cite[Chapter 1]{Nu06}). 
We interpret $\xi(\phi)$ as  a stochastic integral and write
\begin{equation*}
 \int_{\RR\times\TT^d} \phi(t,x) \xi(\dd t, \dd x):= \xi(\phi),
\end{equation*}
for all $\phi\in L^2(\RR\times\TT^d)$. We use this notation, but stress that $\xi$ is almost surely not a measure 
and that the stochastic integral is only defined on a set of measure one which my depend on the specific 
choice of $\phi$.

We also define multiple stochastic integrals (see \cite[Chapter 1]{Nu06}) on $\RR\times\TT^d$ for all
symmetric functions $f$ in $L^2\left((\RR\times\TT^d)^n\right)$, for some $n\in \NN$, i.e. functions such that 
$f(z_1, z_2, \ldots, z_n)=f(z_{i_1}, z_{i_2}, \ldots, z_{i_n})$ for any permutation $(i_1, i_2, \ldots, i_n)$ of 
$(1, 2, \ldots, n)$. Here $z_j$ is an element of $\RR\times \TT^d$, for all $j\in\{1,2,\ldots, n\}$. For such a symmetric
function $f$ we denote its $n$-th iterated stochastic integral by
\begin{equation*}
 I_n(f):=\int_{\{\RR\times\TT^d)^n\}} f(z_1, z_2, \ldots, z_n) \, \xi(\dd z_1\otimes \dd z_2 \otimes \ldots \otimes \dd z_n).
\end{equation*}
The following theorem can be found in \cite[Theorem 1.1.2]{Nu06}.
\begin{theorem} Let $\mathcal{F}_\xi$ be the $\sigma$-algebra generated by the family of random variables
$\left\{\xi(\phi)\right\}_{\phi\in L^2(\RR\times\TT^d)}$. Then every element $X\in L^2(\Omega,\mathcal{F}_\xi,\PP)$
can be written in the following form
\begin{equation*}
 X=\EE(X)+\sum_{n=1}^\infty I_n(f_n),
\end{equation*}
where $f_n\in L^2\left((\RR\times\TT^d)^n\right)$ are symmetric functions, uniquely determined by $X$. 
\end{theorem}

The above theorem implies that $L^2(\Omega,\mathcal{F}_\xi,\PP)$ can be decomposed into a direct sum of the form 
$\bigoplus_{n\geq 0} S_n$, where $S_0:=\RR$ and 
\begin{equation}\label{eq:wiener_chaos}
S_n:=\{I_n(f):f\in L^2\left((\RR\times\TT^d)^n\right) \text{ symmetric }\},
\end{equation}
for all $n\geq 1$. The space $S_n$ is called the $n$-th homogeneous Wiener chaos and the element $I_n(f_n)$
the projection of $X$ onto $S_n$.

Given a symmetric function $f\in L^2\left((\RR\times\TT^d)^n\right)$, we have the isometry
\begin{equation}\label{eq:Ito_Isometry}
 \EE(I_n)^2= n! \|f\|_{L^2\left((\RR\times\TT^d)^n\right)}^2.
\end{equation}
Furthermore, by Nelson's estimate (see \cite[Section 1.4]{Nu06}) for every $n\geq 1$ and $Y\in S_n$,
\begin{equation}\label{eq:Nelson's_Estimate}
 \EE|Y|^p\leq (p-1)^{\frac{n}{2}p} (\EE|Y|^2)^{\frac{p}{2}},
\end{equation}
for every $p\geq 2$.
\stepcounter{section}
\section*{Appendix \thesection}\label{Appendix:Kernel_Conv}
\setcounter{theorem}{0}
\setcounter{equation}{0}

\begin{definition}\label{def:kernel_conv} For symmetric kernels $K_1, K_2:\ZZ^2\to (0,\infty)$ we denote by $K_1\star K_2$ the convolution given by
\begin{equation*}
 K_1\star K_2(m): = \sum_{l\in\ZZ^2} K_1(m-l) K_2(l)
\end{equation*}
and for $N\in \NN$ we let 
\begin{equation*}
 K_1\star_{\leq N} K_2(m) := \sum_{|l|\leq N} K_1(m-l) K_2(l).
\end{equation*}
as well as 
$$
K_1\star_{> N}K_2 := \left(K_1\star K_2\right) - \left(K_1\star_{\leq N} K_2\right).
$$
\end{definition}

We are interested in symmetric kernels $K$ for which there exists $\al\in(0,1]$ such that
\begin{equation*}
 K(m) \leq C \frac{1}{(1+|m|^2)^\al}.
\end{equation*}
In the spirit of \cite[Lemma 10.14]{Ha14} we have the following lemma.

\begin{lemma}\label{lem:kernel_conv} Let $\al,\bt\in(0,1]$ such that $\al+\bt-1>0$ and let $K_1, K_2:\ZZ^2\to (0,\infty)$ be symmetric kernels
such that
\begin{align*}
K_1(m) \leq C \frac{1}{(1+|m|^2)^\al}, \quad K_2(m) \leq C \frac{1}{(1+|m|^2)^\bt}.
\end{align*}
If $\al<1$ or $\bt<1$ then
\begin{align*}
 K_1\star K_2(m) & \leq C \frac{1}{(1+|m|^2)^{\al+\bt-1}} \\ 
 K_1\star_{> N} K_2(m) & \leq C 
 \begin{cases}
  \frac{1}{(1+|m|^2)^{\al+\bt-1}}, & \text{ if } |m|\geq N \\
  \frac{1}{(1+|N|^2)^{\al+\bt-1}}, & \text{ if } |m| < N
 \end{cases}
\end{align*}
and if $\al =\bt = 1$ 
\begin{align*}
 K_1\star K_2(m) & \leq C \frac{\log|m|\vee 1}{1+|m|^2} \\
 K_1\star_{> N} K_2(m) & \leq C 
 \begin{cases}
  \frac{\log|m|\vee 1}{1+|m|^2}, & \text{ if } |m|\geq N \\
  \frac{\log|N|\vee 1}{1+|N|^2}, & \text{ if } |m| < N
 \end{cases}.
\end{align*}
\end{lemma}
\begin{proof} We only prove the estimates for $K_1\star K_2$. The corresponding estimates for $K_1\star_{>N} K_2$ can be proven in 
a similar way. We consider the following regions of $\ZZ^2$,
\begin{align*}
 A_1 & = \left\{l:|l|\leq \frac{|m|}{2}\right\}, \\
 A_2 & = \left\{l:|l-m|\leq \frac{|m|}{2}\right\}, \\
 A_3 & = \left\{l: \frac{|m|}{2} \leq |l| \leq 2|m|, |l-m|\geq \frac{|m|}{2}\right\}, \\
 A_4 & = \left\{l: |l|>2|m|\right\}.
\end{align*}
For every $l\in A_1$ we notice that $|m-l| \geq \frac{3|m|}{4}$, which implies that 
\begin{align*}
 \sum_{l\in A_1} K_1(m-l) K_2(l) & \lesssim \frac{1}{(1+|m|^2)^\al} \sum_{l\in A_1} K_2(l) \\
 & \lesssim 
 \begin{cases}
 \frac{(1+|m|^2)^{\bt-1}}{(1+|m|^2)^\al}, & \text{ if } \bt<1 \\
 \frac{\log|m|\vee 1}{(1+|m|^2)^\al}, & \text{ if } \bt =1
 \end{cases}.
\end{align*}
By symmetry we get that
\begin{align*}
 \sum_{l\in A_2} K_1(m-l) K_2(l) & \lesssim 
 \begin{cases}
 \frac{(1+|m|^2)^{\al-1}}{(1+|m|^2)^\bt}, & \text{ if } \al<1 \\
 \frac{\log|m|\vee 1}{(1+|m|^2)^\bt}, & \text{ if } \al =1
 \end{cases}.
\end{align*}
For the summation over $A_3$ we notice that
\begin{equation*}
 \sum_{l\in A_3} K_1(m-l) K_2(l) \lesssim \frac{1+|m|^2}{(1+|m|^2)^{\al+\bt}}.
\end{equation*}
Finally, for $l\in A_4$ we have that $|m-l| \geq \frac{|l|}{2}$, which implies that
\begin{equation*}
 \sum_{l\in A_4} K_1(m-l) K_2(l) \lesssim \sum_{|l|>2|m|} \frac{1}{(1+|l|^2)^{\al+\bt}} \lesssim \frac{1}{(1+|m|^2)^{\al+\bt}}.
\end{equation*}
Combining all the above we thus obtain the appropriate estimate on $K_1\star K_2(m)$.
\end{proof}

Because we are interested in nested convolutions of the same kernel we introduce the following recursive notation
\begin{align*}
 K \star^1 K = K , \quad K \star^n K  = K \star \left(K \star^{n-1} K\right),
\end{align*}
for every $n\geq 2$, with the obvious interpretation for $K\star_{\leq N}^n K$ and $K\star_{> N}^n K$. We then have the 
following corollary, the proof of which is omitted since it is a straight consequence of Lemma \ref{lem:kernel_conv}.

\begin{corollary}\label{cor:nested_kernel_conv} Let $K$ be a symmetric kernel as above for some $\al\in (\frac{n-1}{n},1]$. If $\al<1$ then
\begin{align*}
 K\star^n K(m) & \leq C \frac{1}{(1+|m|^2)^{n\al-(n-1)}} \\
 K\star_{> N}^n K(m) & \leq C 
 \begin{cases}
  \frac{1}{(1+|m|^2)^{n\al-(n-1)}}, & \text{ if } |m|\geq N \\
  \frac{1}{(1+|N|^2)^{n\al-(n-1)}}, & \text{ if } |m| < N
 \end{cases}
\end{align*}
and if $\al=1$
\begin{align*}
 K\star^n K(m) & \leq C \frac{1}{(1+|m|^2)^{1-\ee}} \\
 K\star_{> N} K(m) & \leq C 
 \begin{cases}
  \frac{1}{(1+|m|^2)^{1-\ee}}, & \text{ if } |m|\geq N \\
  \frac{1}{(1+|N|^2)^{1-\ee}}, & \text{ if } |m| < N
 \end{cases}
\end{align*}
for every $\ee\in(0,1)$. 
\end{corollary}
\stepcounter{section}
\section*{Appendix \thesection}\label{proof:mod_sol}
\setcounter{theorem}{0}
\setcounter{equation}{0}

\begin{proof}[Proof of Theorem \ref{thm:mod_sol}] Let $\phi_1,\phi_2\in L^2(\TT^2)$ and notice that for $t_1,t_2>-\infty$ by \eqref{eq:Ito_Isometry}
\begin{align}
 \EE \<n_black>_{-\infty,t_1}(\phi_1) \<n_black>_{-\infty,t_2}(\phi_2) & = \label{eq:cov_formula}\\ 
 n! \int_{\TT^2}\int_{\TT^2} \phi_1(z_1)\phi_2(z_2) & 
 \left(\int_{-\infty}^{t_1\wedge t_2} H(t_1+t_2 - 2r, z_1-z_2) \dd r\right)^n \dd z_1 \dd z_2, \nonumber
\end{align}
where we also use the semigroup property
\begin{equation*}
 \int_{\TT^2} H(t_1-r, z_1 - z) H(t_2- r, z_2 - z) \dd z =  H(t_1+t_2 - 2r, z_1 - z_2).
\end{equation*}
For $I_m = 1+4\pi^2|m|^2$, $m\in\ZZ^2$, we rewrite \eqref{eq:cov_formula} as 
\begin{align*}
 \EE \<n_black>_{-\infty,t_1}(\phi_1) \<n_black>_{-\infty,t_2}(\phi_2) & = \\
 n! \sum_{\substack{ m_i\in \ZZ^2 \\ i=1,2,\ldots, n \\ m= m_1+\ldots+m_n}} &
 \prod_{i=1}^n \frac{\e^{-I_{m_i}|t_1-t_2|}}{2 I_{m_i}} \,
 \hat \phi_1(m) \overline{\hat \phi_2(m)},
\end{align*}
and if we replace $\phi_1$, $\phi_2$ by $\eta_\kk(z_1-\cdot)$, $\eta_\kk(z_2-\cdot)$ respectively,
for $\kk\geq -1$, $z_1,z_2\in \TT^2$, we have that
\begin{align*}
 \EE \delta_{\kk}\<n_black>_{-\infty,t_1}(z_1) \delta_k\<n_black>_{-\infty,t_2}(z_2) & = \\
 n! \sum_{\substack{ m_i\in \ZZ^2 \\ i=1,2,\ldots, n \\ m= m_1+\ldots+m_n}} & 
 \prod_{i=1}^n \frac{\e^{-I_{m_i}|t_1-t_2|}}{2 I_{m_i}} \,
 |\chi_\kk(m)|^2 e_m(z_1-z_2).
\end{align*}
By a change of variables we finally obtain
\begin{align*}
 \EE \delta_{\kk}\<n_black>_{-\infty,t_1}(z_1) \delta_k\<n_black>_{-\infty,t_2}(z_2) & \approx \\
 n! \sum_{m_1\in \mathcal{A}_{2^\kk}} \sum_{\substack{ m_i\in \ZZ^2 \\ i=2,\ldots, n }} &
 \prod_{i=1}^n \frac{\e^{-I_{m_i-m_{i-1}}|t_1-t_2|}}{2 I_{m_i-m_{i-1}}} \, e_{m_1}(z_1-z_2),
\end{align*}
with the convention that $m_0=0$. Let $K^\gamma(m) = \frac{1}{(1+|m|^2)^{1-\gamma}}$, for $\gamma\in[0,1)$, 
and write $K^\gamma \star^n K^\gamma$ to denote the $n$-th iterated convolution of $K^\gamma$ with itself 
(see Definition \ref{def:kernel_conv}). If we let $z_1=z_2=z$, for $t_1=t_2=t$ we get an estimate of the form
\begin{align*}
  \EE \delta_{\kk}\<n_black>_{-\infty,t}(z)^2 \lesssim 
  \sum_{m\in \A_{2^{\kk}}} K^0 \star^n K^0(m)
\end{align*}
while for $t_1\neq t_2$ and every $\gamma\in (0,1)$
\begin{equation*}
 \EE\left(\delta_\kk\<n_black>_{-\infty,t_1}(z)-\delta_\kk\<n_black>_{-\infty, t_2}(z)\right)^2 
 \lesssim |t_1-t_2|^{n \gamma} \sum_{m\in \A_{2^\kk}} K^\gamma \star^n K^\gamma (m). 
\end{equation*}
By Corollary \ref{cor:nested_kernel_conv} 
\begin{equation*}
 \EE \delta_{\kk}\<n_black>_{-\infty,t}(z)^2 \lesssim 
  \sum_{m\in \A_{2^{\kk}}} \frac{1}{(1+|m|^2)^{1-\ee}},
\end{equation*}
for every $\ee\in(0,1)$, and 
\begin{equation*}
 \EE\left(\delta_\kk\<n_black>_{-\infty,t_1}(z)-\delta_\kk\<n_black>_{-\infty, t_2}(z)\right)^2 
 \lesssim |t_1-t_2|^{n \gamma} \sum_{m\in \A_{2^\kk}} \frac{1}{(1+|m|^2)^{1-n\gamma}}.
\end{equation*}
Using the fact that $m\in\A_{2^\kk}$ we have that for every $\kk\geq -1$
\begin{align*}
 \EE \delta_{\kk}\<n_black>_{-\infty,t}(z)^2 & \lesssim 2^{2\lambda_1 \kk} \\
 \EE\left(\delta_\kk\<n_black>_{-\infty,t_1}(z)-\delta_\kk\<n_black>_{-\infty, t_2}(z)\right)^2 
 & \lesssim |t_1-t_2|^{n \gamma} 2^{2\lambda_2 \kk}
\end{align*}
for every $\lambda_1>0$ and every $\gamma\in(0,\frac{1}{n})$, $\lambda_2> n \gamma$, while for every $p\geq 2$ by Nelson's estimate \eqref{eq:Nelson's_Estimate} 
we finally get
\begin{align*}
 \EE \delta_{\kk}\<n_black>_{-\infty,t}(z)^p & \lesssim 2^{p\lambda_1 \kk} \\
 \EE\left(\delta_\kk\<n_black>_{-\infty,t_1}(z)-\delta_\kk\<n_black>_{-\infty, t_2}(z)\right)^p 
 & \lesssim |t_1-t_2|^{n \frac{p}{2} \gamma} 2^{p\lambda_2 \kk},
\end{align*}
The result then follows from \cite[Lemma 5.2, Lemma 5.3]{MWe15}, the usual Kolmogorov's criterion and the embedding 
$\BB^{-\al+\frac{2}{p}}_{p,p}\hookrightarrow \CC^{-\al}$, for $\al>\frac{2}{p}$.
\end{proof}
\stepcounter{section}
\section*{Appendix \thesection}\label{proof:diagram_convergence}
\setcounter{theorem}{0}
\setcounter{equation}{0}

\begin{proof}[Proof of Proposition \ref{prop:diagram_convergence}] For all $n\geq 1$, using the formula 
$$
\HH_n(X+Y,C) = \sum_{k=0}^n \binom{n}{k} X^k \HH_{n-k}(Y,C)
$$ 
we have
\begin{align*}
 \<n_black>^\varepsilon_{s,t} = \sum_{k=0}^n \binom{n}{k} (-1)^k \left(S(t-s)\<1_black>_{-\infty,s}^\ee\right)^k
 \HH_{n-k} \left(\<1_black>_{-\infty,t}^\varepsilon, \Re^\varepsilon\right).
\end{align*}
Thus it suffices to prove convergence only for $\<n_black>_{-\infty,t}^\ee$, $n\geq 1$.  

By \cite[Proposition 1.1.4]{Nu06} for $t_1,t_2>-\infty$ and $z_1,z_2\in \TT^2$
\begin{equation*}
 \EE \<n_black>^\ee_{-\infty,t_1}(z_1) \<n_black>^\ee_{-\infty,t_2}(z_2) = 
  n! \left(\EE\<1_black>_{-\infty,t_1}^\ee(z_1)\<1_black>_{-\infty,t_2}^\ee(z_2)\right)^n.
\end{equation*}
Using \eqref{eq:cov_formula} we get
\begin{equation*}
 \EE \<n_black>^\ee_{-\infty,t_1}(z_1) \<n_black>^\ee_{-\infty,t_2}(z_2) = n!
 \sum_{\substack{|m_i|\leq \frac{1}{\ee} \\ i=1,2,\ldots, n \\ m=m_1+\ldots+m_n}} 
 \prod_{i=1}^n \frac{\e^{- I_{m_i} |t_1-t_2|}}{2I_{m_i}} e_m(z_1-z_2),
\end{equation*}
and by a change of variables the above implies that for $\kk\geq -1$
\begin{align}
 \EE\delta_\kk\<n_black>^\ee_{-\infty,t_1}(z_1) \delta_\kk\<n_black>^\ee_{-\infty,t_2}(z_2) & \approx \label{eq:cov_1} \\
 n! \sum_{m_1\in \mathcal{A}_{2^\kk}} \sum_{\substack{ |m_i|\leq \frac{1}{\ee} \\ i=2,\ldots, n }} &
 \prod_{i=1}^n \frac{\e^{-I_{m_i-m_{i-1}}|t_1-t_2|}}{2 I_{m_i-m_{i-1}}} \, e_{m_1}(z_1-z_2). \nonumber
\end{align}
In a similar way
\begin{align}
 \EE\delta_\kk\<n_black>_{-\infty,t_1}(z_1) \delta_\kk\<n_black>^\ee_{-\infty,t_2}(z_2) & \approx \label{eq:cov_2}\\
 n! \sum_{m_1\in \mathcal{A}_{2^\kk}} \sum_{\substack{ |m_i|\leq \frac{1}{\ee} \\ i=2,\ldots, n }} &
 \prod_{i=1}^n \frac{\e^{-I_{m_i-m_{i-1}}|t_1-t_2|}}{2 I_{m_i-m_{i-1}}} \, e_{m_1}(z_1-z_2) \nonumber
\end{align}  
and for $K^\gamma(m) = \frac{1}{(1+|m|^2)^{1-\gamma}}$ combining \eqref{eq:cov_1} and \eqref{eq:cov_2} for $z_1=z_2=z$ and $t_1=t_2=t$ we have
that
\begin{align*}
 \EE\left(\delta_\kk\<n_black>_{-\infty,t}(z) -\delta_\kk\<n_black>_{-\infty,t}^\ee(z)\right)^2 \lesssim 
 \sum_{m\in\A_{2^\kk}} K^0 \star^n_{> \frac{1}{\ee}} K^0(m),
\end{align*}
while for $t_1\neq t_2$ and every $\gamma\in(0,1)$ 
\begin{align*}
 \EE\left[\left(\delta_\kk\<n_black>_{-\infty,t_1}(z) -\delta_\kk\<n_black>_{-\infty,t_1}^\ee(z)\right) 
 \left(\delta_\kk\<n_black>_{-\infty,t_2}(z) -\delta_\kk\<n_black>_{-\infty,t_2}^\ee(z)\right)\right] & \lesssim \\
 |t_1-t_2|^{n \gamma} \sum_{m\in\A_{2^\kk}} K^\gamma \star^n_{> \frac{1}{\ee}} K^\gamma(m).
\end{align*}
Proceeding as in the proof of Theorem \ref{thm:mod_sol} (see Appendix \hyperref[proof:mod_sol]{D}) and using Corollary 
\ref{cor:nested_kernel_conv} we obtain that
\begin{align*}
 \EE\left(\delta_\kk\<n_black>_{-\infty,t}(z) -\delta_\kk\<n_black>_{-\infty,t}^\ee(z)\right)^2 \lesssim 2^{2\lambda_1\kk}
 \frac{1}{\left(1+\frac{1}{\ee^2}\right)^{\lambda_1/2}}
\end{align*}
for every $\lambda_1\in(0,1)$, and
\begin{align*}
 \EE\Big[\left(\delta_\kk\<n_black>_{-\infty,t_1}(z) -\delta_\kk\<n_black>_{-\infty,t_1}^\ee(z)\right) &
 \left(\delta_\kk\<n_black>_{-\infty,t_2}(z) -\delta_\kk\<n_black>_{-\infty,t_2}^\ee(z)\right)\Big]  \lesssim \\ & |t_1-t_2|^{n \gamma} 
 2^{2\lambda_2\kk} \frac{1}{\left(1+\frac{1}{\ee^2}\right)^{\lambda_2 - n\gamma}}, 
\end{align*}
for every $\gamma\in(0,\frac{1}{n})$, $\lambda_2>n\gamma$. The result then follows by Nelson's estimate \eqref{eq:Nelson's_Estimate} 
combined with the usual Kolmogorov's criterion, as well as the embedding 
$\BB^{-\al+\frac{2}{p}}_{p,p}\hookrightarrow \CC^{-\al}$, for $\al>\frac{2}{p}$. 
\end{proof}
\end{appendices}

\bibliographystyle{keylabel}
\bibliography{sqe_bibliography}{}

\newcommand{\noopsort}[1]{}
\begin{thebibliography}{HaJV07}

\bibitem[AR91]{AR91}
S.~Albeverio and M.~R\"{o}ckner.
\newblock Stochastic differential equations in infinite dimensions: Solutions
  via {D}irichlet forms.
\newblock {\em Probab. Theory Rel. Fields}, 89(3):347--386, 1991.

\bibitem[BCD11]{BCD11}
H.~Bahouri, J.Y. Chemin, and R.~Danchin.
\newblock {\em Fourier Analysis and Nonlinear Partial Differential Equations}.
\newblock Springer, 2011.

\bibitem[ChF16]{ChF16}
K.~Chouk and P.K. Friz.
\newblock Support theorem for a singular spde: the case of g{PAM}.
\newblock Preprint \emph{arXiv: 1409.4250v3}, 2016.

\bibitem[dPD03]{dPD03}
G.~da~Prato and A.~Debbussche.
\newblock Strong solutions to the stochastic quantization equations.
\newblock {\em Ann. Probab.}, 32(4):1900--1916, 2003.

\bibitem[dPZ92]{dPZ92}
G.~da~Prato and J.~Zabczyk.
\newblock {\em Stochastic Equations in Infinite Dimensions}.
\newblock Cambridge University Press, 1992.

\bibitem[dPZ96]{dPZ96}
G.~da~Prato and J.~Zabczyk.
\newblock {\em Ergodicity for Infinite Dimensional Systems}.
\newblock Cambridge University Press, 1996.

\bibitem[Ev10]{Ev10}
L.C. Evans.
\newblock {\em Partial Differential Equations}.
\newblock AMS, 2010.

\bibitem[Ha14]{Ha14}
M.~Hairer.
\newblock A theory of regularity structures.
\newblock {\em Invent. Math.}, 198(2):269--504, 2014.

\bibitem[HSV07]{HSV07}
M.~Hairer and A.M.~Stuart ans J.~Voss.
\newblock Analysis of {SPDE}s arising in path sampling, {Part II}: The
  nonlinear case.
\newblock {\em Ann. Appl. Probab.}, 17(5/6):1657 -- 1706, 2007.

\bibitem[HM16]{HM16}
M.~Hairer and J.~Mattingly.
\newblock The strong feller property for singular stochastic {PDE}s.
\newblock To appear.

\bibitem[HMS11]{HMS11}
M.~Hairer, J.~Mattingly, and M.~Scheutzow.
\newblock Asympotic coupling and a weak form of harris' theorem with
  applications to stochastic delay equations.
\newblock {\em Probab. Theory Rel. Fields}, 149(1):1657--1706, 2011.

\bibitem[JLM85]{JLM85}
G.~Jona-Lassinio and P.K. Mitter.
\newblock On the stochastic quantization of field theory.
\newblock {\em Comm. Math. Phys.}, 101(3):409--436, 1985.

\bibitem[MWe14]{MWe14}
J.C. Mourrat and H.~Weber.
\newblock Convergence of the two-dimensional dynamic {Ising--Kac} model to
  $\phi_2^4$.
\newblock Preprint \emph{arXiv: 1410.1179}, 2014.

\bibitem[MWe15]{MWe15}
J.C. Mourrat and H.~Weber.
\newblock Global well-posedness of the dynamic $\phi^4$ in the plane.
\newblock Preprint \emph{arXiv: 1501.06191}, 2015.

\bibitem[No86]{No86}
J.R. Norris.
\newblock Simplified {M}alliavin calculus.
\newblock {\em S\'{e}minaire de Probabilit\'{e}s (Strasbourg)}, 20:101--130,
  1986.

\bibitem[NPV11]{NPV11}
E.D. Nezza, G.~Palatucci, and E.~Valdinoci.
\newblock Hitchhiker's guide to fractional sobolev spaces.
\newblock Preprint \emph{arXiv: 1104.4345V3}, 2011.

\bibitem[Nu06]{Nu06}
D.~Nualart.
\newblock {\em The Malliavin Calculus and Related Topics}.
\newblock Springer, 2006.

\bibitem[PW81]{PW81}
G.~Parisi and S.C. Wu.
\newblock Perturbration theory without gauge fixing.
\newblock {\em Sci. Sinica}, 24(4):483--496, 1981.

\bibitem[RY99]{RY99}
D.~Revuz and M.~Yor.
\newblock {\em Continuous Martingales and Brownian Motion}.
\newblock Springer, 1999.

\bibitem[RZZ15]{RZZ15}
M.~R\"{o}ckner, R.~Zhu, and X.~Zhu.
\newblock Restricted markov uniqueness for the stochastic quantization of
  $\mathcal{P}(\phi)$ and its applications.
\newblock Preprint \emph{arXiv: 1511.08030, 2015}, 2015.

\bibitem[RZZ16]{RZZ16}
M.~R\"{o}ckner, R.~Zhu, and X.~Zhu.
\newblock Ergodicity for the stochastic quantization problems on the
  2{D}-torus.
\newblock Preprint \emph{arXiv: 1606.02102, 2016}, 2016.

\bibitem[Ze95]{Ze95}
Eberhard Zeidler.
\newblock {\em Applied functional analysis}.
\newblock Springer-Verlag, New York, 1995.
\newblock Main principles and their applications.

\end{thebibliography}

\begin{flushleft}
\small \normalfont
\textsc{Pavlos Tsatsoulis\\
University of Warwick\\
Coventry, UK}\\
\texttt{\textbf{p.tsatsoulis@warwick.ac.uk}}

\smallskip

\small \normalfont
\textsc{Hendrik Weber\\
University of Warwick\\
Coventry, UK}\\
\texttt{\textbf{hendrik.weber@warwick.ac.uk}}
\end{flushleft}

\end{document}